\documentclass[11pt]{article}
 
\usepackage{amsmath,amssymb,amsfonts,tikz,mathabx,bbold}
\usepackage{textcomp,multicol,enumitem,mdwlist}
\usepackage{t-angles}
\usepackage{authblk}
\usepackage{chngcntr}

\newcommand{\ds}{\displaystyle}

\newcommand{\F}{\mathbb{F}}

\newcommand{\R}{\mathbb{R}}
\newcommand{\C}{\mathbb{C}}
\renewcommand{\H}{\mathbb{H}}
\renewcommand{\O}{\mathbb{O}}
\renewcommand{\P}{\mathbb{P}}
\renewcommand{\S}{\mathbb{S}}
\newcommand{\Zorn}{\mathrm{Zorn}}

\newcommand{\Id}{\mathrm{id}}
\newcommand{\Hom}{\mathrm{Hom}}
\newcommand{\Span}{\mathrm{Span}}

\newcommand{\Set}{\mathrm{\bf Set}}
\newcommand{\Mag}{\mathrm{\bf Mag}}

\newcommand{\Loop}{\mathrm{\bf Loop}}
\newcommand{\Grp}{\mathrm{\bf Grp}}
\newcommand{\catC}{\mathrm{\bf C}} 
\newcommand{\catA}{\mathrm{\bf A}}

\newcommand{\Com}{\mathrm{\bf Com}}
\newcommand{\As}{\mathrm{\bf As}}
\newcommand{\Alg}{\mathrm{\bf Alg}}
\newcommand{\Alt}{\mathrm{\bf Alt}}
\newcommand{\calP}{{\mathcal{P}}}

\newcommand{\bprod}{\bullet}

\newcommand{\sotimes}{\mbox{\small \,$\otimes$\,}}

\newcommand{\Diff}{{\mathrm{Diff}}}
\newcommand{\Inv}{{\mathrm{Inv}}}

\renewcommand{\over}{\slash}
\newcommand{\under}{\backslash}

\newcommand{\HI}{H_{\mathrm{I}}}
\newcommand{\HIex}{H_{\mathrm{I}}^{\scoprod}}
\newcommand{\HU}{H_{\mathrm{U}}}
\newcommand{\HUex}{H_{\mathrm{U}}^{\scoprod}}
\newcommand{\HUCD}{H_{\mathrm{UCD}}}
\newcommand{\HUCDex}{H_{\mathrm{UCD}}^{\scoprod}}
\newcommand{\Hi}{H_{\mathrm{inv}}}

\newcommand{\Hiex}{H_{\mathrm{inv}}^{\scoprod}}

\newcommand{\Hdex}{H_{\mathrm{dif}}^{\scoprod}}
\newcommand{\HFdB}{H_{\mathrm{FdB}}}
\newcommand{\HFdBnc}{H_{\mathrm{FdB}}^{\mathrm{nc}}}
\newcommand{\HFdBex}{H_{\mathrm{FdB}}^{\scoprod}}

\newcommand{\D}{\Delta}
\newcommand{\Di}{\Delta_{\mathrm{inv}}}

\newcommand{\Diex}{\Delta_{\mathrm{inv}}^{\scoprod}}

\newcommand{\DFdB}{\Delta_{\mathrm{FdB}}}
\newcommand{\DFdBnc}{\Delta_{\mathrm{FdB}}^{\mathrm{nc}}}
\newcommand{\DFdBex}{\Delta_{\mathrm{FdB}}^{\scoprod}}

\newcommand{\calC}{\,{\mathcal{C}}}
\newcommand{\calM}{\,{\mathcal{M}}}
\newcommand{\calN}{\,{\mathcal{N}}}
\newcommand{\calE}{\,{\mathcal{E}}}

\newcommand{\bfe}{\mathbf{e}} 

\newcommand{\bfk}{\mathbf{k}}
\newcommand{\bfm}{\mathbf{m}}
\newcommand{\bfn}{\mathbf{n}}
\newcommand{\bfp}{\mathbf{p}}
\newcommand{\bfq}{\mathbf{q}}

\numberwithin{equation}{subsection}

\newtheorem{theorem}[equation]{Theorem} 
\newtheorem{proposition}[equation]{Proposition}
\newtheorem{corollary}[equation]{Corollary}
\newtheorem{lemma}[equation]{Lemma}
\newtheorem{defin}[equation]{Definition}
\newenvironment{definition}{\begin{defin} \em}{\end{defin}\par\vspace{.2cm}}
\newtheorem{rem}[equation]{Remark}
\newenvironment{remark}{\begin{rem} \em}{\end{rem}\par\vspace{.2cm}}
\newtheorem{exa}[equation]{Example}
\newenvironment{example}{\begin{exa} \em}{\end{exa}\par\vspace{.2cm}}
\newtheorem{exas}[equation]{Examples}
\newenvironment{examples}{\begin{exas} \em}{\end{exas}\par\vspace{.2cm}}
\newenvironment{proof}{\bigskip \noindent{\bf Proof.\/}}
        {\hfill$\Box$\par\vspace{.2cm}}
\newenvironment{proof of}[2]{\bigskip \noindent{\bf Proof of #1~\ref{#2}.\/}}
        {\hfill$\square$\par\vspace{.2cm}}
\newenvironment{proof that}[1]{\bigskip \noindent{\bf Proof that #1.\/}}
        {\hfill$\square$\par\vspace{.2cm}}

\newcommand{\treeO}{\, 
\begin{tikzpicture}[scale=.3]
\draw (0,0) -- (0,1);
\end{tikzpicture}\, }

\newcommand{\treeA}{\, 
\begin{tikzpicture}[scale=.15]
\draw (1,0) -- (1,1) ;
\draw (1,1) -- (0,2) ;
\draw (1,1) -- (2,2) ;
\end{tikzpicture}\, }

\newcommand\treeAB{\, 
\begin{tikzpicture}[scale=.12]
\draw (2,0) -- (2,1) ;
\draw (2,1) -- (0,3) ;
\draw (2,1) -- (3,2) ;
\draw (1,2) -- (2,3) ;
\end{tikzpicture}\, }

\newcommand\treeBA{\, 
\begin{tikzpicture}[scale=.12]
\draw (1,0) -- (1,1) ;
\draw (1,1) -- (3,3) ;
\draw (1,1) -- (0,2) ;
\draw (2,2) -- (1,3) ;
\end{tikzpicture}\, }

\newcommand{\veegraft}[2]{\, 
\begin{tikzpicture}[scale=.2]
\draw (1,0) -- (1,1) ;
\draw (1,1) -- (0.2,1.8) ;
\draw (1,1) -- (1.8,1.8) ;
\draw (1,1.5) node[above left] {$#1$} ;
\draw (1,1.5) node[above right] {$#2$} ;
\end{tikzpicture}\, }
 
\newcommand{\lgraft}[2]
{\, 
\begin{tikzpicture}[scale=.2] 
\draw (2.1,.9) -- (1.4,1.6) ;
\draw (2.6,0) node {$#1$} ;
\draw (0.5,2.6) node {$#2$} ;
\end{tikzpicture}\, }

\newcommand{\rgraft}[2]
{\, 
\begin{tikzpicture}[scale=.2] 
\draw (1.5,.9) -- (2.2,1.6) ;
\draw (.4,0) node {$#1$} ;
\draw (3,2.6) node {$#2$} ;
\end{tikzpicture}\, }

\newcommand{\lvertexgraft}[1]
{\, 
\begin{tikzpicture}[scale=.2] 
\draw (2,-1) -- (2,.5) ;
\draw (2,.5) -- (1,1.5) ;
\draw (2,.5) -- (3,1.5) ;
\draw (0.6,3) node {$#1$} ;
\end{tikzpicture}\, }

\newcommand{\combLgraft}[3]
{\, 
\begin{tikzpicture}[scale=.2] 
\draw (5,-.5) -- (5,1) ;
\draw (5,1) -- (0,6) ;
\draw (5,1) -- (6,2) ;
\draw (5.8,3.3) node {.} ;
\draw (5,4.1) node {.} ;
\draw (4.1,4.9) node {.} ;
\draw (2,4) -- (3,5) ;
\draw (1,5) -- (2,6) ;
\draw (2.6,8) node {$#1$} ;
\draw (4,6.5) node {$#2$} ;
\draw (7.5,3.2) node {$#3$} ;
\end{tikzpicture}\, }

\begin{document}

\title{Loop of formal diffeomorphisms and Fa\`a di Bruno coloop bialgebra}

\author[,1]{Alessandra Frabetti\thanks{Corresponding author: 
\texttt{frabetti@math.univ-lyon1.fr}}}
\author[,2]{Ivan P. Shestakov\thanks{Email: \texttt{shestak@ime.usp.br}}}

\affil[1]{Univ Lyon, Université Claude Bernard Lyon 1, CNRS UMR 5208, 
Institut Camille Jordan, F-69622 Villeurbanne, France} 
\affil[2]{Instituto de Mat\'ematica e Estatistica, 
Universitade de S\~ao Paulo, 
Caixa Postal 66281, S\~ao Paulo, SP, 05315-970, Brazil}

\date{\today}
 
\maketitle


\begin{abstract}
We consider a generalization of (pro)algebraic loops defined on general 
categories of algebras and the dual notion of a coloop bialgebra suitable 
to represent them as functors. 
We prove that the natural loop of formal diffeomorphisms with associative
coefficients is proalgebraic, and we give the closed formulas of the
codivisions on its coloop bialgebra. 
This result provides a generalization of the Lagrange inversion formula to 
series with non-commutative coefficients, and a loop-theoretic explanation 
to the existence of the non-commutative Fa\`a di Bruno Hopf algebra. 
\end{abstract}

{\em MSC:\/} 20N05, 14L17, 18D35, 16T30
\bigskip

{\em Keywords:\/} Loops; Algebraic groups; Non-associative algebras; 
Cogroups; Coloops; Formal series; Formal diffeomorphisms.   
\bigskip 


\section{Introduction}

\subsection{Presentation and overview of the results}

An affine proalgebraic group $G$ is a representable functor in groups defined 
on the category of commutative associative algebras over a field $\F$.  
The algebra representing $G$ is the commutative Hopf algebra $\F[G]$ 
of regular functions. 
In this paper we consider two generalizations of proalgebraic groups, on one
side to representable functors {\em on categories of non-commutative algebras}, 
on the other side to functors taking values {\em in non-associative groups with 
divisions}, that is, {\em loops}. 

Our main motivation comes from two proalgebraic groups of formal series 
appearing in renormalization in quantum field theory: the group of invertible 
series with constant term equal to $1$, represented by the Hopf algebra of 
symmetric functions, and that of formal diffeomorphisms tangent to the
identity, represented by the Fa\`a di Bruno Hopf algebra.
Details on the role played by these series in quantum field theory are given
in a separate section below. 

Both types of series make sense with non-commutative coefficients, and both
representative Hopf algebras admit a non-commutative version~\cite{BFK}. 
We are interested in the relationship between the non-commutative algebras 
and the sets of series. For this, we first consider generalizations of 
proalgebraic groups to categories of non-commutative algebras.

Functors in groups on general categories have been studied by algebraic 
topologists in the late 50's.
D. Kan considered them on the category of groups~\cite{KanMonoids}, and 
B. Eckmann and P. Hilton~\cite{EckmannHiltonIII} 
introduced them on general categories. 
Their representative Hopf-type object is called a {\em cogroup}.
In a category with coproduct $\scoprod$ and initial object, a cogroup 
is an object $H$ endowed with a comultiplication, a counit and an antipode 
satisfying the usual properties of Hopf algebras, where the comultiplication 
takes values in $H\scoprod H$ instead of $H\sotimes H$ (which is not 
necessarily defined). 
Cogroups are then generalizations of commutative Hopf algebras which,
unlike quantum groups in the case of associative algebras, 
preserve the functorial properties and the adjoint constructions. 
They have proved to be very fruitful in homotopy theory, where they appear as
special $H$-spaces~\cite{KanHomotopy}, as shown by I. Berstein~\cite{Berstein}. 
A comprehensive study of cogroups in many varieties of algebras can be found 
in G. Bergman and A.~Hausknecht's book~\cite{BergmanHausknecht}. 

Not all proalgebraic groups admit an extention to non-commutative algebras.
For instance, while the group of invertible formal series naturally 
extends as a proalgebraic group to the category of associative algebras, 
the group of formal diffeomorphism does not. 
We show, on this example, that the extention of the functor is sometimes 
possible if we regard the original group as a {\em loop}. 

Loops are multiplicative sets with unit and with a left and a right division 
instead of two-sided inverses. 
They first appeared, with some extra properties, in the work of R.~Moufang 
\cite{Moufang} on alternative rings, that is, rings where the associator 
$(a,b,c) = (ab)c-a(bc)$ is skew-symmetric. For an excellent historical review
on loops, see \cite{Pflugfelder2000}. 
Associative loops are groups. 
Similarly to Lie groups, the tangent space of a smooth loop carries 
a particular algebraic structure called a Sabinin algebra 
\cite{SabininMikheev,MikheevSabinin}, which reduces to a Mal'cev algebra 
\cite{Malcev} for smooth Moufang loops. 
The notion of universal enveloping algebra has been extended to 
Sabinin algebras by I. Shestakov, U. U. Umirbaev~\cite{ShestakovUmirbaev} 
and J. Mostovoy, J. M. P\'erez-Izquierdo~\cite{MostovoyPerezIzquierdo}. 
\medskip 

In this paper we consider functors in loops on a general category $\catC$
with coproduct and initial object and call their representative objects 
{\em coloops in $\catC$\/}. 
We specialise $\catC$ to be a variety of algebras over a field $\F$ 
to have a reasonable notion of generalized (pro)algebraic loop. 
The first simple example is the extention of the functors of invertible
elements in a unital algebra and that of unitary elements in a unital
involutive algebra. 
As expected, the largest category on which these functors are representable as 
loops turn out to be respectively that of alternative and of alternative 
involutive algebras (Prop.~\ref{invertible-loop} and Prop.~\ref{unitary-loop}). 
We also show that the loop of unitary elements in the Cayley-Dickson extention 
of an involutive algebra is not representable on non-commutative algebras 
(Prop.~\ref{Cayley-Dickson-loop}), even if examples of such loops exist. 
Then we turn to loops of formal series with coefficients in a non-commutative 
algebra. 
First we consider the set of invertible series (with constant term equal to
$1$). The algebra of series with coefficients in an alternative algebra is
alternative. 
Surprisingly, in contrast to the previous results, we find that the set of 
invertible series is a proalgebraic loop on all algebras, not necessarily 
alternative (Thm.~\ref{thm-invertible series}). 
Finally, our main result concerns the natural loop of formal diffeomorphisms 
(tangent to the identity) with associative coefficients. 
We show that it is proalgebraic, and give the closed formulas of the
codivisions on its representative Fa\`a di Bruno coloop bialgebra 
(Def.~\ref{FaaDiBrunoColoopBialgebra} and Thm.~\ref{Diff(A)-proalgebraic-loop}). 
For this, we express the co-operations in terms of some recursive operators 
defined on any positively graded algebra (Thm.~\ref{FaaDiBruno-operators}), 
which extend the natural pre-Lie product of the Witt Lie algebra 
(cf.~\cite{BFM,FrabettiManchon}) but not as a multibrace product 
(cf.~\cite{LodayRonco}), and which turn out to be very rich in combinatorial 
properties. 
The coefficients appearing in the divisions show up sequences of integer
numbers typical of the Lagrange inversion formula (as Catalan numbers) and some 
new ones, that we call {\em (labeled) Lagrange coefficients} 
(Def.~\ref{definition-d} and \ref{definition-de}). 
This result is a generalization of the Lagrange inversion formula to 
series with non-commutative coefficients, and gives a loop-theoretic 
explanation to the existence of the non-commutative Fa\`a di Bruno Hopf 
algebra~\cite{BFK}. 


\subsection{Motivation: formal series in quantum field theory}

The main object of study in perturbative quantum field theory are the
correlation functions of the fields describing some elementary particles, 
from which one can compute the probability amplitude of any event involving
the particles.
These quantities are asymptotic series in the powers of a measurable parameter
$\lambda$, such as the electric charge, called the coupling constant. 
For instance, for a self-interacting field $\phi$ with coupling $\lambda$ and 
mass $m$, the $k$-point correlation function is a series 
$$
G^{(k)}(x_1,...,x_k) = \langle \phi(x_1)\cdots \phi(x_k)\rangle 
= \sum_{n=0}^\infty G^{(k)}_n(x_1,...,x_k;m,\hbar)\ \lambda^n 
$$ 
where the $n$th coefficient is a finite sum of amplitudes of suitable Feynman 
graphs with $k$ fixed external legs, which depend on the mass $m$ 
and on the Plank constant's $\hbar$, and $n$ is related to the number of 
internal vertices of the graph. 

The computation of the correlation functions gives rise to some divergent 
integrals, or equivalently to the ill-defined product of singular
distributions. 
Giving a meaning to such terms requires a renormalization procedure, 
which globally amounts to suitably multiplying and composing the correlation 
functions with some others series, called renormalization factors, 
obtained by assembling the counterterms needed to cure each divergence 
\cite{Dyson}, \cite{ItzyksonZuber}.
Given an ambient algebra $A$, typically $\C$ or the algebra 
$\C(\!(\varepsilon)\!)$ of Laurent series in a regularization parameter 
$\varepsilon$, in renomalization theory there appear two groups of formal 
series in the variable $\lambda$ and coefficients in $A$: 
\begin{itemize}
\item
the set 
$\ds\Inv(A)=\Big\{ a(\lambda)=\sum_{n\geq 0} a_n\ \lambda^n\ |\ 
a_0=1, a_n\in A\Big\}$ 
of invertible series, endowed with the pointwise multiplication 
$(ab)(\lambda)=a(\lambda)\ b(\lambda)$ and the unit $1(\lambda)=1$, 
which represent the Green's functions (up to an invertible factor) 
and the renormalization factors;   

\item
the set $\ds\Diff(A)=\Big\{ a(\lambda)=\sum_{n\geq 0} a_n\ 
\lambda^{n+1}\ |\ a_0=1, a_n\in A\Big\}$ 
of formal diffeomorphisms, endowed with the composition law 
$(a\circ b)(\lambda)=a\big( b(\lambda)\big)$ 
and the unit $e(\lambda)=\lambda$, which represent the bare coupling constants
(i.e. the coupling constants before the renormalization is performed). 
\end{itemize}
Dyson's renormalization formulas~\cite{Dyson} are modeled by the 
semi-direct product $\ds\Diff(A)\ltimes \Inv(A)$, endowed with the law 
$$
(a_1,b_1) \cdot (a_2,b_2) 
= \big(a_1\circ a_2, (b_1\circ a_2)\ b_2 \big), 
$$
where $a_1,a_2\in \Diff(A)$ and $b_1,b_2\in \Inv(A)$, which is well defined 
because formal diffeomorphisms act on invertible series from the right, 
by composition. 

These groups are proalgebraic on commutative algebras, so they are perfectly
described by their re*resentative Hopf algebra. Physically, this means 
that the overall renormalization procedure (except the scheme which says how 
to compute the counterterms) is independent of the chosen field theory,
whenever the latter leads to commutative amplitudes.  
The recent results on the Renormalization Hopf algebras, initiated by A.~Connes 
and D.~Kreimer~\cite{ConnesKreimerI,ConnesKreimerII}, show even a stronger 
result: co-operations dual to the multiplication and the composition of series 
exist even on Hopf algebras generated by Feynman graphs, which contain the
coordinate ring of the usual groups of power series. In other words, there 
exist proalgebraic groups of series expanded over Feynman graphs, or over
various types of trees, which project onto the groups $\Inv$ qnd $\Diff$
and which turn out to be extremely efficient in handling the combinatorial
content of renormalization 
procedures~\cite{BFqedren,BFqedtree,BFK,vanSuijlekom,Pinter}. 

The toy model $\phi^3$ theory used by Connes-Kreimer is a scalar field theory 
and leads to the commutative algebra $A=\C$ of amplitudes.  
However, interesting physical situations involve non-commutative algebras. 
In fact, Feynman amplitudes are complex numbers for single scalar fields, 
the coupling constants and the renormalization factors, but they are
$4\times 4$ complex matrices for the fermionic or bosonic fields, and may be
represented by higher order matrices for theories involving several interacting
fields. 
In this case, forcing the final counterterms to be scalar, as imposed by the 
fact that the renormalization factors act on the (scalar) Lagrangian,
prevents us from describing the renormalization in a functorial way,
as shown by the results in \cite{vanSuijlekom-multiplicative}, where the Hopf
algebra does not represent a functorial group on $A=M_4(\C)$.  
In order to preserve this functoriality, there is a need to understand Dyson's 
formulas for sets of series $\Inv(A)$ and $\Diff(A)$ also when $A$ is not a 
commutative algebra. This is the motivation for the present work. 
\bigskip

\noindent
{\bf Acknowledgments.} 
The authors warmly thank Jos\'e Mar\'{\i}a P\'erez-Izquierdo and Jacob Mostovoy 
for the interesting discussions and suggestions related to the topic of this 
paper, and Jiang Zeng for pointing out an alternative combinatorial argument.
The authors warmly thank the anonymous referee for his clever comments on
some key steps of our construction. 

This work was partially supported by the COFECUB Project MA 157-16, by the
Brazilian Ministry of Science and Technology grant CNPq 303916/2014-1 and by 
the LABEX MILYON (ANR-10-LABX-0070) of Universit\'é de Lyon, within the
program "Investissements d'Avenir" (ANR-11-IDEX-0007) operated by the
French National Research Agency (ANR).


{\small
\tableofcontents
}


\section{Loops and coloops}

\subsection{Loops and functors in loops}
\label{section-loops}

A {\bf loop} is a non-empty set $Q$ endowed with a {\bf multiplication} 
$Q\times Q\longrightarrow Q,\,(a,b)\longmapsto a\cdot b$, 
a (two-sided) {\bf unit} $1\in Q$, a {\bf left division} 
$\backslash:Q\times Q\longrightarrow Q$ and a {\bf right division} 
$\slash:Q\times Q\longrightarrow Q$ satisfying the {\bf cancellation properties}
\begin{align}
\label{left-cancellation}
a\cdot (a\backslash b) = b, \qquad & a\backslash (a\cdot b) = b, 
\\ 
\label{right-cancellation}
(a\slash b)\cdot b = a, \qquad & (a\cdot b)\slash b = a. 
\end{align}
Given two loops $Q$ and $Q'$, a {\bf homomorphism of loops} 
$f:Q\longrightarrow Q'$ is of course a map which preserves the 
multiplication, and therefore the unit and the divisions.   

The multiplication in a loop $Q$ is not necessarily associative, if it is 
associative then the loop is a group. 
Any element $a$ in a loop $Q$ has a {\bf right inverse} $1\slash a$ 
and a {\bf left inverse} $a\backslash 1$, which do not necessarily coincide 
and do not necessarily determine the divisions, 
in the sense that they do not satisfy the identities 
\begin{align}
\label{loop-division}
(a\backslash 1)\cdot b = a\backslash b \qquad\mbox{and}\qquad  
b\cdot (1\slash a) = b\slash a 
\end{align}
for any $a,b\in Q$, which hold in any group. 

Examples of loops which are not groups are known since a long time, see 
for instance~\cite{Bruck} or \cite{Pflugfelder}. 
For finite loops, the multiplication table is a Latin square, and the number
of non isomorphic loops is known up to order 11 (cf. \cite{A057771}).
For instance, the subset $\{\pm 1, \pm i, \pm j, \pm k\}$ of the hyperbolic
quaternions \cite{Macfarlane} (where $i^2=j^2=k^2=1$ and $ij=k=-ji$, $jk=i=-kj$,
$ki=j=-ik$) forms a finite loop.
Among infinite loops, two well known examples are the set of invertible
octonions and that of unitary octonions, which is homeomorphic to $\S^7$. 
\medskip 

Denote by $\Loop$ the category of loops and let $F:\Loop\longrightarrow \Set_*$ 
be the forgetful functor to the category of pointed sets. 
As for functors in groups, a functor $Q:\catC \longrightarrow \Loop$ on a
given category $\catC$ is said to be {\bf representable} if the composite
functor $F Q$ is representable, cf. \cite{MacLane}. 
This means that $Q$ is naturally isomorphic to a hom-set functor 
$A \mapsto \Hom_{\catC}(H,A)$ for a given object $H$ in $\catC$, and implies
that the loop operations on any loop $Q(A)$ are determined by dual
co-operations on $H$, by convolution. 
Following an established terminology on cogroups, the representative object 
$H$ can then be called a {\bf coloop in $\catC$}.
A reasonable notion of {\bf (pro)algebraic loop} is obtained for $\catC$ being
a variety of algebras over a field $\F$, its representative coloop then being
a sort of bialgebra. 

In this section we describe coloops in an axiomatic way. In the next sections 
we give some easy examples of algebraic and non-algebraic loops on 
associative and non-associative algebras, and then study extensively the loop 
of invertible series and that of formal diffeomorphisms.


\subsection{Coloops in general categories}
\label{subsection-coloops}

Given a category $\catC$, Yoneda Lemma says that the category of representable
functors from $\catC$ to $\Set$, with natural transformations,
is equivalent to $\catC$. The equivalence is realized by the contravariant
Yoneda functor $Y$ from $\catC$ to the functor category $\Set^{\catC}$, 
defined on any object $H$ in $\catC$ by the functor $Y(H)=\Hom_{\catC}(H,\ )$, 
and on any map $\phi:H'\rightarrow H$ by the natural transformation 
$Y(\phi):Y(H)\rightarrow Y(H'):\alpha \mapsto \alpha\, \phi$ 
(cf.~\cite{MacLane} for details). 
In this section we characterise the subcategory of $\catC$ equivalent to 
representable functors from $\catC$ to $\Loop$. 
\medskip

The cartesian product of two functors $Y(H_1)$ and $Y(H_2)$ is known to be 
represented by the categorical coproduct $H_1\scoprod H_2$, i.e. 
$Y(H_1)\times Y(H_2)=Y(H_1\scoprod H_2)$, and the constant functor to the base 
point is known to be represented by an initial object $I$, i.e. it is 
of the form $Y(I)$. We recall the categorical notations about the coproduct, 
the initial object and some related categorical maps 
we need to define coloops. 
\bigskip 

The coproduct in a category $\catC$ is a bifunctor $\scoprod$ defined on
two objects $A$ and $B$ as the unique object $A\scoprod B$ together with two
maps $i_1:A\rightarrow A\scoprod B$ and $i_2:B\rightarrow A\scoprod B$
satisfying the following universal property:
for any maps $f:A\rightarrow C$ and $g:B\rightarrow C$, there exists a unique
map $\langle f,g \rangle: A\scoprod B \rightarrow C$ such that
$\langle f,g \rangle\, i_1 = f$ and $\langle f,g \rangle\, i_2 = g$. 
On two maps $f:A\rightarrow A'$ and $g:B\rightarrow B'$, the bifunctor is
defined as the map
$f\scoprod g = \langle i_1'f,i_2'g \rangle:A\scoprod B\rightarrow A'\scoprod B'$.
The coproduct can be extended to several objects and maps with similar
universal constructions, and turns out to be an associative bifunctor,
in the sense that
$(A\scoprod B)\scoprod C = A\scoprod (B\scoprod C)= A\scoprod B\scoprod C$
for any three objects and
$(f\scoprod g)\scoprod h = f\scoprod (g\scoprod h)= f\scoprod g\scoprod h$
for any three maps in $\catC$. 

An initial object in $\catC$ is an object $I$ together with a unique map
$u_A:I\rightarrow A$ on any object, which commutes with any map
$f:A\rightarrow B$, that is, $f\,u_A = u_B$. 
Then, there are canonical isomorphisms $A\scoprod I \cong A \cong I\scoprod A$
given by the universal maps
\begin{align*}
& \varphi_1:=i_1:A \longrightarrow A \scoprod I
\qquad\mbox{with inverse}\qquad
\psi_1=\langle \Id_A,u_A \rangle :A \scoprod I \longrightarrow A, 
\\
& \varphi_2:=i_2:A \longrightarrow I \scoprod A
\qquad\mbox{with inverse}\qquad
\psi_2=\langle u_A,\Id_A \rangle :I \scoprod A \longrightarrow A. 
\end{align*}
In particular, we have $I\scoprod I\cong I$, $\langle u_A,u_A \rangle = u_A$ 
and therefore also $u_{A\scoprod B} = u_A\scoprod u_B$.

For any objects $A$ and $B$, there is a canonical symmetry operator 
$\tau_{A,B}= \langle i_2,i_1 \rangle:A\scoprod B\rightarrow B\scoprod A$ 
such that $\tau_{A,B}^{-1}=\tau_{B,A}$. Note that $B\scoprod A=A\scoprod B$ 
as objects in $\catC$, but the maps $i_1$ and $i_2$ are inverted. 
The twist $\tau$ is precisely the map which identifies $A\scoprod B$ and 
$B\scoprod A$ as universal objects.
To sum up, 
{\em $(\catC,\scoprod,I,\tau)$ is a strict symmetric monoidal category\/}. 

Furthermore, for any $A$, there exists a canonical folding map 
$\mu_A=\langle \Id_A,\Id_A \rangle:A\scoprod A \rightarrow A$ such that, 
for any maps $f,g:A\rightarrow B$, we have 
$$
\langle f,g\rangle = \mu_B\, (f\scoprod g).  
$$
It follows that $\mu$ preserves the unit, i.e.
$\mu_A (u_A\scoprod u_A) = \langle u_A,u_A \rangle =u_A$, 
that it is associative, i.e. 
$\mu_A(\mu_A\scoprod\Id_A) = \mu_A(\Id_A\scoprod\mu_A)$, 
and that it is commutative, i.e. $\mu_A\,\tau_{A,A}=\mu_A$. 
It also follows that $\mu$ commutes with any map $f:A\rightarrow B$ in $\catC$, 
i.e. $\mu_B\, (f\scoprod f) = f\,\mu_A$. 
To sum up, we can say that {\em any object $(A,\mu_A,u_A)$ is a commutative 
monoid in $\catC$\/}, with respect to the monoidal product $\scoprod$, 
and that {\em any map $f:A\rightarrow B$ in $\catC$ is a morphism of monoids\/}. 
Finally, one can prove that the folding map on $A\scoprod B$ is 
given by $\mu_{A\scoprod B} = 
(\mu_A \scoprod \mu_B)\, (\Id_A\scoprod \tau_{B,A} \scoprod \Id_B)$. 

\begin{definition}
\label{def-coloop}
Let us call {\bf coloop in $\catC$} an object $H$ endowed with the following 
maps in $\catC$: 
\begin{enumerate}[label=\roman*),leftmargin=*]
\item
a {\bf comultiplication} $\D:H\longrightarrow H\scoprod H$; 

\item
a {\bf counit} $\varepsilon:H\longrightarrow I$ satisfying the {\bf counitary} 
property
\begin{align}
\label{counitary}
(\varepsilon\scoprod \Id)\,\Delta = \varphi_2
\qquad\mbox{and}\qquad 
(\Id\scoprod \varepsilon)\,\Delta = \varphi_1 ,  
\end{align}
where $\varphi_1:H\rightarrow H\scoprod I$ and 
$\varphi_2:H\rightarrow I\scoprod H$ are the canonical isomorphisms;  

\item 
a {\bf right codivision} $\delta_r:H\longrightarrow H\scoprod H$ 
satisfying the {\bf right cocancellation} properties
\begin{align}
\label{right-cocancellation}
(\Id\scoprod\mu)\, (\delta_r\scoprod\Id)\, \D = i_1 
\qquad\mbox{and}\qquad 
(\Id\scoprod\mu)\, (\Delta\scoprod\Id)\, \delta_r = i_1 ,  
\end{align}
where $i_1:H\rightarrow H\scoprod H$ can be factorized as 
$i_1=(\Id\scoprod u)\,\varphi_1$; 

\item
a {\bf left codivision} $\delta_l:H\longrightarrow H\scoprod H$ 
satisfying the {\bf left cocancellation} properties  
\begin{align}
\label{left-cocancellation}
(\mu\scoprod \Id)\, (\Id\scoprod\delta_l)\, \D = i_2 
\qquad\mbox{and}\qquad 
(\mu\scoprod \Id)\, (\Id\scoprod\Delta)\, \delta_l = i_2 ,   
\end{align}
where $i_2:H\rightarrow H\scoprod H$ can be factorized as 
$i_2=(u\scoprod \Id)\,\varphi_2$. 
\end{enumerate}
If $H$ and $H'$ are two coloops in $\catC$, we say that a map 
$f:H\longrightarrow H'$ is a {\bf homomorphism of coloops} if 
it commutes with the coproducts, the counits and the codivisions. 
\end{definition}

\begin{proposition}
Let $H$ be a coloop in $\catC$. 
\begin{enumerate}
\item
The codivisions verify the identities 
\begin{align}
\label{multiplication-codivision}
\mu\, \delta_r = u\, \varepsilon 
\qquad\mbox{and}\qquad 
\mu\, \delta_l = u\, \varepsilon , 
\end{align}
and the following {\bf partial counitality properties} 
\begin{align}
\label{counit-codivision}
(\Id\scoprod \varepsilon)\, \delta_r = \varphi_1 
\qquad\mbox{and}\qquad
(\varepsilon\scoprod\Id)\, \delta_l = \varphi_2 .  
\end{align}

\item
We can define a {\bf right antipode} $S_r:H\longrightarrow H$
and a {\bf left antipode} $S_l:H\longrightarrow H$ by setting 
\begin{align}
\label{antipode-left-right}
S_r := \psi_2\, (\varepsilon \scoprod \Id)\, \delta_r 
\qquad\mbox{and}\qquad 
S_l := \psi_1\, (\Id \scoprod \varepsilon)\, \delta_l \ ,
\end{align}
where $\psi_1=\langle \Id,u_H\rangle : H\scoprod I \rightarrow H$ and 
$\psi_2=\langle u_H,\Id\rangle : I\scoprod H \rightarrow H$ are isomorphisms. 
The antipodes satisfy the following {\bf left} and {\bf right 5-terms 
identities} 
\begin{align}
\label{antipode-left-right-5terms}
\mu\, (S_r \scoprod \Id)\, \D = u\, \varepsilon
\qquad\mbox{and}\qquad
\mu\, (\Id \scoprod S_l)\, \D = u\, \varepsilon. 
\end{align}
\end{enumerate}
\end{proposition}

These properties are easily verified. A proof using tangle diagrams is 
given in the Appendix. 

\begin{theorem}
\label{theorem-Q-functor}
Let $\catC$ be a category with coproduct and initial object. 
Then the Yoneda functor is a contravariant equivalence of categories 
from the category of coloops in $\catC$ to that of covariant representable 
functors $Q:\catC\longrightarrow \Loop$. 
\end{theorem}

\begin{proof}
We follow the ideas of Eckmann-Hilton~\cite{EckmannHiltonIII}, who 
characterized the subcategories of $\catC$ equivalent to the category 
of representable functors from $\catC$ respectively to the category $\Mag$ 
of unital multiplicative sets, called unital magmas in~\cite{Serre,Bourbaki}, 
and to the category $\Grp$ of groups\footnote{
  Eckmann-Hilton require $\catC$ to have zero-maps, we replace them with an
  initial object.}. 

Let us first prove that the Yoneda functor, applied to coloops in $\catC$, 
gives rise to a functor in loops. 
On a given coloop $H$, let us call $Q=Y(H)$. We define the multiplication 
and the divisions on each set $Q(A)=\Hom_{\catC}(H,A)$ as usual convolution 
with the coproduct and the codivisions in $H$, namely
\begin{align}
\nonumber
\alpha\cdot \beta &= \langle \alpha,\beta\rangle\, \Delta
= \mu_A\, (\alpha\scoprod \beta)\, \Delta
\\ 
\label{convolution}
\alpha\slash \beta &= \langle \alpha,\beta\rangle\, \delta_r
= \mu_A\, (\alpha\scoprod \beta)\, \delta_r
\\ 
\nonumber
\alpha\backslash \beta &= \langle \alpha,\beta\rangle\, \delta_l
= \mu_A\, (\alpha\scoprod \beta)\, \delta_l , 
\end{align}
for any $\alpha,\beta\in Q(A)$. The unit in $Q(A)$ is given, as usual,
by the map $1_A=u_A\,\varepsilon$, and the left and right inverses of
$\alpha$ are then easily described as
$\alpha\backslash 1 = \alpha\, S_l$ and $1\slash \alpha = \alpha\, S_r$. 
Then, using the cocancellation identities (\ref{right-cocancellation}) and 
(\ref{left-cocancellation}), and because $\mu_A$ is associative and commutes 
with $\catC$-maps, it is easy to verify that the divisions given by 
(\ref{convolution}) satisfy the cancellation properties 
(\ref{right-cancellation}) and (\ref{left-cancellation}). 

Now fix a homomorphism of coloops $\phi: H'\longrightarrow H$, and call 
$Q=Y(H)$, $Q'=Y(H')$ and $\Phi=Y(\phi)$. Yoneda Lemma tells us already that 
$\Phi$ is given on an object $A$ by the map 
\begin{align*}
\Phi_A:\ & Q(A)\longrightarrow Q'(A) \\
& \alpha \longmapsto \Phi_A(\alpha) = \alpha\, \phi,  
\end{align*}
and that, for any $f:A\rightarrow B$, $\Phi$ acts on the map 
$Q(f):Q(A)\rightarrow Q(B)$ given by $Q(f)(\alpha)= f\,\alpha$ as a natural 
transformation, i.e. 
\begin{align*}
\Phi_B\big(Q(f)(\alpha)\big) 
= f\, \alpha\, \phi 
= Q'(f)\big(\Phi_A(\alpha)\big).   
\end{align*}
It is then easy to verify that $\Phi_A$ is a homomorphism of loops, that is, 
for any $\alpha,\beta\in Q(A)$, we have
\begin{align*}
\Phi_A(\alpha\cdot \beta) & = \Phi_A(\alpha)\cdot \Phi_A(\beta), 
\end{align*}
and similarly for the other co-operations. 

Viceversa, let us describe how a functor in loops $Q$ gives rise to a coloop
structure on its representative object $H$. 
Suppose that the covariant functor $Q$ is represented by an object $H$, 
i.e. $Q=Y(H)$, that the set $Q(A)$ is a loop for any $A$ in $\catC$, 
and that for any map $f:A\rightarrow B$ the induced map 
$Q(f):Q(A)\rightarrow Q(B)$ given by $\alpha\mapsto Q(f)(\alpha)=f \alpha$ 
is a loop homomorphism. 
We use repeatedly the fact that, given $\alpha,\beta\in Q(A)$, for the
composite maps $f (\alpha\cdot \beta), f \alpha, f \beta\in Q(B)$ we have
\begin{align}
\label{propQ(f)}
f\,(\alpha\cdot \beta) &= Q(f)(\alpha\cdot \beta)
= Q(f)(\alpha)\cdot Q(f)(\beta)
= (f \alpha)\cdot (f \beta)
\end{align}
and similarly for the operations $\slash$ and $\backslash$.
Seeing $i_1,i_2:H\rightarrow H\scoprod H$ as elements of the loop
$Q(H\scoprod H)$, we define the comultiplication and the codivisions on $H$ by 
\begin{align*}
\Delta = i_1\cdot i_2, \qquad
\delta_r = i_1 \slash i_2, \qquad 
\delta_l = i_1 \backslash i_2
\end{align*}
and the counit $\varepsilon$ as the unit $1_I$ in $Q(I)$.
It follows that the antipodes are the inverses of the identity map, 
$S_r = 1_H \slash \Id_H$ and $S_l = \Id_H\backslash 1_H$. 

Let us show that these maps give a coloop structure to $H$, and that the
functor $Q\mapsto H$ is inverse to the Yoneda one, $H\mapsto Q=Y(H)$. 
For any $\alpha,\beta \in Q(A)$, we apply (\ref{propQ(f)}) to $A=H\scoprod H$, 
$B=A$, $f=\langle \alpha,\beta \rangle:H\scoprod H\rightarrow A$ and
to the elements $\alpha=i_1$, $\beta=i_2$ of $Q(H\scoprod H)$, and get
\begin{align}
\nonumber 
\langle \alpha,\beta \rangle\, \Delta
& = \langle \alpha,\beta \rangle\, (i_1\cdot i_2)
\\ 
\label{comult-mult}
& = (\langle \alpha,\beta \rangle\, i_1) \cdot
(\langle \alpha,\beta \rangle\, i_2)
= \alpha\cdot \beta
\end{align}
and similarly for the operations $\slash$ and $\backslash$. 
Now apply $Q$ to a unit map $u_A:I\rightarrow A$.
Since $Q(u_A):Q(I)\rightarrow Q(A)$ is a homomorphism of loops, it preserves
the units, and therefore, for $\varepsilon=1_I\in Q(I)$, we have
$$
Q(u_A)(\varepsilon) = u_A\, \varepsilon = 1_A. 
$$
In particular we have $u_H\,\varepsilon = 1_H$, and therefore, using
(\ref{comult-mult}), we have 
\begin{align*}
\langle u_H\varepsilon,\Id_H\rangle\, \Delta
& = 1_H\cdot \Id_H = \Id_H. 
\end{align*}
On the other side, we have 
\begin{align*}
\langle u_H\varepsilon,\Id_H\rangle\, \Delta
& = \langle u_H,\Id_H\rangle\, (\varepsilon\scoprod \Id)\, \Delta \\
& = \psi_2\, (\varepsilon\scoprod \Id)\, \Delta, 
\end{align*}
and we obtain the equality $\psi_2\, (\varepsilon\scoprod \Id)\, \Delta = \Id$. 
Since $\psi_2$ is the inverse map to $\varphi_2$, we obtain 
$(\varepsilon\scoprod \Id)\, \Delta = \varphi_2$, which proves (\ref{counitary}). 
Let us show equalities (\ref{right-cocancellation}).
Firstly, we have trivially that
$$
(\Id\scoprod u_H)\, i_1 = \langle i_1,i_2\, u_H\rangle\, i_1 = i_1. 
$$
Secondly, note that $\delta_r \cdot i_2 = (i_1\slash i_2)\cdot i_2 = i_1$,
therefore 
\begin{align*}
(\Id\scoprod \mu)\, (\delta_r\scoprod \Id)\, \Delta 
& = \langle i_1,i_2\mu \rangle\, \langle i_1 \delta_r,i_2 \rangle\,
(i_1\cdot i_2) \\
& = \langle i_1,i_2\mu \rangle\, (\delta_r \cdot i_2) \\
& = \langle i_1,i_2\mu \rangle\, i_1 = i_1. 
\end{align*}
Finally, since
$\langle \alpha,\beta \rangle (i_1\slash i_2) = \alpha \slash \beta$ 
and $(i_1\cdot i_2)\slash i_2 = i_1$, we also have 
\begin{align*}
(\Id\scoprod \mu)\, (\Delta\scoprod \Id)\, \delta_r 
& = \langle i_1,i_2\mu \rangle\, \langle i_1 \Delta,i_2 \rangle\,
(i_1\slash i_2) \\
& = \langle i_1,i_2\mu \rangle\, \big((i_1\cdot i_2) \slash i_2\big) \\
& = \langle i_1,i_2\mu \rangle\, i_1 = i_1. 
\end{align*}
The same arguments apply to the left codivision. 
\end{proof}
\medskip 


The relationship between coloops and cogroups is straightforward. 
As usual, a coloop $H$ is {\bf coassociative} if 
\begin{align}
\label{coassociative}
(\D\scoprod \Id)\,\D & = (\Id\scoprod\D)\,\D, 
\end{align}
and $H$ is {\bf cocommutative} if $\tau\, \D=\D$. 

We say that $H$ has the {\bf left} and {\bf right coinverse property} 
if the codivisions are determined by the antipodes, that is, 
\begin{align}
\label{coinverse}
\delta_l = (S_l\scoprod \Id)\,\D. 
\qquad\mbox{and}\qquad 
\delta_r = (\Id\scoprod S_r)\,\D 
\end{align}
These identities correspond to the analogues (\ref{loop-division}) in the 
loop $Q=Y(H)$. 

Furthermore, an {\bf antipode} on $H$ is a map $S:H\rightarrow H$ satisfying 
the {\bf 5-terms identity} 
\begin{align}
\label{antipode-5-terms}
\mu\,  (S \scoprod \Id) \,  \D 
= \mu\,  (\Id \scoprod S) \,  \D 
= u\, \varepsilon .  
\end{align}
This can happen if and only if $S_r=S_l=:S$. Note that the unicity of the
antipode satisfying (\ref{antipode-5-terms}) does not imply that the coinverse
properties (\ref{coinverse}) are verified. A counterexample is given by the
coloop of formal diffeomorphisms, cf. Section \ref{section-diffeomorphisms}. 

A cogroup in a category $\catC$ is an object $H$ endowed with a coassociative 
comultiplication $\Delta$, a counit $\varepsilon$ satisfying the counitary 
property (\ref{counitary}) and an antipode $S$ satisfying the 5-terms identity 
(\ref{antipode-5-terms}), cf.~\cite{Berstein}.
It follows from the dual statement on loops and groups (cf. \cite{Bruck}), that

\begin{proposition}
\label{coloop-cogroup}
If $H$ is a coassociative coloop, then it is a cogroup. 
\end{proposition}

In the Appendix we prove with tangles that coassociativity implies 
that the left and the right antipodes coincide, and therefore $H$ has an 
antipode satisfying the coinverse property.


\subsection{(Pro)algebraic loops}

Let $\catA$ be a variety of unital algebras over a field $\F$, that is, the 
subcategory of vector spaces over $\F$ which collects all algebras of a certain 
{\em type}, given by a set of operations of various arities, including the unit 
of arity $0$, defined by a set of identities (cf.~\cite{MacLane} ch.~V). 
For instance, $\catA$ can be the category of $\calP$-algebras, where $\calP$ 
is an algebraic operad with $\calP(0)=\{\,\mbox{unit map}\,\}$
(cf.~\cite{LodayVallette}). 

Then, $\catA$ has a coproduct and an initial object 
(cf.~\cite{MacLane} ch.~IX), therefore we can apply to $\catA$ the results 
of the previous section. 
More precisely, the initial object is given by the trivial unital algebra $\F$. 
Suppose that in $\catA$ there are operations $p$ of arity $n>0$, and let 
$\bar{A}$ denote the subalgebra of an algebra $A$ determined by such
operations. Then, given two algebras $A$ and $B$ in $\catA$, the coproduct
$A\scoprod B$ is the quotient of the free algebra
$\catA(\bar{A}\oplus \bar{B})$ (which always exists, cf.~\cite{MacLane} ch.~V)
by the ideal generated by the identities
\begin{align}
\nonumber
& p_{\catA(A\oplus B)}(a_1,...,a_n) = p_{A}(a_1,...,a_n) \in \bar{A}, 
\qquad \mbox{for any $a_1,...,a_n\in \bar{A}$} 
\\ 
\label{free A-product relations}
& p_{\catA(A\oplus B)}(b_1,...,b_n) = p_{B}(b_1,...,b_n) \in \bar{B}, 
\qquad \mbox{for any $b_1,...,b_n\in \bar{B}$} 
\\
\nonumber
& 1_A = 1_B = 1_{\catA(A\oplus B)},  
\end{align}
for all the operations $p$ admitted in $\catA$. 
The universal properties of $\scoprod$ follow from the universal properties 
of the free algebra $\catA(\bar{A}\oplus \bar{B})$. 

\begin{examples}
\begin{enumerate}
\item 
In the category $\Com_\F$ of unital commutative and associative algebras over 
$\F$, the free algebra $\Com_\F(V)$ on a vector space $V$ is the 
symmetric algebra $S(V)$, and the coproduct of two algebras $A$ and $B$ 
is the tensor product $A\sotimes B$. 

\item 
In the category $\As_\F$ of unital associative algebras over $\F$, 
the free algebra $\As_\F(V)$ is the tensor algebra $T(V)$, and the 
coproduct\footnote{
In associative algebras, the coproduct is usually called {\em free product} 
and denoted by $\star$.
} 
$A\scoprod B$ of two algebras $A$ and $B$ is the tensor algebra 
$T(\bar{A}\oplus \bar{B})$ modulo relations (\ref{free A-product relations}), 
which mean that $a\sotimes a' = a\, a$ whenever $a$ and $a'$ are both in 
$\bar{A}$ or both in $\bar{B}$. As a vector space, we then have
$$
A\scoprod B = \F \oplus \bigoplus_{n\geq 1}\ 
\Big[\underset{n}{
\underbrace{\bar{A}\otimes \bar{B}\otimes \bar{A} \otimes\cdots}} \,\oplus\, 
\underset{n}{
\underbrace{\bar{B}\otimes \bar{A}\otimes \bar{B}\otimes\cdots}}\Big],   
$$
and the multiplication in $A\scoprod B$ is given by the 
concatenation modulo the above relations. 
For instance, if we denote the multiplication in $A\scoprod B$ by $\bprod$, 
we have
\begin{align*}
& (b\sotimes a)\bprod (b'\sotimes a'\sotimes b'') 
= b\sotimes a\sotimes b'\sotimes a'\sotimes b'' \\ 
& (b\sotimes a)\bprod (a'\sotimes b'\sotimes a'')
= b\sotimes (a\,a')\sotimes b'\sotimes a''.    
\end{align*}

\item 
Let $\Alg_\F$ be the category of unital algebras (not necessarily associative, 
also called magmatic) over $\F$. The free unital algebra on a vector space $V$ 
is the tensor algebra with parenthesizing $T\{V\}$, and the coproduct 
$A\scoprod B$ of two algebras $A$ and $B$ is the quotient of 
$T\{\bar{A}\oplus \bar{B}\}$ modulo relations 
(\ref{free A-product relations}), which again mean that $a\sotimes a' = a\, a$ 
whenever $a$ and $a'$ are both in $\bar{A}$ or both in $\bar{B}$. 
The multiplication in $A\scoprod B$ is the concatenation with parenthesis. 

\item 
An {\bf alternative algebra} is an algebra $A$ such that the associator 
$(a,b,c) = (ab)c-a(bc)$ is skew-symmetric, that is 
\begin{align}
\label{alternative}
(b,a,c) = -(a,b,c) \qquad\mbox{and}\qquad (a,c,b) = -(a,b,c)
\end{align}
for any $a,b,c\in A$. This is equivalent to requiring that 
$$
(ab)b = a(bb) \qquad\mbox{and}\qquad (aa)b = a(ab)
$$
for any $a,b\in A$. 
Unital alternative algebras over a field $\F$ form a subcategory of $\Alg_\F$, 
denoted by $\Alt_\F$, which is a variety with initial object $\F$. 
The coproduct $A\scoprod B$ of two unital alternative algebras is the quotient 
of the coproduct in $\Alg_\F$ by the relations (\ref{alternative}). 
For details see~\cite{ZSSS} or \cite{KuzminShestakov}. 

\item 
In a category $\catA$, an (anti) {\bf involution} is a unary linear operation 
${}^*:A\longrightarrow A$ such that 
\begin{align}
\label{involution}
(a^*)^* = a
\qquad\mbox{and}\qquad 
(a_1\,a_2)^* = a_2^*\,a_1^*, 
\end{align}
for any $a,a_1,a_2\in A$. Each of the four previous categories of algebras 
can be considered with involution, and denote by $\catA^*$. For such algebras, 
the initial object and the coproduct are the same as in $\catA$, the involution 
on $A\scoprod B$ is automatically defined from the involutions on $A$ and $B$ 
by properties (\ref{involution}). 
Note that, in $\Alg_\F$ and in $\Alt_\F$, the parenthesizing of a word 
$a_1\sotimes\cdots\sotimes a_n$ is inverted from left to right by the involution, 
together with the single letters of the word. 
\end{enumerate}
\end{examples}

\begin{definition}
\label{def-proalgebraic-loop}
A coloop $H$ in a variety of unital algebras $\catA$ is called a 
{\bf coloop $\catA$-bialgebra}. 
Its associated functor in loops $Q=Y(H)$ is then called an 
{\bf algebraic loop on $\catA$} if $H$ is a finitely generated algebra, 
and a {\bf proalgebraic loop on $\catA$} if $H$ is not finitely generated. 
In this case, it is an inductive limit of finitely generated coloop 
$\catA$-bialgebras. 
\end{definition}

There are not many known examples of algebraic loops on non-commutative
algebras, but some of them are quite special.    
In section \ref{section-unitary}, we give two easy examples of algebraic groups 
on commutative algebras which can be extended as groups to associative algebras 
(the groups of invertible elements and that of unitary ones), 
and another one which can not be extended to a functor on associative algebras
even as a loop (the Cayley-Dickson loop). 
Viceversa, in section \ref{section-diffeomorphisms} we give the example of 
a proalgebraic group which can be extended to associative algebras only as a 
proalgebraic loop (the loop of formal diffeomorphisms).

Finally, the two groups of invertible and unitary elements can be extended to
alternative algebras if we regard them as algebraic loops, and in section
\ref{section-invertible series} we also give an example of an algebraic group
which can be extended to all associative algebras as a group,
and to non-associative algebras as a loop (the loop of invertible series).

While the functoriality of the first examples is straightforward, for the
two loops of formal series it is not. The group of formal diffeomorphisms
is a local approximation of the most simple group of smooth diffeomorphisms
on a manifold, and the existence of a proalgebraic version on associative
algebras is a new step in the study of non-commutative geometry.
In particular, its existence as a proalgebraic loop allows us to consider
a physical ``renormalization loop'' to replace the standard group
\cite{ItzyksonZuber}, which could be applied in any perturbative theory
when the function rings have to be replaced by tensor algebras, as in
\cite{Herscovich}. 

\begin{remark}
All known examples of algebraic groups and loops on non-commutative algebras
have free underlying algebra structure. The fact that this should hold in any
category (under certain completeness hypothesis) has not been proved,
but it was proved for cogroups in several categories:
by D. Kan~\cite{KanMonoids} in the category of groups, 
by I. Berstein~\cite{Berstein} (and later reproved by J. Zhang in~\cite{Zhang}) 
in the category of graded connected associative algebras, 
and by B. Fresse~\cite{Fresse} in the category of complete algebras over any 
operad. 
For coloops, this result is proved by G. Bergman and A.O. Hausknecht 
\cite{BergmanHausknecht} in the category of graded connected associative 
rings.
\end{remark}

Before giving the examples, we mention two maps which allow us to compare 
coloop and cogroup bialgebras to usual Hopf algebras. 
A coloop $\catA$-bialgebra has the operations 
$p:H^{\otimes\, n}\longrightarrow H$ from $\catA$, 
and the categorical folding map $\mu:H\scoprod H\rightarrow H$ needed 
to describe the coloop axioms, which can be iterated on $n$ copies of $H$. 
In general, there is no relationship between these two types of operations,
since $H^{\otimes\, n}$ need not be an algebra in $\catA$. 
\medskip 

Assume that $\catA$ is a category of algebras such that, for any 
$\catA$-algebras $A$ and $B$, the tensor product $A\sotimes B$ is again 
an $\catA$-algebra with componentwise operations 
$$
p^{(n)}_{A\otimes B}(a_1\otimes b_1,\cdots,a_n\otimes b_n) = 
p^{(n)}_{A}(a_1,\cdots,a_n)\otimes p^{(n)}_{B}(b_1,\cdots,b_n)
$$
and unit $1_{A\otimes B}=1_A\sotimes 1_B$.

\begin{definition}
\label{canonical-projection}
For any $n\geq 2$ and for any $n$ algebras $A_k$, with $k=1,...,n$, 
we call {\bf canonical projection} of $A_1\scoprod\cdots\scoprod A_n$
onto $A_1\sotimes\cdots\sotimes A_n$ the algebra homomorphism
$$
\pi:= \langle j_1,...,j_n \rangle:
A_1\scoprod\cdots\scoprod A_n \longrightarrow A_1\sotimes\cdots\sotimes A_n
$$
induced by the injective algebra maps
$j_k:A_k\rightarrow A_1\sotimes \cdots \sotimes A_n$ given by 
$$
j_k(a_k)=1_{A_1}\sotimes \cdots \sotimes a_k\sotimes \cdots \sotimes 1_{A_n}. 
$$
The map $\pi$ reorders the elements of $A_1\scoprod\cdots\scoprod A_n$ and 
then multiplies them within each $A_k$ to get elements in 
$A_1\sotimes\cdots\sotimes A_n$. For instance, if we denote by $a^{(k)}$ an 
element $a\in A_k$ seen in the coproduct $A_1\scoprod\cdots\scoprod A_n$, we have 
$$
\pi\big(a^{(1)} b^{(2)} c^{(1)} d^{(2)}\big) = (ac)\sotimes (bd).
$$
This map is surjective, because a preimage of any 
$a_1\sotimes a_2\sotimes \cdots\sotimes a_n\in A_1\sotimes\cdots\sotimes A_n$ 
by $\pi$ is given by 
$a_1^{(1)} a_2^{(2)}\cdots a_n^{(n)}\in A_1\scoprod\cdots\scoprod A_n$. 

Note that, when all $A_k$ coincide and we are given an operation $p$ 
of arity $n$, the map $p\, \pi: A^{\scoprod n} \longrightarrow A$ 
is not, in general, an algebra homomorphism (because $p$ is not), 
and therefore it surely differs from the folding map
$\mu= \langle \Id_A,...,\Id_A \rangle$. 
In fact, $\mu$ multiplies the elements of $A$ in the order they appear 
in $A^{\scoprod n}$ (it is a {\em concatenation}), 
while $p\, \pi$ first reorders the factors in $A^{\scoprod n}$ with $\pi$,
as explained above, then multiplies them (it is a {\em componentwise 
operation}).  
\end{definition}

\begin{definition}
\label{canonical-inclusion}
On the other side, for any $n\geq 2$ and for any $n$ algebras $A_k$, with
$k=1,...,n$, there are categorical maps 
$i_k:A_k\rightarrow A_1\scoprod\cdots\scoprod A_n$. 
For any operation $p$ of arity $n$ in $\catA$, we call 
{\bf canonical inclusion} of $A_1\sotimes\cdots\sotimes A_n$ in 
$A_1\scoprod\cdots\scoprod A_n$ the linear map 
$\iota_p:A_1\sotimes\cdots\sotimes A_n \longrightarrow
A_1\scoprod\cdots\scoprod A_n$
defined by 
$$
\iota_p(a_1\sotimes\cdots\sotimes a_n) := p_\scoprod\big(i_1(a_1),...,i_n(a_n)\big), 
$$
where $p_\scoprod:(A_1\scoprod\cdots\scoprod A_n)^{\sotimes n}\rightarrow 
A_1\scoprod\cdots\scoprod A_n$ 
denotes the operation $p$ on the coproduct algebra
$A_1\scoprod\cdots\scoprod A_n$. 
It follows from the definition of $\scoprod$ that this map is injective. 

Note that $\iota_p$ is not, in general, an algebra homomorphism, because 
the operation $p$ in $\catA$ is not. 
However, when all $A_k$ coincide (say, with $A$), the map $\iota_p$ allows us
to recover the operation $p_A:A^{\sotimes n}\rightarrow A$ from the folding map
$\mu$, in the sense that $\mu\, \iota_p=p_A$, because 
$$
\mu\ p_\scoprod\big(i_1(a_1),...,i_n(a_n)\big) = p_A(a_1,...,a_n)
$$ 
for any $a_1,...,a_n\in A$. 
\end{definition}

\begin{proposition}
When the map $\iota_p$ is well defined, we have
$\pi\,\iota_p=\Id_{A_1\otimes \cdots \otimes A_n}$.
\end{proposition}

\begin{proof}
Denote by $p_{\otimes}$ the operation $p$ on the tensor algebra
$A_1\sotimes\cdots\sotimes A_n$. Since $\pi$ is an algebra homomorphism,
for any $a_k\in A_k$, with $k=1,...,n$, we have 
\begin{align*}
\pi\, \iota_p(a_1\sotimes\cdots\sotimes a_n)
&= \pi\, p_\scoprod\big(i_1(a_1),...,i_n(a_n)\big) 
= p_{\otimes}\Big(\pi\big(i_1(a_1)\big),...,\pi\big(i_n(a_n)\big)\Big)  
\\ 
& = p_{\sotimes}\Big(\langle j_1,...,j_n \rangle \big(i_1(a)\big),...,
\langle j_1,...,j_n \rangle \big(i_n(a_n)\big)\Big)
\\
& = p_{\otimes}\Big(j_1(a_1),...,j_n(a_n)\Big)
\\
& = a_1\sotimes \cdots \sotimes a_n. 
\end{align*}
\end{proof}

\begin{remark}
These maps allow us in particular to compare the coloop bialgebra representing 
some loop to other types of bialgebras related to it which appear in the 
literature. In particular, the universal enveloping algebra of the Sabinin 
algebra associated to the loop has been studied in 
\cite{Perez-Izquierdo, MostovoyPerezIzquierdo, 
MostovoyPerezIzquierdoShestakov-Hopf}. 
Because of the axioms, it is clear that the graded dual of this universal 
enveloping algebra does not coincide with the bialgebra $H^\otimes$ induced 
by a coloop bialgebra $H$, nor in $\Alg_\F$ nor in $\As_\F$.
\end{remark}

Finally, let us use these maps to compare associative coloop bialgebras and 
Hopf algebras. Let $H$ be a coloop bialgebra in $\As_\F$. 
Denote by $H^\otimes$ the algebra $H$ endowed with the usual co-operations 
$$
\D^\otimes=\pi\,\D,\ \delta_r^\otimes=\pi\,\delta_r,\ 
\delta_l^\otimes=\pi\,\delta_r: 
H^\otimes \longrightarrow H^\otimes\otimes H^\otimes, 
$$
the counit $\varepsilon$ and the antipodes $S_r$, $S_l$, which are all still 
algebra homomorphisms on $H^\otimes$. 

\begin{proposition}
\label{Coloop to Hopf} 
If $\D$ is coassociative, then $\D^\otimes$ is coassociative.
Moreover, we have 
$$
S_r= (\varepsilon \otimes \Id)\,\delta_r^\otimes 
\qquad\mbox{and}\qquad 
S_l= (\Id\otimes \varepsilon)\,\delta_l^\otimes.  
$$
\end{proposition}

\begin{proof}
If $\D$ is coassociative, the two terms 
$$
(\D^\otimes\otimes\Id)\,  \D^\otimes 
= (\pi\otimes\Id)\,  \pi_{(H\scoprod H)\scoprod H} \,  (\D\scoprod \Id)\,  \D
$$
and 
$$
(\Id\otimes \D^\otimes)\,  \D^\otimes 
= (\Id\otimes \pi)\,  \pi_{H\scoprod (H\scoprod H)} \,  (\Id\scoprod \D)\,  \D 
$$
coincide, because $\D$ is coassociative and because the two maps 
$(\pi\otimes\Id)\otimes \pi_{(H\scoprod H)\scoprod H}$ 
and $(\Id\otimes \pi)\otimes \pi_{H\scoprod (H\scoprod H)}$ coincide with the 
standard projection 
$\pi:H\scoprod H\scoprod H \longrightarrow H\otimes H\otimes H$. 
Therefore $\D^\otimes$ is coassociative. 

For any $a\in H$, the term $\delta_r(a)\in H^{(1)}\scoprod H^{(2)}$ 
is a finite sum of products of elements of $H^{(1)}$ and of $H^{(2)}$ 
in alternative order. 
The right antipode $S(a)_r=(\varepsilon\scoprod \Id)\,\delta_r(a)$ 
turns all the factors belonging to $H^{(1)}$ into scalars, which can then be 
positioned on the lefthand side of all the remaining elements belonging to 
$H^{(2)}$. Therefore the result is the same that we obtain if we first 
reorder the factors in $H^{(1)}$ all at the leftmost position by 
applying $\delta_r(a)^\otimes$. 
Same with $S_l$ by putting all the scalars on the rightmost position. 
\end{proof}

Note however that $S_r$ and $S_l$ do not necessarily satisfy the left 
and right 5-terms identities for $\D^\otimes$ on $H^\otimes$, because 
$$
m\,(S_r\sotimes\Id)\,\D^\otimes = m\,\pi\,(S_r\scoprod\Id)\,i\,\D^\otimes 
= \mu\,\iota\,\pi\,(S_r\scoprod\Id)\,\iota\,\pi\,\D
$$
and $\iota\,\pi$ is not the identity map on $H\scoprod H$. 
Therefore, even if $H$ is a cogroup bialgebra, $H^\otimes$ is not necessarily 
a Hopf algebra. 


\section{Coloops of invertible and unitary elements} 
\label{section-unitary}

\subsection{Loop of invertible elements} 

In this section we give an example of an abelian algebraic group which can be 
extended to associative algebras as a group, to alternative algebras as a loop, 
but not to non-associative algebras, even as a loop. 
\medskip 

Let $\F$ be a field. 
For any unital commutative algebra $A$ over $\F$, the set 
$$
I(A) = \{ a\in A\ |\ \mbox{$a$ admits an inverse $a^{-1}$}\ \}
$$
is the abelian group of invertible elements in $A$. 
The functor $I$ is represented on $\Com_\F$ by the commutative 
(and cocommutative) Hopf algebra of Laurent polynomials $\HI = \F[x,x^{-1}]$, 
with co-operations 
\begin{align*}
\Delta (x) = x \otimes x, \qquad \varepsilon (x) = 1, \qquad S(x) = x^{-1}. 
\end{align*}
In fact, elements $a\in I(A)$ are in bijection with algebra homomorphisms
$\alpha:\HI\to A$ such that $\alpha(x)=a$ and $\alpha(x^{-1})=a^{-1}$.
Then, if $\alpha$ and $\beta$ give respectively
the elements $a$ and $b$, their convolution product coincides with the product
in A, because we have
\begin{align*}
  (\alpha \cdot \beta)(x) & = \mu_A (\alpha\otimes \beta) \Delta(x)
  = \alpha(x) \beta(x) = ab, \\
  \alpha^{-1}(x) & = \alpha S(x) = \alpha (x^{-1})=a^{-1}. 
\end{align*}
We show that the functor $I$ admits an extention to associative algebras as 
a group, and that it admits an extention to non-associative algebras, as a loop, 
only on alternative algebras.  

\begin{definition}
\label{invertible-coloop}
We call {\bf invertible coloop bialgebra} on $\F$ the 
associative algebra $\HIex = \F[x,x^{-1}]$ endowed with the following co-operations 
with values in the coproduct $\HIex \scoprod \HIex$ of the category $\As_\F$: 
\begin{align*}
& \D(x) = x^{(1)}\, x^{(2)} && \D(x^{-1}) = (x^{-1})^{(2)}\, (x^{-1})^{(1)},
\\
& \epsilon(x) = 1 && \epsilon(x^{-1}) = 1 ,
\\
& \delta_r(x) = x^{(1)}\, (x^{-1})^{(2)} && \delta_r(x^{-1}) = x^{(2)}\, (x^{-1})^{(1)} ,
\\
& \delta_l(x) = (x^{-1})^{(1)}\, x^{(2)} && \delta_l(x^{-1}) = (x^{-1})^{(2)}\, x^{(1)} ,  
\end{align*}
where $x^{(k)}=i_k(x)$ is the generator $x$ seen in the $k$th copy of $\HIex$
of the coproduct algebra $\HIex \scoprod \HIex$, for $k=1,2$, and similarly
for $(x^{-1})^{(k)}=i_k((x^{-1})^{(k)})$. 

It follows that there is a two-sided antipode given by $S(x)=x^{-1}$ and 
$S(x^{-1})=x$. 
\end{definition}

\begin{proposition}
The algebra $\HIex$ is a cogroup bialgebra in $\As_\F$ and represents, 
for any associative algebra $A$, the group 
$$
I(A) = \Hom_{\As_\F}(\HIex,A)
$$
of invertible elements of $A$. 
Moreover, the group $I(A)$ is abelian if $A$ is commutative. 
\end{proposition}

\begin{proof}
The axioms of a coloop bialgebra for the codivisions are easily verified.
For instance, the computations
\begin{align*}
(\Id\scoprod \mu)(\delta_r\scoprod \Id) \Delta (x)
  & = (\Id\scoprod \mu)(\delta_r\scoprod \Id)(x^{(1)} x^{(2)}) \\ 
  & = (\Id\scoprod \mu)(x^{(1)} (x^{-1})^{(2)} x^{(3)}) \\ 
  & = x^{(1)} (x^{-1})^{(2)} x^{(2)} \\
  & = x^{(1)} = i_1(x) \\
(\Id\scoprod \mu)(\Delta\scoprod \Id) \delta_r (x)
  & = (\Id\scoprod \mu)(\Delta\scoprod \Id)(x^{(1)} (x^{-1})^{(2)}) \\ 
  & = (\Id\scoprod \mu)(x^{(1)} x^{(2)} (x^{-1})^{(3)}) \\ 
  & = x^{(1)} x^{(2)} (x^{-1})^{(2)} \\
  & = x^{(1)} = i_1(x)   
\end{align*}
and the analogue computations for $x^{-1}$ prove the cocancellation
(\ref{right-cocancellation}) for $\delta_r$. 
The first claim is then ensured by the fact that $\D$ is coassociative. 
In fact,
\begin{align*}
  (\Delta \scoprod \Id) \Delta (x) & = x^{(1)}\, x^{(2)}\, x^{(3)}
  = (\Id\scoprod \Delta) \Delta (x),  
\end{align*}
and similarly for $x^{-1}$. 
Thus, $I(A)$ is a group by Theorem \ref{theorem-Q-functor} and Proposition
\ref{coloop-cogroup}. 

The fact that the group $I(A)$ is abelian if $A$ is commutative 
is less evident because $\D$ is not cocommutative. In fact, we have 
$$
\tau\,\D(x) = \tau(x^{(1)}\, x^{(2)}) = x^{(2)}\, x^{(1)} 
\neq x^{(1)}\, x^{(2)} = \D(x). 
$$
It is however true because the generators $x$ and $x^{-1}$ are group-like, and 
therefore the commutativity of the convolution product only depends on that 
of the multiplication in $A$. 
\end{proof}

\begin{example}
The group $I(A)$ is the simplest algebraic group at all: it describes
invertible elements in an associative algebra $A$ whatever is the nature of
$A$, that is, without making use of any internal structure of $A$. 
The simplest non-trivial example is the group $I(M_n(\F)) = GL_n(\F)$,
which is recovered as the set of $M_n(\F)$-valued algebra homomorphisms
on $\HI$ without using the non-homogeneous relation $\det(A)\neq 0$
which defines invertible matrices (or, more precisely, the relation
$\det(A)=t$ where $t$ determines a new scalar invertible generator
of the coordinate ring). 
\end{example}

\begin{proposition}
\label{invertible-loop}
The algebraic group $I$ can be extended as a loop to a variety of algebras 
$\catA \subset \Alg_\F$ if $\catA$ is a subcategory of alternative algebras
$\Alt_\F$ admitting coproduct and initial object. 
In particular, it is an algebraic loop on $\Alt_\F$. 
\end{proposition}

\begin{proof}
If $I$ could be extended as an algebraic loop to $\Alg$, its representative 
coloop bialgebra should be the algebra $\HIex = \F[x,x^{-1}]$ with 
co-operations defined on generators as in Def.~\ref{invertible-coloop} but 
taking values in the coproduct $\HIex\scoprod \HIex$ of the category $\Alg$. 
This algebra is not a coloop bialgebra in $\Alg$, because the 
codivisions do not satisfy the cocancelation properties 
(\ref{left-cocancellation}) and (\ref{right-cocancellation}). 
In fact, the element 
\begin{align*}
(\Id\scoprod \mu)\,(\delta_r\scoprod \Id)\,\Delta(x) 
& = \big(x^{(1)}\, (x^{-1})^{(2)}\big)\, x^{(2)} 
\end{align*}
can not coincide with $i_1(x) = x^{(1)}$ in $\HIex\scoprod \HIex$. 
However, the conditions under which the cocancellation properties hold, all 
similar to the one above, are guaranteed in the category of alternative 
algebras, where $(ab)b^{-1}=a=b^{-1}(ba)$ for any $a$ and any invertible $b$
(cf.~\cite{ZSSS}). 
\end{proof}

\begin{example}
The octonions $\O$ form an alternative algebra, therefore one can apply $I$
to $\O$. The set $I(\O)$ of invertible octonions is a well
known Moufang loop (cf.~\cite{Bruck}), that is, it is a loop satisfying
the Moufang identities 
\begin{align*}
a(b(ca)) = ((ab)a)c \qquad\qquad (ab)(ca) & = (a(bc))a \\
a(b(cb)) = ((ab)c)b \qquad\qquad (ab)(ca) & = a((bc)a 
\end{align*} 
for any elements $a,b,c$. 
\end{example}


\subsection{Loop of unitary elements} 
\label{subsection-unitary}

Consider now involutive algebras $A$, and the subgroup of $I(A)$ made of 
unitary elements in $A$, namely
$$
U(A) = \{ a\in A\ |\ a\,a^*=1 \}, 
$$
when $A$ is commutative. Exactly as for $I$, the functor $U$ is represented 
on $\Com_\F^*$ by the commutative Hopf algebra $\HU = \F[x,x^*\ |\ x\,x^*=1 ]$, 
with co-operations 
\begin{align*}
\Delta (x) = x \otimes x, \qquad  
\varepsilon (x) = 1, \qquad 
S(x) = x^* . 
\end{align*}

\begin{definition}
\label{unitary-coloop}
Let us call {\bf unitary coloop bialgebra} on $\F$ the associative algebra 
$\HUex = \F[x,x^*\ |\ x\,x^*=1 ]$ endowed with the co-operations defined 
on generators exactly as those in Def.~\ref{invertible-coloop}, 
where the generator $x^{-1}$ is replaced by $x^*$. 
\end{definition}

As for the invertible coloop bialgebra, one can prove that 

\begin{proposition}
The algebra $\HUex$ is a cogroup bialgebra in $\As_\F^*$ and represents, 
for any involutive associative algebra $A$, the group 
$$
U(A) = \Hom_{\As_\F^*}(\HUex,A)
$$
of unitary elements of $A$. 
Moreover, the group $U(A)$ is abelian if $A$ is commutative. 
\end{proposition}

\begin{examples}
For $\F=\R$, this functor allows us to describe several groups of 
unitary matrices. 
\begin{enumerate}
\item
Applied to the algebra $M_n(\R)$, if we take the transposition of matrices as 
involution, it gives $U(\R) = \{1,-1\}$, and the orthogonal group 
$U\big(M_n(\R)\big)=O(n)$ for $n>1$.  

\item 
On $M_n(\C)$, we take as involution the complex conjugate of the transposition. 
Then $U(\C)=U(1) \cong \S^1$ and $U\big(M_n(\C)\big)=U(n)$ is the unitary 
group. 

\item 
Let $\H$ be the algebra of quaternions, spanned over $\R$ by $1$ and by
three imaginary units $i$, $j$, $k$ which anticommute with each other.
The conjugate of a quaternion $q=a+b\,i+c\,j+d\,k$ is the quaternion
$q^*=a-b\,i-c\,j-d\,k$. The conjugation is an involution, and the
real number $\| q\| = \sqrt{q q^*} = \sqrt{q^* q} = \sqrt{a^2+b^2+c^2+d^2}$
defines a multiplicative norm on $\H$. Then, the functor $U$ applied to $\H$
gives the subgroup $U(\H) \cong Sp(1) \cong SU(2) \cong \S^3$ of $I(\H)$
consisting of unit norm quaternions. 

On the set of matrices $M_n(\H)$, we take as involution the quaternionic
conjugate of the transposition. Then
$U\big(M_n(\H)\big) \cong U(n,\H) \cong Sp(n)$ is the compact symplectic group,
also called the hyperunitary group. 
\end{enumerate}
\end{examples}

Again exactly as for the invertible coloop bialgebra, one can prove the 
next result. 

\begin{proposition}
\label{unitary-loop}
The algebraic group $U$ can be extended as a loop to a variety of algebras 
$\catA \subset \Alg_\F^*$ if $\catA$ is a subcategory of involutive 
alternative algebras $\Alt_\F^*$ admitting coproduct and initial object. 
In particular, $U$ is an algebraic loop on $\Alt_\F^*$. 
\end{proposition}

\begin{example}
\label{unitary-loop-example}
An alternative algebra of octonions $\O=\O(\alpha,\beta,\gamma)$ is spanned
over a field $\F$ of characteristic not $2$ by $e_0=1$ and by seven imaginary
units $e_i$, for $i=1,...,7$, with an involved table of multiplication
(cf. \cite{Bruck}, \cite[Ch.2]{ZSSS}, \cite{KuzminShestakov}).
Over the field $\R$ there are two non-isomorphic octonion algebras:
the classical division Cayley octonions $\O$ and the split matrix
Cayley-Dickson algebra $\Zorn(\R)$ \cite{Bruck}, also known as the Zorn
vector-matrix algebra \cite{Paige}. The last one may be defined over an
arbitrary commutative ring.
The conjugate of an octonion $q=\sum_{i=0}^7a_ie_i$ is the octonion
$q^*=a_0e_0-\sum_{i=1}^7a_ie_i$. Again, the conjugation is an involution,
and the scalar $n(q)=qq^*=q^*q$ defines a multiplicative norm on $\O$
(and an isotropic quadratic form on $\Zorn(\R)$).

Then, for the classical Cayley octonions $\O$, the set $U(\O)$ is the Moufang
subloop of the loop $I(\O)$ consisting of unit norm octonions, which is  
homeomorphic to the sphere $S^7$, while for the matrix Cayley-Dickson algebra
$\Zorn(\R)$ the loop $U(\Zorn(\R))$ is not compact. 
The loops $U(\O)$ and $U(\Zorn(\R))$ can be compared to the groups
$U(\H)\cong \S^3$ and $U(\P)\cong SL_2(\R)$ obtained respectively for
division quaternions $\H$ and for split quaternions $\P$.  
\end{example}


\subsection{Unitary Cayley-Dickson loops}

In this section we give an example of a loop which is not algebraic on 
associative algebras. 
\medskip 

Let $\F$ be a field and $j$ denote an imaginary unit.  
For any involutive commutative algebra $A$ over $\F$, the set 
$$
U_{CD}(A) = \{ a+b\,j \in A+A\,j\ |\ a\,a^*+b\,b^*=1 \}
$$
gives the group of unitary elements in the Cayley-Dickson algebra $A+A\,j$ 
with multiplication 
$$
(a+b\,j)\,(c+dj) = (ac-d^*b)+(da+bc^*)\,j ,
$$
unit $1$, and involution $(a+b\,j)^* = a^*-b\,j$. 

The functor $A\longmapsto U_{CD}(A)$ is representable on $\Com_\F^*$, 
by the commutative Hopf algebra 
$$
\HUCD = \F[x,x^*,y,y^*\ |\ x\,x^*+y\,y^*=1 ] 
$$
with co-operations 
\begin{align*}
& \Delta (x) = x\sotimes x - y\sotimes y^* & 
& \Delta (y) = x\sotimes y + y\sotimes x^*,  
\\  
& \varepsilon (x) = 1 & & \varepsilon (x) = 0, 
\\
& S(x) = x^* & & S(y) = -y . 
\end{align*}

\begin{proposition}
\label{Cayley-Dickson-loop}
The algebraic group $U_{CD}$ can not be extended as an algebraic loop to the
category of involutive associative algebras. 
\end{proposition}

\begin{proof}
If $U_{CD}$ could be extended to an algebraic loop to $\As_\F^*$, its 
representative coloop bialgebra $\HUCDex$ should be an associative algebra 
generated by $x$, $x^*$, $y$ and $y^*$ submitted to conditions which give 
$x\,x^*+y\,y^*=1$ if the variables commute. 
The co-operations should then be defined on generators exactly as in the 
commutative case, but taking values in the coproduct $\HUCDex\scoprod \HUCDex$
of the category $\As_\F^*$. 

The conditions $x\,x^*=x^*x$ and $x\,x^*+y^*y=1$ are enough to guarantee that 
the algebra $\HUCDex$ has a well defined comultiplication, a counit and an 
antipode satisfying the 5-terms relations. 
However, the codivisions, defined according to the coinverse properties 
(\ref{coinverse}) as 
\begin{align*}
& \delta_r(x) = x^{(1)}\, (x^*)^{(2)} + (y^*)^{(2)}\, y^{(1)} 
& & \delta_r(y) = - y^{(2)}\, x^{(1)} + y^{(1)}\, x^{(2)}
\\
& \delta_l(x) = (x^*)^{(1)}\, x^{(2)} + (y^*)^{(2)}\, y^{(1)} 
& & \delta_l(y) = y^{(2)}\, (x^*)^{(1)} - y^{(1)}\, (x^*)^{(2)} , 
\end{align*}
satisfy the cocancellation identities (\ref{right-cocancellation}) and  
(\ref{left-cocancellation}) if and only if 
$$
x^{(1)}\,(x^*x)^{(2)} = (x^*x)^{(2)}\,x^{(1)} 
\qquad\mbox{and}\qquad 
y^{(1)}\,(y^*y)^{(2)} = (y^*y)^{(2)}\,y^{(1)} 
$$
in $\HUCDex\scoprod \HUCDex$. This could happen for two reasons. 
The first is that the map $n:\HUCDex\longrightarrow \HUCDex$ given by 
$n(a) = a a^* = a^*a$ has scalar values, i.e. its image is in 
$u(\F)\subset \HUCDex$. 
This is the case if $\HUCDex$ is a {\em composition algebra},
cf.~\cite{Albert}. 
But composition algebras do not have a categorical coproduct. 
The second possibility to verify these conditions is that the identity
$a^{(1)} b^{(2)} = b^{(2)} a^{(1)}$ holds in $\HUCDex\scoprod \HUCDex$ for any 
elements $a,b\in \HUCDex$. This means that $\scoprod = \sotimes$ and 
therefore it is only possible in the category $\Com_\F^*$. 
\end{proof}

\begin{examples} 
In agreement with this result, namely that the construction $U_{CD}$ is not
functorial on associative algebras, there are few examples of loops arising 
as sets of unitary elements in the Cayley-Dickson algebra constructed 
on an associative algebra. For instance, we can consider the associative
algebras of matrices $A=M_n(K)$ with entries in involutive algebras $K$ over
the field $\F=\R$, with involution given by the transposition of the matrices
plus the involution of their matrix elements.
The unitary elements in $A+Aj$ are preserved by divisions if $A$ is a
composition algebra, and matrix algebras, in general, are not. So, in general,  
$U_{CD}(A)$ is not a loop. There are few exceptions: 
\begin{enumerate}
\item
The set $U_{CD}(M_n(\R))$ is a loop for $n=1,2$. 
For $n=1$ (when $A=\R$ is commutative) it is an abelian group
$U_{CD}(\R) = U(\C) \cong \S^1$, and for $n=2$ the loop $U_{CD}(M_2(\R))$
coincides with the loop $U(\Zorn(\R))$ from Example
\ref{unitary-loop-example}, since $M_2(\R)+M_2(\R)j\cong \Zorn(\R)$ is a
matrix Cayley-Dickson algebra.

\item 
The set $U_{CD}(M_n(\C))$ is a loop for $n=1,2$.
For $n=1$ (when $A=\C$ is commutative) it is a group
$U_{CD}(\C) = U(\H) \cong \S^3$, and for $n=2$ the loop $U_{CD}(M_2(\C))$
coincides with the loop $U(\Zorn(\C))$ of unital elements in the split
matrix Cayley-Dickson algebra over the complex numbers $\C$.

\item
The set $U_{CD}(M_n(\H))$ is a Moufang loop only for $n=1$, and we have 
$U_{CD}(\H) \cong U(\O) \cong \S^7$. For $n>1$, the set $U_{CD}(M_n(\H))$
is not a loop because $M_n(\H)$ is not a composition algebra. 
\end{enumerate}
\end{examples}


\section{Coloop of invertible series}
\label{section-invertible series}

The group of invertible series (with constant term equal to $1$), is the set 
of formal series
$$
\Inv(A) = 
\Big\{\ a(\lambda)=\sum_{n\geq 0} a_n\ \lambda^n\ |\ a_0=1,\ a_n\in A\ \Big\}
$$
with coefficients $a_n$ taken in a commutative algebra $A$, 
endowed with the pointwise multiplication 
$(ab)(\lambda)=a(\lambda)\ b(\lambda)$, unit $1(\lambda)=1$, 
and where the inverse of a series $a(\lambda)$ is found by recursion. 
It is an abelian proalgebraic group on $\Com$, represented by the 
cocommutative Hopf algebra 
\begin{align*}
\Hi & = \F[x_n,\ n\geq 1] \qquad (x_0=1) \\
\Di(x_n) & = \sum_{m=0}^n x_m\otimes x_{n-m} 
\end{align*}
known as Hopf algebra of symmetric functions~\cite{Geissinger}. 

The functor $\Inv$ admits an evident extention to associative algebras 
as a functor in groups (but not abelian), represented by the cogroup bialgebra 
\cite{BFK} 
\begin{align*}
\Hi & = \F\langle x_n,\ n\geq 1 \rangle \qquad (x_0=1) \\
\Di(x_n) & = \sum_{m=0}^n x_m^{(1)}\ x_{n-m}^{(2)} 
\end{align*}
with antipode defined recursively. The projection of this bialgebra by the 
canonical map $\pi$ given in Def.~\ref{canonical-projection} coincides with 
the Hopf algebra of non-commutative symmetric functions (cf.~\cite{GKLLRT}). 

In this section we show that the functor $\Inv$ can be extended to 
non-associative algebras, as a proalgebraic loop. 


\subsection{Loop of invertible series}

\begin{definition}
Let $A$ be a unital algebra and let $\lambda$ be a formal variable. 
We call invertible series in $\lambda$ with coefficients in $A$ 
the formal series in the set 
$$
\Inv(A)=\Big\{\ a=\sum_{n\geq 0} a_n\,\lambda^n\ |\ a_0=1,\ a_n\in A\ \Big\}, 
$$ 
endowed with the multiplication
$$
a\cdot b = 
\sum_{n\geq 0}\ \sum_{m=0}^n\ a_m\ b_{n-m}\,\lambda^n
$$ 
and the unit $e$ given by $e_0=1$ and $e_n=0$ for all $n>1$. 
For instance, 
\begin{align*}
(a\cdot b)_1 & = a_1+b_1, \\ 
(a\cdot b)_2 & = a_2+a_1b_1+b_2, \\
(a\cdot b)_3 & = a_3+a_2b_1+a_1b_2+b_3.  
\end{align*}
\end{definition}

\begin{proposition}
For any unital algebra $A$, the set of invertible series $\Inv(A)$ is a loop. 
\end{proposition}

\begin{proof}
It is clear that the series $e$ is a unit for the given multiplication, so 
we only have to show that there exist a left and a right divisions satisfying 
the cancellation properties (\ref{right-cancellation}) and 
(\ref{left-cancellation}). 
Since the multiplication is completely symmetric in the the two variables, 
the proof for the two divisions is exactly the same. We do it for the right 
division. 

Given two series $a=\sum a_n\,\lambda^n$ and $b=\sum b_n\,\lambda^n$, 
we define the right division $a\slash b=\sum (a\slash b)_n\,\lambda^n$ 
so that $(a\slash b)\dot b = a$, that is
$$
\sum_{m=0}^n\ (a\slash b)_m\ b_{n-m} = a_n  \qquad\mbox{for any $n\geq 0$}.  
$$ 
These equations are solved recusively from $(a\slash b)_0=1$, and give 
the $n$th term
$$
(a\slash b)_n=a_n-\sum_{m=0}^{n-1} (a\slash b)_m\ b_{n-m}. 
$$
Let us then prove by induction that $(a\cdot b)\slash b = a$, that is, 
$\big((a\cdot b)\slash b\big)_n = a_n$ for any $n\geq 0$. 
We have $\big((a\cdot b)\slash b\big)_0 = a_0=1$ and, for any $n\geq 1$, 
\begin{align*}
\big((a\cdot b)\slash b\big)_n 
& = (a\cdot b)_n-\sum_{m=0}^{n-1} (a\cdot b)_m\,b_{n-m} 
\\
& = a_n+ \sum_{m=0}^{n-1} \big(a_m-(a\cdot b)_m\big)\,b_{n-m} , 
\end{align*}
so if we suppose that $\big((a\cdot b)\slash b\big)_m = a_m$ for any $m\geq n-1$, 
we have $\big((a\cdot b)\slash b\big)_n = a_n$. 
\end{proof}

For instance, for the right division we find 
\begin{align*}
(a\slash b)_1 & = a_1-b_1, 
\\ 
(a\slash b)_2 & = a_2-a_1b_1-b_2+b_1b_1, 
\\ 
(a\slash b)_3 & = a_3-(a_1b_2+a_2b_1)+(a_1b_1)b_1-b_3+(b_1b_2+b_2b_1)-(b_1 b_1)b_1,  
\end{align*}
and for the left division we find
\begin{align*}
(a\backslash b)_1 & = b_1-a_1, 
\\ 
(a\backslash b)_2 & = b_2-a_1b_1-a_2+a_1a_1, 
\\ 
(a\backslash b)_3 & = b_3-(a_1b_2+a_2b_1)+a_1(a_1b_1)-b_3+(a_1a_2+a_2a_1)-a_1(a_1a_1).
\end{align*}


\subsection{Coloop bialgebra of invertible series}
\label{subsection-invertible-coloop-bialgebra}

For any $n\geq 1$, let $x_n$ be a graded variable of degree $n$. 
For $X=\Span_\F\{x_n,\ n\geq 1\}$, the tensor algebra $H=T(X)$ can be seen as 
the set of non-commutative polynomials in the variables $x_1,x_2,...$, that 
we denote by $\F\langle x_n,\ n\geq 1\rangle$. It is then useful to denote 
the unit $1$ of $H$ by $x_0$. 

The unital associative coproduct algebra $H\scoprod H$ is then the tensor 
algebra $T(X^{(1)}\oplus X^{(2)})$ on two identical sets of variables, and
similarly $H\scoprod H\scoprod H = T(X^{(1)}\oplus X^{(2)}\oplus X^{(3)})$. 
To simplify the notations, in this section we denote by $x_n=x_n^{(1)}$, 
$y_n=x_n^{(2)}$ and $z_n=x_n^{(3)}$ the generators taken in the different copies
of $X$ in a coproduct algebra.
\medskip 

For any integer $n\geq 1$ and any $1\leq \ell \leq n$, let $\calC_n^\ell$ 
denote the set of compositions of $n$ of length $\ell$, that is, 
the set of ordered sequences $\bfn = (n_1,...,n_\ell)$ such that 
\begin{align}
\label{sequence-C}
n_1+\cdots +n_\ell=n, \qquad\mbox{and}\qquad n_1,...,n_\ell\geq 1.  
\end{align}
For instance, for $\ell=1,2,3$, we have 
\begin{align*}
& \calC_1^1 = \big\{ (1) \big\}, \qquad 
\calC_2^1 = \big\{ (2) \big\}, \quad 
\calC_2^2 = \big\{ (1,1) \big\}, 
\\ 
& \calC_3^1 = \big\{ (3) \big\}, \quad 
\calC_3^2 = \big\{ (2,1),(1,2) \big\}, \quad 
\calC_3^3 = \big\{ (1,1,1) \big\}. 
\end{align*}

\begin{definition}
Let us call {\bf coloop bialgebra of invertible series} the free unital algebra 
$$
\Hiex = T\{x_n\ |\ n\geq 1\} 
$$
with the following graded co-operations: 
\begin{itemize}
\item
comultiplication $\Diex: \Hiex \longrightarrow \Hiex \scoprod \Hiex$ given by
\begin{align*}
\Diex(x_n) &= \sum_{m=0}^n x_m\,y_{n-m} ; 
\end{align*}

\item
counit $\varepsilon:\Hiex\longrightarrow \F$ given by 
$\varepsilon(x_n) = \delta_{n,0}$; 

\item
right codivision $\delta_r:\Hiex \longrightarrow \Hiex \scoprod \Hiex$ 
given by
\begin{align*}
\delta_r(x_n) &= x_n-y_n + \sum_{\ell=1}^{n-1} (-1)^\ell \sum_{\bfn\in \calC_n^{\ell+1}}
\Big(\big(((x_{n_1}-y_{n_1})\,y_{n_2})\, y_{n_3}\big)\cdots \Big) y_{n_{\ell+1}}, 
\end{align*}
where $\calC_n^{\ell+1}$ is the set of compositions of $n$ of length $\ell+1$, 
cf.~(\ref{sequence-C}); 

\item
left codivision $\delta_l:\Hiex \longrightarrow \Hiex \scoprod \Hiex$ 
given by
\begin{align*}
\delta_l(x_n) &= y_n-x_n + \sum_{\ell=1}^{n-1} (-1)^\ell \sum_{\bfn\in \calC_n^{\ell+1}}
x_{n_1}\Big(\cdots \big(x_{n_2}\,(x_{n_\ell} (y_{n_{\ell+1}}-x_{n_{\ell+1}}))\big) \Big).  
\end{align*}
\end{itemize} 
\end{definition}

\begin{theorem}
\label{thm-invertible series}
The algebra $\Hiex$ is a coloop bialgebra and represents the loop of 
invertible series as a functor $\Inv:\Alg\longrightarrow \Loop$. 

As a consequence, given an algebra $A$, a series 
$a=\sum_{n\geq 0} a_n\,\lambda^n \in \Inv(A)$ can be seen as an algebra 
homomorphism $a:\Hiex\longrightarrow A$ defined on the generators of $\Hiex$ 
by $a(x_n)=a_n$, and the right and left division $a\slash b$ and 
$a\backslash b$ are given at any order $n$ by the following closed formulas: 
\begin{align*}
(a\slash b)_n & = \mu_A\,(a\scoprod b)\, \delta_r(x_n) 
\\ 
& = a_n-b_n + \sum_{\ell=1}^{n-1} (-1)^\ell \sum_{\bfn\in \calC_n^{\ell+1}} 
\Big(\big(((a_{n_1}-b_{n_1}) b_{n_2})b_{n_3}\big)\cdots\Big)b_{n_{\ell+1}} , 
\\ 
(a\backslash b)_n & = \mu_A\,(a\scoprod b)\, \delta_l(x_n) 
\\
& = b_n-a_n + \sum_{\ell=1}^{n-1} (-1)^\ell \sum_{\bfn\in \calC_n^{\ell+1}}
a_{n_1}\Big(\cdots \big(a_{n_2}\,(a_{n_\ell} (b_{n_{\ell+1}}-a_{n_{\ell+1}}))\big) \Big) . 
\end{align*}
\end{theorem}

\begin{proof}
The algebra $\Hiex$ clearly represents the functor $\Inv$ with values in sets, 
and the comultiplication $\Diex$ represents the pointwise multiplication 
of series. 
The only thing which should be proved is that $\Hiex$ is a coloop bialgebra 
with the given codivisions. 
The formulas for the left and for the right codivisions are perfectly 
symmetric, in the sense that $\delta_l=\tau\, \delta_r$, 
so it suffices to give the details for one codivision. 
Let us then show that the right codivision satisfies the two equations 
(\ref{right-cocancellation}). 

Concerning the first one, we have 
\begin{align*}
(\delta_r\scoprod \Id)\Diex (x_n) 
& = \delta_r(x_n) + z_n + \sum_{m=1}^{n-1} \delta_r(x_m)\,z_{n-m}, 
\end{align*}
which is an element of $\Hiex\scoprod \Hiex\scoprod \Hiex$, and since
$\Id\scoprod \mu:\Hiex\scoprod \Hiex\scoprod \Hiex \longrightarrow
\Hiex\scoprod \Hiex$ multiplies the variables $y$ and $z$
(in the order they appear) and puts the result in the right-hand side copy of
$\Hiex$ in $\Hiex\scoprod \Hiex$, we have 
\begin{align*}
(\Id\scoprod \mu)(\delta_r\scoprod \Id)\Diex (x_n) 
& = \delta_r(x_n) + y_n 
+ \sum_{m=1}^{n-1} \delta_r(x_m)\,y_{n-m} 
\\
& = x_n - y_n +\sum_{\ell=1}^{n-1} (-1)^\ell \sum_{\bfn\in \calC_n^{\ell+1}}
\Big(\big((u_{n_1}\,y_{n_2})\, y_{n_3}\big)\cdots \Big) y_{n_{\ell+1}} 
+ y_n 
\\
& \hspace{1cm}
+ \sum_{m=1}^{n-1} u_m\,y_{n-m} 
+ \sum_{m=1}^{n-1} \sum_{\lambda=1}^{m-1} (-1)^\lambda \sum_{\bfm\in \calC_m^{\lambda+1}}
\Big(\big((u_{m_1}\,y_{m_2})\cdots \big) y_{m_{\lambda+1}} \Big)\,y_{n-m} 
\end{align*}
where we set $u_n:=x_n-y_n$ and therefore we have 
\begin{align*}
\sum_{m=1}^{n-1} u_m\,y_{n-m} 
& = \sum_{\bfn\in \calC_n^2} u_{n_1}\,y_{n_2} . 
\end{align*}
Setting $\ell=\lambda+1$ in the last sum, we have $2\leq \ell\leq n-1$ 
and $\ell\leq m\leq n-1$ with 
$$
\bigcup_{m=\lambda}^{n-1} \calC_m^\ell \times \calC_{n-m}^1 = \calC_n^{\ell+1}, 
$$
therefore
\begin{align*}
\sum_{m=1}^{n-1} \sum_{\lambda=1}^{m-1} (-1)^\lambda \sum_{\bfm\in \calC_m^{\lambda+1}}
\Big(\big((u_{m_1}\,y_{m_2})\cdots \big) y_{m_{\lambda+1}} \Big)\,y_{n-m} 
& = 
\sum_{\ell=2}^{n-1} (-1)^\ell \sum_{\bfn\in \calC_n^{\ell+1}}
\Big(\big((u_{n_1}\,y_{n_2})\, y_{n_3}\big)\cdots \Big) y_{n_{\ell+1}}.
\end{align*}
Thus, we finally obtain 
\begin{align*}
(\delta_r\scoprod \Id)\Diex(x_n) & = x_n . 
\end{align*}

For the second identity, we rewrite the comultiplication as 
\begin{align*}
\Diex(x_n) &= x_n + y_n + \sum_{\bfn\in\calC_n^2} x_{n_1}\,y_{n_2} 
\end{align*}
and using the fact that 
$$
\calC_n^{\ell+1}= \bigcup_{m=1}^{n-1} \calC_m^1 \times \calC_{n-m}^\ell, 
$$
and setting $\lambda=\ell$, we rewrite the right codivision as  
\begin{align*}
\delta_r(x_n) 
&= u_n+ \sum_{m=1}^{n-1} \sum_{\lambda=1}^{n-m} (-1)^\lambda \sum_{\bfk\in\calC_{n-m}^\lambda} 
\Big(\big((u_m\,y_{k_1})\, y_{k_2}\big)\cdots \Big) y_{k_\lambda}. 
\end{align*}
We then have 
\begin{align*}
(\Diex \scoprod \Id)\delta_r(x_n) 
& = \Diex(x_n)-z_n 
+ \sum_{m=1}^{n-1} \sum_{\lambda=1}^{n-m} (-1)^\lambda \!\! \sum_{\bfk\in\calC_{n-m}^\lambda} \!\! 
\Big(\big((\Diex(x_{m})-z_m)\,z_{k_1}\big)\cdots \Big) z_{k_\lambda} 
\\ 
& = x_n + y_n + \sum_{\bfn\in\calC_n^2} x_{n_1}\,y_{n_2} 
+ \sum_{m=1}^{n-1} \sum_{\lambda=1}^{n-m} (-1)^\lambda \!\! \sum_{\bfk\in\calC_{n-m}^\lambda} \!\!
\big((x_{m}\,z_{k_1})\cdots \big) z_{k_\lambda} -z_n 
\\ 
& \hspace{1cm}
+ \sum_{m=1}^{n-1} \sum_{\lambda=1}^{n-m} (-1)^\lambda \!\! \sum_{\bfk\in\calC_{n-m}^\lambda} \!\!
\big((y_{m}\,z_{k_1})\cdots \big) z_{k_\lambda} 
\\ 
& \hspace{1cm}
+ \sum_{m=1}^{n-1} \sum_{\bfm\in\calC_m^2} 
\sum_{\lambda=1}^{n-m} (-1)^\lambda \!\! \sum_{\bfk\in\calC_{n-m}^\lambda} \!\!
\Big(\big((x_{m_1} y_{m_2})\,z_{k_1}\big)\cdots \Big) z_{k_\lambda} 
\\ 
& \hspace{1cm}
- \sum_{m=1}^{n-1} \sum_{\lambda=1}^{n-m} (-1)^\lambda \!\! \sum_{\bfk\in\calC_{n-m}^\lambda} \!\!
\big((z_{m}\,z_{k_1})\cdots \big) z_{k_\lambda} .
\end{align*}
When we then apply $\Id\scoprod\mu$, we identify $z_m=y_m$ and 
$z_{k_i}=y_{k_i}$ for $i=1,...,\lambda$, and therefore we have  
\begin{align*}
(\Id\scoprod\mu)(\Diex \scoprod \Id)\delta_r(x_n) 
& = x_n + \sum_{\bfn\in\calC_n^2} x_{n_1}\,y_{n_2} 
+ \sum_{m=1}^{n-1} \sum_{\lambda=1}^{n-m} (-1)^\lambda \!\! \sum_{\bfk\in\calC_{n-m}^\lambda} \!\!
\big((x_{m}\,y_{k_1})\cdots \big) y_{k_\lambda} 
\\ 
& \hspace{1cm}
+ \sum_{m=1}^{n-1} \sum_{\bfm\in\calC_m^2} 
\sum_{\lambda=1}^{n-m} (-1)^\lambda \!\! \sum_{\bfk\in\calC_{n-m}^\lambda} \!\!
\Big(\big((x_{m_1} y_{m_2})\,y_{k_1}\big)\cdots \Big) y_{k_\lambda} 
\end{align*}
where 
\begin{align*}
\sum_{m=1}^{n-1} \sum_{\lambda=1}^{n-m} (-1)^\lambda \!\! \sum_{\bfk\in\calC_{n-m}^\lambda} \!\!
\big((x_{m}\,y_{k_1})\cdots \big) y_{k_\lambda} 
& = \sum_{\ell=1}^{n-1} (-1)^\ell \sum_{\bfn\in \calC_n^{\ell+1}}
\big((x_{n_1}\,y_{n_2})\cdots \big) y_{n_{\ell+1}} , 
\end{align*}
and 
\begin{align*}
\sum_{\bfn\in\calC_n^2} x_{n_1}\,y_{n_2} + \sum_{m=1}^{n-1} \sum_{\bfm\in\calC_m^2} 
\sum_{\lambda=1}^{n-m} (-1)^\lambda \!\! \sum_{\bfk\in\calC_{n-m}^\lambda} \!\!
\Big(\big((x_{m_1} y_{m_2})\,y_{k_1}\big)\cdots \Big) y_{k_\lambda} 
\\ 
& \hspace{-8cm}
= - \sum_{\bfn\in\calC_n^{1+1}} (-1)^1 x_{n_1}\,y_{n_{1+1}} 
- \sum_{\ell=2}^{n-1} (-1)^\ell \sum_{\bfn\in \calC_n^{\ell+1}}
\big((x_{n_1}\,y_{n_2})\cdots \big) y_{n_{\ell+1}} . 
\end{align*}
This we finally have
\begin{align*}
(\Id\scoprod\mu)(\Diex \scoprod \Id)\delta_r(x_n) 
& = x_n . 
\end{align*}
\end{proof}


\subsection{Properties of the loop of invertible series}

Loops satisfying weak versions of associativity have many applications,
for instance in Blaschke's Web Geometry through nets \cite{BlaschkeBol}.
It is therefore interesting to ask what kind of identities are satisfied
by the loops $\Inv(A)$ of invertible series. 

\begin{proposition}
Given an algebra $A$, the loop $\Inv(A)$ satisfies an identity 
$$
(*)\qquad u(a,b,...,c) = v(a,b,...,c)
$$
for any series $a,b,...,c \in \Inv(A)$ if and only if the identity $(*)$ 
is satisfied in $A$, that is, for any elements $a,b,...,c\in A$.
\end{proposition}

\begin{proof}
Roughly speaking, this result follows from the fact that the comultiplication 
$\Diex$ is linear on both sides on generators. 
More precisely, if in the identity $(*)$ the operators $u$ and $v$ are 
multilinear, the implication 
``$(*)$ on $A$\quad$\Rightarrow$\quad $(*)$ on $\Inv(A)$'' 
is proved by direct inspection, and the opposite implication 
is proved by considering series of the form 
$a=1+a_1\,\lambda$, $b=1+b_1\,\lambda$,..., $c=1+c_1\,\lambda$. 

If in the identity $(*)$ the operators $u$ and $v$ are not multilinear, 
for instance the element $a$ appears $k$ times, it suffices to linearize
them, by considering the sum $a=a^1+\cdots +a^k$ of $k$ different elements. 
\end{proof}

In particular, this result implies that $\Inv(A)$ is a Moufang loop if and
only if $A$ is alternative, that $\Inv(A)$ is a group if and only if 
$A$ is associative, and that $\Inv(A)$ is an abelian group if and only if 
$A$ is commutative and associative.

We give below counterexamples to some interesting properties of loops which
fail on the loops $\Inv(A)$ for associative algebras $A$, which can be
deduced by the coloop bialgebra $\Hiex$. 

\begin{example}
\label{example-Inv-inverses}
The left and the right inverses of any $a\in \Inv(A)$ do not coincide,
that is
$$
a\backslash e \neq e \slash a. 
$$
In fact, the left and right inverses in $\Inv(A)$ coincide if and only if
the left antipode $S_l$ and the right antipode $S_r$ of $\Hiex$ coincide. 
Applying equations (\ref{antipode-left-right}), we find
\begin{align*}
S_r(x_n) & := (\varepsilon \scoprod \Id) \,  \delta_r(x_n) 
\\ 
& = -x_n - \sum_{\ell=1}^{n-1} (-1)^\ell \sum_{\bfn\in \calC_n^{\ell+1}}
\Big(\big((x_{n_1}\,x_{n_2})\, x_{n_3}\big)\cdots \Big) x_{n_{\ell+1}} 
\end{align*}
and 
\begin{align*}
S_l(x_n) & := (\Id \scoprod \varepsilon) \,  \delta_l(x_n)
\\ 
& = -x_n - \sum_{\ell=1}^{n-1} (-1)^\ell \sum_{\bfn\in \calC_n^{\ell+1}}
x_{n_1}\Big(\cdots \big(x_{n_2}\,(x_{n_\ell} x_{n_{\ell+1}})\big) \Big),
\end{align*}
therefore the two antipodes do not coincide. 
For instance, for a series $a=1+a_1\,\lambda$, we have 
\begin{align*}
e\slash a & := a\,  S_r 
= 1 - a_1\,\lambda + a_1 a_1\,\lambda^2 - (a_1 a_1)a_1\,\lambda^3 
+ \big((a_1 a_1)a_1\big)a_1\,\lambda^4 + \cdots 
\\ 
a\backslash e & := a\,  S_l 
= 1 - a_1\,\lambda + a_1 a_1\,\lambda^2 - a_1(a_1 a_1)\,\lambda^3 
+ a_1\big(a_1(a_1 a_1)\big)\,\lambda^4 + \cdots  
\end{align*}
To have a counterexample to the equality $e\slash a=a\backslash e$,
take for $A$ the algebra of $2\times 2$ matrices over the sedenions,
spanned by $1$ and by the imaginary units $e_i$ for $i=1,...,15$. 
If $a=1+a_1\,\lambda$ is the series with coefficient 
$$
a_1 = \left(\begin{array}{cc}
  e_1+e_{10} & e_5+e_{14} \\ 0 & 1
\end{array}\right),
$$
we have 
$$
e\slash a - a\backslash e 
= \left(\begin{array}{cc} 
0 & -2(e_5+e_{14}) \\ 0 & 0 \end{array}\right) \,\lambda^3 + O(\lambda^4). 
$$
\end{example}

\begin{example}
\label{example-Inv-divisions}
The left and right inversions in $\Inv(A)$ do not allow us to construct
the divisions, that is, 
$$
a \slash b \neq a\,  (e\slash b) \qquad\mbox{and}\qquad 
a\backslash b \neq (a\backslash e)\,  b 
$$
for any $a,b\in \Inv(A)$. 

In fact, to show that $a \slash b \neq a\, (e\slash b)$ and 
$a \backslash b \neq (a\backslash e)\,b$ in the loop $\Inv(A)$ is 
equivalent to show that 
$$
\delta_r \neq (\Id\scoprod S_r)\,  \Diex 
\qquad\mbox{and}\qquad 
\delta_l \neq (S_l\scoprod \Id)\,  \Diex 
$$
in the coloop bialgebra $\Hiex$. Let us show it for the right codivision. 
For any generator $x_n$, we have 
\begin{align*}
(\Id\scoprod S_r)\,  \Diex(x_n) 
& = x_n + S_r(y_n) + \sum_{m=1}^{n-1} x_m\,S_r(y_{n-m}) 
\\ 
& = x_n - y_n - \sum_{\ell=1}^{n-1} (-1)^\ell \sum_{\bfn\in \calC_n^{\ell+1}}
\big((y_{n_1}\,y_{n_2})\cdots \big) y_{n_{\ell+1}} 
\\ 
& \hspace{1cm}
- \sum_{m=1}^{n-1} x_m\,y_{n-m} 
- \sum_{m=1}^{n-1} \sum_{\lambda=1}^{n-m-1} (-1)^\lambda \sum_{\bfk\in \calC_{n-m}^{\lambda+1}}
x_m\,\Big(\big((y_{k_1}\,y_{k_2})\cdots \big) y_{k_{\lambda+1}}\Big). 
\end{align*}
Writing the last two sums in terms of compositions of $n$ yields 
\begin{align*}
(\Id\scoprod S_r)\,  \Diex(x_n) 
& = u_n + \sum_{\ell=1}^{n-1} (-1)^\ell \sum_{\bfn\in \calC_n^{\ell+1}}
\left( x_{n_1}\,\big(\big((y_{n_2}\,y_{n_3})\cdots \big) y_{n_{\ell+1}}\big)- 
\big((y_{n_1}\,y_{n_2})\cdots \big) y_{n_{\ell+1}} \right), 
\end{align*}
which is clearly different from the expression of $\delta_r(x_n)$. 
\end{example}

\begin{example}
\label{example-Inv-alternative}
A loop $Q$ is {\bf left alternative} if $a(ab)=(aa)b$ for any $a,b\in Q$,
and it is {\bf right alternative} if $(ab)b=a(bb)$. 
The proalgebraic loop $\Inv$ on the category $\Alg$ is not left nor right
alternative.

For this, it suffices to show that the coloop bialgebra $\Hiex$ is not
right coalternative, that is $(\Id\scoprod\mu)\,K\neq 0$, where 
$K=(\Diex\scoprod\Id)\,\Diex(\Id\scoprod\Diex)\,\Diex$ is the
coassociator.  
The first deviation from right alternativity appears on the generator
$x_3$, since we have
\begin{align*}
K(x_3) & = (x_1y_1)z_1-x_1(y_1z_1) \\ 
(\Id\scoprod\mu)\,K(x_3) &= (x_1y_1)y_1-x_1(y_1 y_1) \neq 0.
\end{align*}
For instance, if $A$ is the algebra of sedenions, the deviation from
right alternativity can be seen comparing $(a b) b$ and $a(b b)$
for the two series 
$$
a=1+(e_1+e_{10})\lambda \qquad\mbox{and}\qquad 
b=1+(e_5+e_{14})\lambda 
$$
because $(e_1+e_{10})(e_5+e_{14}) = 0$ and therefore 
$$
(a b) b - a(b b) 
= -(e_1+e_{10})(e_5+e_{14})^2\,\lambda^3 = 2(e_1+e_{10})\,\lambda^3. 
$$
\end{example}

\begin{example}
\label{example-Inv-power associative}
A loop $Q$ is {\bf power associative} if every element of the loop
generates an abelian subgroup. 
The proalgebraic loop $\Inv$ on the category $\Alg$ is not power
associative. In particular, power associativity requires the associativity
$(aa)a=a(aa)$ for any element. Therefore, it suffices to show that 
$$
\mu\, (\Id\scoprod\mu)\,K(x_3) = (x_1x_1)x_1-x_1(x_1x_1) \neq 0. 
$$

For instance, if we take $A$ to be the algebra of $2\times 2$ matrices 
with coefficients in the sedenion algebra, for the series 
$a=1+a_1\lambda$ of Example \ref{example-Inv-divisions} with 
$$
a_1 = \left(\begin{array}{cc} e_1+e_{10} & e_5+e_{14} \\ 0 & 1 \end{array}\right) 
$$
we have 
$$
a_1^2a_1 = \left(\begin{array}{cc} 
-2(e_1+e_{10}) & -(e_5+e_{14}) \\ 0 & 1 \end{array}\right) 
\qquad\mbox{and}\qquad 
a_1\,a_1^2= \left(\begin{array}{cc} 
-2(e_1+e_{10}) &  e_5+e_{14} \\ 0 & 1 \end{array}\right)  
$$
and therefore 
$$
(a a) a - a(aa) = \left(\begin{array}{cc} 
0 & -2(e_5+e_{14}) \\ 0 & 0 \end{array}\right) \,\lambda^3. 
$$
\end{example}


\section{Coloop of formal diffeomorphisms}
\label{section-diffeomorphisms}

The group of formal diffeomorphisms (tangent to the identity) 
is the set of series 
$$
\Diff(A)=\Big\{\ a=\sum_{n\geq 0} a_n\,\lambda^{n+1}\ |\ a_0=1,\ a_n\in A\ \Big\} 
$$ 
with coefficients $a_n$ taken in a commutative algebra $A$, 
endowed with the composition law $(a\circ b)(\lambda)=a\big( b(\lambda)\big)$, 
unit $e(\lambda)=\lambda$, and where the inverse of a series $a(\lambda)$ 
is given by the Lagrange inversion formula~\cite{Lagrange}. 
It is a proalgebraic group on $\Com$, represented by the Fa\`a di Bruno Hopf 
algebra~\cite{Doubilet}, \cite{JoniRota}
\begin{align*}
\HFdB & =\F[x_n,\ n\geq 1] \qquad (x_0=1) \\
\DFdB(x_n) &=\sum_{m=0}^n x_m \otimes \sum_{(p)} 
\dfrac{(m+1)!}{p_0! p_1!\cdots p_n!}\ x_1^{p_1}\cdots x_n^{p_n} 
\end{align*}
where the sum is done over the set of tuples $(p_0,p_1,p_2,...,p_n)$ 
of non-negative integers such that $p_0+p_1+p_2+\cdots +p_n=m+1$ and 
$p_1+2p_2+\cdots +np_n=n-m$. 
In this section we show that this group can be extended as a proalgebraic loop 
to the category $\As$. 


\subsection{Loop of formal diffeomorphisms}

\begin{definition}
\label{definition-Diff-loop}
Let $A$ be a unital associative algebra, non necessarily commutative, and 
let $\lambda$ be a formal variable. 
We call formal diffeomorphisms in $\lambda$ with coefficients in $A$ 
the formal series in the set 
$$
\Diff(A)=\Big\{\ a=\sum_{n\geq 0} a_n\,\lambda^{n+1}\ |\ 
a_0=1,\ a_n\in A\ \Big\}, 
$$ 
endowed with the composition law 
\begin{align*}
a\circ  b 
& = \sum_{n\geq 0}\ \sum_{m=0}^n\ 
\underset{k_0,...,k_m\geq 0}{\sum_{k_0+\cdots +k_m=n-m}} 
a_m\ b_{k_0}\cdots b_{k_m} \,\lambda^{n+1} \\ 
& = \sum_{n\geq 0}\ \left(a_n + b_n + \sum_{m=1}^{n-1}\ a_m\ \sum_{l=1}^n 
\binom{m+1}{l} \underset{k_1,...,k_m\geq 1}{\sum_{k_1+\cdots +k_m=n-m}} 
b_{k_1}\cdots b_{k_m}\right)\,\lambda^{n+1}
\end{align*}
and the unit $e$ given by $e_0=1$ and $e_n=0$ for all $n>1$. 
For instance, 
\begin{align*}
(a\circ b)_1 & = a_1+b_1, \\ 
(a\circ b)_2 & = a_2+2a_1b_1+b_2, \\
(a\circ b)_3 & = a_3+3a_2b_1+a_1(2b_2+b_1^2)+b_3.  
\end{align*}
The indeterminate $\lambda$ is not necessary to define the loop law, 
but helps to keep track of the degree of the terms in the sum. 
\end{definition}

\begin{proposition}
For any unital associative algebra $A$, the set $\Diff(A)$ is a loop.  
\end{proposition}

\begin{proof}
It is clear that the composition is a well-defined operation, and that $e$ 
is a unit. Let us show that the left and right divisions exist.  
\bigskip 

\noindent
i) Let us prove that there exists a right division $\slash$ 
satisfying the two equations (\ref{right-cancellation}). 
Given two series $a=\sum a_n\,\lambda^{n+1}$ and $b=\sum b_n\,\lambda^{n+1}$, 
let us define the series $a\slash b=\sum (a\slash b)_n\,\lambda^{n+1}$ 
so that $(a\slash b)\circ b = a$, that is
$$
\sum_{m=0}^n\ \underset{k_0,...,k_m\geq 0}{\sum_{k_0+\cdots +k_m=n-m}}
(a\slash b)_m\ b_{k_0}\cdots b_{k_m} = a_n  \qquad\mbox{for any $n\geq 0$}.  
$$ 
From now on, in the sum over the integers $k_0,...,k_p$ we omit to write 
that all integers can be zero. 
These equations are solved recursively, starting from $(a\slash b)_0=a_0=1$. 
The $n$th term is given by 
\begin{align*}
(a\slash b)_n & = a_n - \sum_{m=0}^{n-1}\ \sum_{k_0+\cdots +k_m=n-m}  
(a\slash b)_m\ b_{k_0}\cdots b_{k_m}. 
\end{align*}
To prove that $(a\circ b)\slash b = a$, i.e. that 
$\Big((a\circ b)\slash b\Big)_n = a_n$ for any $n\geq 0$, 
we proceed by induction. 
We have $\Big((a\circ b)\slash b\Big)_0=(a\circ b)_0=1$, therefore 
$$
\Big((a\circ b)\slash b\Big)_1 
= (a\circ b)_1 - \Big((a\circ b)\slash b\Big)_0 b_1 
= a_1+b_1-b_1 = a_1
$$
and 
\begin{align*}
\Big((a\circ b)\slash b\Big)_n & = (a\circ b)_n - \sum_{m=0}^{n-1}\ 
\sum_{k_0+\cdots +k_m=n-m}
\Big((a\circ b)\slash b\Big)_m\ b_{k_0}\cdots b_{k_m} \\ 
& = \sum_{m=0}^{n}\ \sum_{k_0+\cdots +k_m=n-m} 
a_p\ b_{k_0}\cdots b_{k_m} 
- \sum_{m=0}^{n-1}\ \sum_{k_0+\cdots +k_m=n-m} 
\Big((a\circ b)\slash b\Big)_m\ b_{k_0}\cdots b_{k_m} \\ 
& = a_n +  \sum_{m=0}^{n-1}\ \sum_{k_0+\cdots +k_m=n-m} 
\left(a_m-\Big((a\circ b)\slash b\Big)_m\right)\ b_{k_0}\cdots b_{k_m}, 
\end{align*}
so, if we suppose that $\Big((a\circ b)\slash b\Big)_m=a_m$ for any 
$m\leq n-1$, we have $\Big((a\circ b)\slash b\Big)_n=a_n$. 
\bigskip 

\noindent
ii) To prove the existence of the left division we proceed in the same way: 
the series $a\backslash b$ that satisfies the identity 
$a\circ (a\backslash b)=b$ of equations (\ref{left-cancellation}), that is, 
$$
\sum_{m=0}^n\ \sum_{k_0+\cdots +k_m=n-m} 
a_m\ (a\backslash b)_{k_0}\cdots (a\backslash b)_{k_m} = b_n 
\qquad\mbox{for any $n\geq 0$}, 
$$
is given recursively by $(a\backslash b)_0=1$ and 
\begin{align*}
(a\backslash b)_n & = b_n - \sum_{m=1}^n\ 
\sum_{k_0+\cdots +k_m=n-m} 
a_m\ (a\backslash b)_{k_0}\cdots (a\backslash b)_{k_m}. 
\end{align*}
The identity $a\backslash (a\circ b)=b$ means that, for any $n\geq 0$, we have 
$\Big(a\backslash (a\circ b)\Big)_n = b_n$. This is proved by induction. 
We have $\Big(a\backslash (a\circ b)\Big)_0=1$, therefore 
$$
\Big(a\backslash (a\circ b)\Big)_1 
= (a\circ b)_1 
- a_1 \Big(a\backslash (a\circ b)\Big)_0 \Big(a\backslash (a\circ b)\Big)_0 
= a_1+b_1-a_1 = b_1
$$
and 
\begin{align*}
\Big(a\backslash (a\circ b)\Big)_n 
& = (a\circ b)_n - \sum_{m=1}^n\ \sum_{k_0+\cdots +k_m=n-m} 
a_m \Big(a\backslash (a\circ b)\Big)_{k_0} \cdots 
\Big(a\backslash (a\circ b)\Big)_{k_m} \\ 
& = \sum_{m=0}^{n}\ \sum_{k_0+\cdots +k_m=n-m}
a_m\ b_{k_0}\cdots b_{k_m} 
- \sum_{m=1}^n\ \sum_{k_0+\cdots +k_m=n-m} 
a_m \Big(a\backslash (a\circ b)\Big)_{k_0} \cdots 
\Big(a\backslash (a\circ b)\Big)_{k_m} \\ 
& = a_0b_n +  \sum_{m=1}^n\ \sum_{k_0+\cdots +k_m=n-m} 
a_m\ \left(b_{k_0}\cdots b_{k_m} - \Big(a\backslash (a\circ b)\Big)_{k_0}\cdots  
\Big(a\backslash (a\circ b)\Big)_{k_m}\right) \\ 
& = b_n +  \sum_{m=1}^{n-1}\ \sum_{k_0+\cdots +k_m=n-m} 
a_m\ \left(b_{k_0}\cdots b_{k_m} - \Big(a\backslash (a\circ b)\Big)_{k_0}\cdots  
\Big(a\backslash (a\circ b)\Big)_{k_m}\right) \\ 
& \hspace{2cm} 
+ a_n\ \left(b_0\cdots b_0 - \Big(a\backslash (a\circ b)\Big)_0\cdots  
\Big(a\backslash (a\circ b)\Big)_0\right)\\ 
& = b_n +  \sum_{m=1}^{n-1}\ \sum_{k_0+\cdots +k_m=n-m} 
a_m\ \left(b_{k_0}\cdots b_{k_m} - \Big(a\backslash (a\circ b)\Big)_{k_0}\cdots  
\Big(a\backslash (a\circ b)\Big)_{k_m}\right),  
\end{align*}
so, if we suppose that $\Big(a\backslash (a\circ b)\Big)_m=b_m$ for any 
$m\leq n-1$, we have $\Big(a\backslash (a\circ b)\Big)_n=b_n$. 
\end{proof}

For instance, the first terms of the right division are
\begin{align*}
(a\slash b)_1 & = a_1-b_1, \\ 
(a\slash b)_2 & = a_2-\Big[b_2+2(a\slash b)_1 b_1\Big] 
= a_2-2a_1b_1-(b_2-2b_1^2), \\ 
(a\slash b)_3 & = a_3-\Big[b_3+(a\slash b)_1(2b_2+b_1^2)+3(a\slash b)_2b_1 \Big]\\
& = a_3-\big(2a_1b_2+3a_2b_1\big)+5a_1b_1^2-\big[b_3-(2b_1b_2+3b_2b_1)+5b_1^3\big],
\end{align*}
and the first terms of the left division are 
\begin{align*}
(a\backslash b)_1 & =b_1-a_1, \\ 
(a\backslash b)_2 & =b_2-\Big[2a_1 (a\backslash b)_1+a_2\Big] 
= b_2-2a_1b_1-(a_2-2a_1^2), \\ 
(a\backslash b)_3 & =b_3-\Big[a_1(2(a\backslash b)_2+(a\backslash b)_1^2)
+a_2(3(a\backslash b)_1)+a_3\Big] \\ 
& = b_3-\big(2a_1b_2+3a_2b_1\big)+\big(5a_1^2b_1+a_1b_1a_1-a_1b_1^2\big)
-\big[a_3-(2a_1a_2+3a_2a_1)+5a_1^3\big]. 
\end{align*}

We now prove that the loop of formal diffeomorphism is proalgebraic over 
associative algebras, and give its representative coloop bialgebra. 


\subsection{Fa\`a di Bruno coloop bialgebra} 

As in Section~\ref{subsection-invertible-coloop-bialgebra}, let
$X=\Span_\F\{x_n,\ n\geq 1\}$ be the set of graded variables $x_n$ of degree
$n$, and identify the tensor algebra $T(X)=\F\langle x_n,\ n\geq 1\rangle$
with the set of non-commutative polynomials $\F\langle x_n,\ n\geq 1\rangle$. 
We endow this algebra with the structure of a coloop biagebra which represents
the loop $\Diff$. 

As before, to simplify the notations, we denote by $x_n=x_n^{(1)}$ and 
$y_n=x_n^{(2)}$ the generators taken in the two copies of $X$ in 
$T(X)\scoprod T(X)\cong T(X^{(1)}\oplus X^{(2)})$. 

To describe the codivisions we need to introduce some sets of sequences 
and two types of related integer coefficients. 

\begin{definition}
\label{definition-d}
For any $\ell \geq 1$, let $\calM_\ell$ denote the set of 
sequences $\bfm = (m_1,...,m_\ell)$ such that  
\begin{align}
\label{sequence-M}
m_1+\cdots +m_\ell=\ell, \qquad\mbox{and}\qquad 
m_1+\cdots +m_j\geq j \quad\mbox{for}\quad j=1,...,\ell-1 . 
\end{align} 
For instance, for $\ell=1,2,3$, we have 
$$
\calM_1 = \big\{ (1) \big\}, \quad \calM_2 = \big\{ (2,0),(1,1) \big\}, \quad 
\calM_3 = \big\{ (3,0,0),(2,1,0),(2,0,1),(1,2,0),(1,1,1) \big\}. 
$$

For any $\ell\geq 1$ and any sequence $(n_1,...,n_{\ell+1})$ of positive 
integers, we call {\bf Lagrange coefficient}\footnote{
These coefficients appear in the Lagrange inversion formula~\cite{Lagrange}, 
cf.~\cite{BFK}. 
}
the number
$$ 
d_\ell(n_1,...,n_\ell) 
= \sum_{\bfm\in\calM_\ell} \binom{n_1+1}{m_1}\cdots 
\binom{n_\ell+1}{m_\ell}. 
$$
For $\ell=0$, $\calM_0$ is empty and we set $d_0 = 1$.

For instance, for $\ell=1,2,3$, we have
\begin{align*}
d_1(n_1) & = \binom{n_1+1}{1} ,
\\ 
d_2(n_1,n_2) & = \binom{n_1+1}{2} + \binom{n_1+1}{1} \binom{n_2+1}{1} ,
\\ 
d_3(n_1,n_2,n_3)
& = \binom{n_1+1}{3} 
+ \binom{n_1+1}{2} \binom{n_2+1}{1} + \binom{n_1+1}{2} \binom{n_3+1}{1} 
\\ 
& \hspace{1cm}
+ \binom{n_1+1}{1} \binom{n_2+1}{2} 
+ \binom{n_1+1}{1} \binom{n_2+1}{1} \binom{n_3+1}{1} .
\end{align*}
\end{definition}

\begin{definition}
\label{definition-de}
For any $\ell \geq 1$, let $\calE_\ell$ be the set of sequences
$\bfe = (e_1,...,e_\ell)$ of bits $e_i\in\{1,2\}$. 
For any $\bfe\in \calE_\ell$, let $\calM_\ell^{\bfe}$ be the set of sequences 
$\bfm = (m_1,...,m_\ell)\in \calM_\ell$ such that  
\begin{align*}
m_i=0 \quad\mbox{if}\quad e_i=2, \quad\mbox{for}\quad i=2,...,\ell .
\end{align*}
The bits $e_i$ will be used in Eq.~(\ref{labeled-variables}) to label
the generators $x_i$ of the coloop bialgebra in order to determine to which
copy of the coproduct algebra the variables $x_i^{(e_i)}$ belong.
To simplify the final formulas for the codivisions, we chose for $e_i$
the bits $1$ and $2$, even if, for the present discussion, the bits $1$ and $0$
would be more appropriate. 

In particular, if $\bfe = (1,1,...,1)$ then $\calM_\ell^{\bfe} = \calM_\ell$.
If $\bfe$ starts with the bit $e_1=2$, then $\calM_\ell^{\bfe}$ is empty,
because the condition (\ref{sequence-M}) implies that $m_1\geq 1$. 
If $\bfe$ starts with the bit $e_1=1$ and contains at least a bit value $2$,
then the set $\calM_\ell^{\bfe}$ is a proper subset of $\calM_\ell$ obtained
by keeping only those sequences $\bfm$ which have the value $0$ in all
the positions where the bit value of $\bfe$ is $2$. 
For instance, $\calM_1^{(1)} = \calM_1 = \{ (1) \}$ and for $\ell=2$ we have
\begin{align}
\nonumber
\calM_2^{(1,1)} & = \calM_2 = \{ (2,0),(1,1) \}, 
\\ \label{labeled-sequence M_2}
\calM_2^{(1,2)} & = \{ (2,0) \}. 
\end{align}
For $\ell=3$, we have 
\begin{align}
\nonumber
\calM_3^{(1,1,1)} & = \calM_3 
= \big\{ (3,0,0),(2,1,0),(2,0,1),(1,2,0),(1,1,1) \big\}  
\\ \nonumber
\calM_3^{(1,1,2)} & = \{ (3,0,0),(2,1,0),(1,2,0) \}  
\\ \label{labeled-sequence M_3}
\calM_3^{(1,2,1)} & = \{ (3,0,0),(2,0,1) \}  
\\ \nonumber
\calM_3^{(1,2,2)} & = \{ (3,0,0) \} . 
\end{align}

For any $\ell\geq 1$, any sequence $\bfe\in \calE_\ell$ 
and any sequence $(n_1,...,n_{\ell+1})$ of positive integers, we call 
{\bf labeled Lagrange coefficient} the number 
\begin{align*}
d_\ell^{\bfe}(n_1,...,n_\ell) & = \sum_{\bfm \in \calM_\ell^{\bfe}}
\binom{n_1+1}{m_1}\cdots \binom{n_\ell+1}{m_\ell}.
\end{align*}
For $\ell=0$, $\calM_0$ and $\calE_0$ are empty and we set $d_0^{\bfe}=d_0 1$. 
Of course, if $\bfe = (1,1,...,1)$ then 
$d_\ell^{\bfe}(n_1,...,n_\ell) = d_\ell(n_1,...,n_\ell)$,
if $\bfe$ starts by $2$ then $d_\ell^{\bfe}=0$, 
and if $\bfe$ starts by $1$ and contains some bit values equal to $2$, then 
$d_\ell^{\bfe}(n_1,...,n_\ell) < d_\ell(n_1,...,n_\ell)$. 

Here are the values of $d_\ell^{\bfe}$ for $\ell=1,2,3$, obtained by summing
the patterns according to the labeled sequences given in
(\ref{labeled-sequence M_2}) and (\ref{labeled-sequence M_3}): 
\begin{align*}
d_1^{(1)}(n_1) & = d_1(n_1) = \binom{n_1+1}{1} , 
\\ 
d_2^{(1,1)}(n_1,n_2) & = d_2(n_1,n_2) 
= \binom{n_1+1}{2} + \binom{n_1+1}{1} \binom{n_2+1}{1},
\\ 
d_2^{(1,2)}(n_1,n_2) & = \binom{n_1+1}{2} , 
\end{align*}
where the term $\binom{n_1+1}{1} \binom{n_2+1}{1}$ disappears because it 
corresponds to the sequence $\bfm = (1,1)$ which does not have the value $0$
in the same position as the bit $2$ in $\bfe = (1,2)$, 
\begin{align*}
d_3^{(1,1,1)}(n_1,n_2,n_3) & = d_3(n_1,n_2,n_3) 
\\ 
& = \binom{n_1+1}{3} 
+ \binom{n_1+1}{2} \binom{n_2+1}{1} + \binom{n_1+1}{2} \binom{n_3+1}{1} 
\\ 
& \hspace{1cm}
+ \binom{n_1+1}{1} \binom{n_2+1}{2} 
+ \binom{n_1+1}{1} \binom{n_2+1}{1} \binom{n_3+1}{1} ,
\\ 
d_3^{(1,1,2)}(n_1,n_2,n_3) 
& = \binom{n_1+1}{3} 
+ \binom{n_1+1}{2} \binom{n_2+1}{1} 
+ \binom{n_1+1}{1} \binom{n_2+1}{2} , 
\\ 
d_3^{(1,2,1)}(n_1,n_2,n_3) 
& = \binom{n_1+1}{3} + \binom{n_1+1}{2} \binom{n_3+1}{1} ,
\\
d_3^{(1,2,2)}(n_1,n_2,n_3) 
& = \binom{n_1+1}{3} . 
\end{align*}
\end{definition}

\begin{definition}
\label{FaaDiBrunoColoopBialgebra}
We call {\bf Fa\`a di Bruno coloop bialgebra} the free unital associative 
algebra 
\begin{align*}
\HFdBex & =\F\langle x_n,\ n\geq 1\rangle, \qquad x_0=1 
\end{align*} 
of non-commutative polynomials in the graded variables $x_n$, 
with the following graded co-operations: 
\begin{itemize}
\item 
comultiplication $\DFdBex:\HFdBex\longrightarrow \HFdBex \scoprod \HFdBex$ 
given by 
\begin{align*}
\DFdBex(x_n) 
& = x_n + y_n + \sum_{\ell=1}^{n-1} \sum_{\bfn\in \calC_n^{\ell+1}} \binom{n_1+1}{\ell}\ 
x_{n_1} \ y_{n_2}\cdots y_{n_{\ell+1}} , 
\end{align*} 
where $\calC_n^{\ell+1}$ is the set of compositions of $n$ of length $\ell+1$, 
cf.~(\ref{sequence-C}); 

\item 
counit $\varepsilon:\HFdBex\longrightarrow \F$ given by 
$\varepsilon(x_n) = \delta_{n,0}$; 

\item
right codivision 
$\delta_r:\HFdBex\longrightarrow \HFdBex \scoprod \HFdBex$ given by
\begin{align*}
\delta_r(x_n) 
& = \sum_{\ell=0}^{n-1} (-1)^\ell \sum_{\bfn \in \calC_n^{\ell+1}} 
d_\ell(n_1,...,n_\ell)\ 
(x_{n_1}-y_{n_1})\, y_{n_2}\cdots y_{n_{\ell+1}},  
\end{align*}
where the Lagrange coefficients $d_\ell$ are given in Def.~\ref{definition-d};

\item 
left codivision 
$\delta_l:\HFdBex\longrightarrow \HFdBex \scoprod \HFdBex$ given by
\begin{align*}
\delta_l(x_n) 
& = 
\sum_{\ell=0}^{n-1} (-1)^\ell \!\sum_{\bfn\in \calC_{n}^{\ell+1}} 
\sum_{\bfe\in \calE_\ell} (-1)^{\bfe}\ d_\ell^{\bfe}(n_1,...,n_\ell)\ 
x^{(e_1)}_{n_1} x^{(e_2)}_{n_2}\cdots x^{(e_\ell)}_{n_\ell}\, (y_{n_{\ell+1}}-x_{n_{\ell+1}}) 
\end{align*}
where the set of sequences $\calE_\ell$ and the labeled Lagrange coefficients 
$d_\ell^{\bfe}$ are given in Def.~\ref{definition-de}, and where we set 
$$
(-1)^{\bfe} = (-1)^{e_1+\cdots +e_\ell-\ell}
$$
and, according to the previous convention, we set 
\begin{align}
\label{labeled-variables}
x^{(e_i)}_n & = \left\{ \begin{array}{ll} 
x_n & \mbox{if $e_i=1$} \\ 
y_n & \mbox{if $e_i=2$}
\end{array}\right. . 
\end{align}
In particular, since $d_\ell^{\bfe}=0$ if $e_1=2$, the first variable is always
$x^{(e_1)}_{n_1}=x_{n_1}$.
\end{itemize}
\end{definition} 

For instance, on the first five generators, the comultiplication is 
\begin{align*}
\DFdBex(x_1) &= x_1 + y_1 
\\ 
\DFdBex(x_2) &= x_2 + y_2 + 2x_1y_1 
\\ 
\DFdBex(x_3) &= x_3 + y_3 + \big( 2x_1y_2 + 3x_2y_1 \big) + x_1y_1^2
\\ 
\DFdBex(x_4) &= x_4 + y_4 + \big( 2x_1y_3 + 3x_2y_2 + 4x_3y_1\big) 
+ \big( x_1(y_1y_2+y_2y_1) + 3x_2y_1^2 \big) 
\\ 
\DFdBex(x_5) 
&= x_5 + y_5 + \big( 2x_1y_4 + 3x_2y_3 + 4x_3y_2 + 5x_4y_1 \big) 
\\ 
& \qquad\quad 
+ \big( x_1(y_1y_3+y_2^2+y_3y_1) + 3x_2(y_1y_2+y_2y_1) +6x_3y_1^2 \big) 
+ x_2y_1^3 
\end{align*}
the right codivision, with $u_n=x_n-y_n$, is 
\begin{align*}
\delta_r(x_1) &= u_1 
\\ 
\delta_r(x_2) &= u_2 - 2u_1y_1 
\\ 
\delta_r(x_3) &= u_3 - \big( 2u_1y_2 + 3u_2y_1 \big) + 5u_1y_1^2 
\\ 
\delta_r(x_4) &= u_4 - \big( 2u_1y_3 + 3u_2y_2 + 4u_3y_1 \big) 
+ \big( 5u_1y_1y_2 + 7u_1y_2y_1 + 9u_2y_1^2 \big) - 14u_1y_1^3 
\\ 
\delta_r(x_5) &= u_5 - \big( 2u_1y_4 + 3u_2y_3 + 4u_3y_2 + 5u_4y_1 \big) 
\\ 
& \qquad\qquad 
+ \big( 5u_1y_1y_3 + 7u_1y_2^2 + 9u_1y_3y_1 + 9u_2y_1y_2 + 12u_2y_2y_1 
+ 14u_3y_1^2 \big) 
\\ 
& \qquad\qquad 
- \big( 14u_1y_1^2y_2 + 19u_1y_1y_2y_1 + 23u_1y_2y_1^2 + 28u_2y_1^3 \big)   
+ 42u_1y_1^4 
\\ 
\end{align*}
and the left codivision has additional terms which contain both 
variables $x$ and $y$ in alternative order beside the first position 
which is always $x$, and last position which is always $v_n=y_n-x_n$:
\begin{align*}
\delta_l(x_1) &= v_1 
\\ 
\delta_l(x_2) &= v_2 - 2x_1v_1  
\\ 
\delta_l(x_3) &= v_3 - \big( 2x_1v_2 + 3x_2v_1 \big) + 5x_1^2v_1 - x_1y_1v_1 
\\ 
\delta_l(x_4) &= v_4 - \big( 2x_1v_3 + 3x_2v_2 + 4x_3v_1 \big) 
+ \big( 5x_1^2v_2 + 7x_1x_2v_1 + 9x_2x_1v_1 \big) - 14x_1^3v_1 
\\ 
& \qquad 
-\big( x_1y_1v_2 + x_1y_2v_1 + 3x_2y_1v_1 \big) 
+ \big( 4x_1^2y_1v_1 + 2x_1y_1x_1v_1\big) 
\\ 
\delta_l(x_5) &= v_5 - \big( 2x_1v_4 + 3x_2v_3 + 4x_3v_2 + 5x_4v_1 \big) 
\\ 
& \qquad\qquad 
+ \big( 5x_1^2v_3 + 7x_1x_2v_2 + 9x_1x_3v_1 + 9x_2x_1v_2 + 12x_2^2v_1 
+ 14x_3x_1v_1 \big) 
\\ 
& \qquad\qquad 
- \big( 14x_1^3v_2 + 19x_1^2x_2v_1 + 23x_1x_2x_1v_1 + 28x_2x_1^2v_1 \big)   
+ 42x_1^4v_1 
\\ 
& \qquad 
-\big( x_1y_1v_3 + x_1y_2v_2 + x_1y_3v_1 + 3x_2y_1v_2 + 3x_2y_2v_1 
+ 6x_3y_1v_1 \big) 
\\ 
& \qquad\qquad\quad 
+ \big( 4x_1^2y_1v_2 + 4x_1^2y_2v_1 + 9x_1x_2y_1v_1 + 10x_2x_1y_1v_1 
\\ 
& \qquad\qquad\qquad\qquad  
+ 2x_1y_1x_1v_2 + 3x_1y_1x_2v_1 + 2x_1y_2x_1v_1 
+ 7x_2y_1x_1v_1 - x_2y_1^2v_1 \big) 
\\ 
& \qquad\qquad\quad 
- \big( 14x_1^3y_1v_1 + 9x_1^2y_1x_1v_1 + 5x_1y_1x_1^2v_1 
- x_1^2y_1^2v_1 - x_1y_1x_1y_1v_1 \big) . 
\end{align*}
\bigskip 

We now want to prove that the algebra given above is indeed a coloop bialgebra. 
The only difficulty is to prove that the codivisions satisfy the 
cocancellation properties (\ref{left-cocancellation}) and 
(\ref{right-cocancellation}), which are equivalent to some recurrence 
relations on the Lagrange coefficients $d_\ell$ and $d_\ell^{\bfe}$. 

We prove in fact a stronger result, namely, that there exist some 
operators $R_\ell$ and $R_\ell^{\bfe}$ defined on the tensor space $T(A)$ 
over any positively graded algebra $A$, which produce the Lagrange coefficients 
and which satisfy the wished recurrence relations. 
These operators provide an alternative definition of the Fa\`a di Bruno 
codivisions when applied to the non-unital associative coproduct 
algebra $A=\overline{\HFdBex}\scoprod \overline{\HFdBex} 
= \overline{T}(X^{(1)}\oplus X^{(2)})$. 


\subsection{Fa\`a di Bruno co-operations in terms of recursive operators} 

Let $A=\oplus_{n\geq 1} A_n$ be a positively graded associative algebra over 
a field $\F$, and let us denote by $|a|$ the degree of an element $a\in A$, 
that is, the integer $n$ such that $a\in A_n$. 
The tensor algebra $T(A) = \bigoplus_{\ell\geq 0} A^{\sotimes \ell}$ is then bigraded, 
on one side by the tensor power $\ell$, that we call {\bf length}, 
and on the other side by the grading induced by that of $A$, that we call 
{\bf degree}, 
$$
|a_1\sotimes a_2\sotimes \cdots \sotimes a_\ell| = \sum_{i=1}^{\ell} |a_i|. 
$$
A multi-monomial is a homogeneous element of $T(A)$ with respect to the length,
that is, an element of the form $a_1\sotimes a_2\sotimes \cdots \sotimes a_\ell$
for some $\ell\geq 1$. 
Then $T(A)$ can be decomposed into the following direct sum with respect to 
the degree\footnote{
Note that if $A$ had a null degree component $A_0$, then $T(A)$ would contain 
an infinite sum of terms in each degree, namely
$T(A)_0 = \bigoplus_{p\geq 0} A_0^{\otimes p}$ and 
$T(A)_n = \bigoplus_{\ell=1}^n\ \bigoplus_{\bfn\in\calC_n^\ell}\ 
\bigoplus_{p\geq 0} \binom{\ell+p}{\ell}\ A_0^{\otimes p} \sotimes 
A_{n_1}\sotimes\cdots\sotimes A_{n_\ell}$
for $n\geq 1$. 
}: 
\begin{align*}
T(A) & = \F \oplus 
\bigoplus_{n\geq 1}\ \left( \bigoplus_{\ell=1}^n\ \bigoplus_{\bfn\in\calC_n^\ell}\ 
A_{n_1}\sotimes\cdots\sotimes A_{n_\ell} \right) , 
\end{align*}
where the compositions $\bfn\in \calC_n^\ell$ are defined by 
eq.~(\ref{sequence-C}). 

\begin{definition}
\label{definition-rhd}
Let us define a graded linear operation 
$$
\rhd : T(A) \sotimes T(A) \longrightarrow \F\oplus A
$$
by setting 
\begin{align*} 
a \rhd (b_1\sotimes \cdots \sotimes b_\ell)
& = \binom{|a|+1}{\ell}\, a\, b_1\cdots b_\ell
\\
(a_1\sotimes\cdots\sotimes a_{\ell'}) \rhd (b_1\sotimes\cdots\sotimes b_{\ell''}) 
& = a_1 \rhd (a_2\sotimes\cdots\sotimes a_{\ell'} \sotimes
b_1\sotimes\cdots\sotimes b_{\ell''}) 
\\ 
& = \binom{|a_1|+1}{\ell'\!+\!\ell''\!-\!1}\, 
a_1\cdots a_{\ell'}\, b_1\cdots b_{\ell''}
\end{align*}
where the expressions on the right-hand side mean the product in the algebra 
$A$. 

In particular, if we apply these rules to $1\in \F = A^{\sotimes 0}$, we have 
\begin{align*} 
1 \rhd 1 & = 1 
\\ 
1 \rhd b & = b 
\\
1 \rhd (b_1\sotimes \cdots \sotimes b_\ell)
& = 0 \qquad\mbox{if $\ell >1$} 
\\
a \rhd 1 & = a 
\\
(a_1\sotimes\cdots\sotimes a_\ell) \rhd 1
& = \binom{|a_1|+1}{\ell\!-\!1}\, a_1\cdots a_\ell . 
\end{align*}
\end{definition}

\begin{remark}
The restriction $\rhd:A\otimes T(A) \longrightarrow A$ is a brace product 
on $A$ which is symmetric if $A$ is commutative and generalises the natural 
pre-Lie product of the Lie subalgebra of strictly positive generators in 
the Witt algebra (cf.~\cite{BFM,FrabettiManchon}). 
Note however that $\rhd$ on $T(A) \sotimes T(A)$ is not a multibrace product 
(cf.~\cite{LodayRonco}), even excluding the scalar component, because the first 
non-trivial multibrace identity 
$M_{21}(a\sotimes b+b\sotimes a;c)+M_{11}(M_{11}(a;b);c) = 
M_{11}(a;M_{11}(b;c))+M_{12}(a;b\sotimes c+c\sotimes b)$ is not satisfied. 
Moreover, a unit for $\rhd$ can not exist, because of length arguments, 
and $\rhd$ is not associative, since for any $a,b,c\in A$ we have 
$$
(a\rhd b) \rhd c - a\rhd (b\rhd c) = (|a|\!+\!1) |a|\, abc \neq 0. 
$$
The algebraic structure described by the operator $\rhd$ in terms of generators 
and relations is an open question. 
\end{remark}

\begin{definition}
\label{definition-L}
We call {\bf left recursive operator} $L:T(A)\longrightarrow T(A)$ 
the collection $L=\{ L_\ell,\ \ell\geq 0 \}$ of (non homogeneous) linear
operators $L_0=\Id :\F\longrightarrow \F$ and 
\begin{align*}
& L_\ell: 
A^{\sotimes \ell} \longrightarrow \bigoplus_{\lambda=1}^{\ell} A^{\sotimes \lambda}, 
\qquad \ell \geq 1
\end{align*}
defined recursively by 
\begin{align*}
L_\ell(a_1,...,a_\ell) &= \sum_{i=0}^{\ell-1} (-1)^{\ell-1-i}
\Big( L_i(a_1,...,a_i)\rhd a_{i+1} \Big) \sotimes 
a_{i+2} \sotimes \cdots \sotimes a_\ell ,  
\end{align*}
where we denote 
$L_\ell(a_1,...,a_\ell):= L_\ell(a_1\sotimes\cdots \sotimes a_\ell)$
and $L_0$ is understood as acting on $1$. 
\end{definition}

The first left operators give
\begin{align*}
L_1(a) & = L_0(1)\rhd a =1 \rhd a = a, 
\\
L_2(a,b) & = -( L_0(1)\rhd a) \sotimes b + L_1(a)\rhd b = 
- a\sotimes b + a\rhd b 
\\ 
& = - a \sotimes b + \binom{|a|\!+\!1}{1}\, ab , 
\\
L_3(a,b,c) & = (L_0(1)\rhd a)\sotimes b\sotimes c - (L_1(a)\rhd b) \sotimes c
+ L_2(a,b) \rhd c 
\\ 
& = a \sotimes b \sotimes c - (a\rhd b) \sotimes c 
- (a\sotimes b) \rhd c + (a\rhd b) \rhd c 
\\
& = 
a\sotimes b\sotimes c 
- \binom{|a|\!+\!1}{1}\, ab\sotimes c 
+ \Big(\binom{|a|\!+\!1}{1}\binom{|a|\!+\!|b|\!+\!1}{1}
- \binom{|a|\!+\!1}{2}\Big)\, abc 
. 
\end{align*}
\bigskip 

The left operators $L_\ell$ can be easily described in a closed way. 

\begin{lemma}
\label{lemma-L1}
For any $\ell \geq 2$ and any $a_1,...,a_\ell\in A$ we have
\begin{align*}
L_\ell(a_1,...,a_\ell) 
&= L_{\ell-1}(a_1,...,a_{\ell-1}) \rhd a_\ell 
- L_{\ell-1}(a_1,...,a_{\ell-1}) \sotimes a_\ell. 
\end{align*}

As a consequence, $L_\ell(a_1,...,a_\ell)$ is the sum of the $2^{\ell-1}$ possible 
multi-monomials obtained by combining the operations $\rhd$ and $\sotimes$ 
with fixed parenthesizing on the left, namely
\begin{align*}
L_\ell(a_1,...,a_\ell) 
&= \sum_{\sigma_1,...,\sigma_{\ell-1} \in \{0,1\}} (-1)^{\sigma_1+\cdots +\sigma_{\ell-1}} 
\big(\cdots \big( (a_1 \ast_{\sigma_1} a_2) \ast_{\sigma_2} a_3 \big) 
\ast_{\sigma_3} \cdots \ast_{\sigma_{\ell-2}} a_{\ell-1} \big) \ast_{\sigma_{\ell-1}} a_l
\end{align*}
where we set 
$$
\ast_{\sigma} = \left\{ \begin{array}{ll} 
\rhd & \quad\mbox{if $\sigma=0$} \\ 
\sotimes & \quad\mbox{if $\sigma=1$}
\end{array} \right. . 
$$
\end{lemma}

\begin{proof}
By induction on $\ell$. For $\ell=2$, we have
\begin{align*}
L_1(a_1) \rhd a_2 - L_1(a_1) \sotimes a_2 
& = a_1\rhd a_2 - a_1 \sotimes a_2 = L_2(a_1,a_2). 
\end{align*}
Now suppose that for any $i=2,...,\ell-1$ we have 
$$
L_i(a_1,...,a_i) 
= L_{i-1}(a_1,...,a_{i-1}) \rhd a_i - L_{i-1}(a_1,...,a_{i-1}) \sotimes a_i.
$$
Let us expand the sum defining $L_\ell(a_1,...,a_\ell)$. 
At each step, we separate the first two terms of the sum over $i=0,...,\ell-1$:
\begin{align*}
L_\ell(a_1,...,a_\ell) 
&= \sum_{i=0}^{\ell-1} (-1)^{\ell-1-i}
\Big( L_i(a_1,...,a_i)\rhd a_{i+1} \Big) \sotimes 
a_{i+2} \sotimes \cdots \sotimes a_\ell 
\\ 
&= (-1)^{\ell-1} a_1\sotimes a_2\sotimes \cdots \sotimes a_\ell 
+ (-1)^{\ell-2} (a_1\rhd a_2) \sotimes \cdots \sotimes a_\ell 
\\ 
& \hspace{2cm}
+ \sum_{i=2}^{\ell-1} (-1)^{\ell-1-i}
\Big( L_i(a_1,...,a_i)\rhd a_{i+1} \Big) \sotimes 
a_{i+2} \sotimes \cdots \sotimes a_\ell 
\\ 
&= (-1)^{\ell-2} L_2(a_1,a_2)\sotimes a_3\sotimes \cdots \sotimes a_\ell 
+ (-1)^{\ell-3} \Big( L_2(a_1,a_2)\rhd a_3\Big)\sotimes \cdots \sotimes a_\ell
\\ 
& \hspace{2cm}
+ \sum_{i=3}^{\ell-1} (-1)^{\ell-1-i}
\Big( L_i(a_1,...,a_i)\rhd a_{i+1} \Big) \sotimes 
a_{i+2} \sotimes \cdots \sotimes a_\ell = \ldots
\end{align*}
Iterating this expansion we obtain 
\begin{align*}
L_\ell(a_1,...,a_\ell) 
&= (-1)^{\ell-(\ell-1)} L_{\ell-1}(a_1,...,a_{\ell-1}) \sotimes a_\ell 
+ (-1)^{\ell-\ell} L_{\ell-1}(a_1,...,a_{\ell-1}) \rhd a_\ell 
\\ 
&= L_{\ell-1}(a_1,...,a_{\ell-1}) \rhd a_\ell 
- L_{\ell-1}(a_1,...,a_{\ell-1}) \sotimes a_\ell . 
\end{align*}
\end{proof}

\begin{definition}
\label{definition-R}
We call {\bf right recursive operator} $R:T(A)\longrightarrow T(A)$ 
the collection $R=\{ R_\ell,\ \ell\geq 0 \}$ of (non homogeneous) linear
operators $R_0=\Id :\F\longrightarrow \F$ and 
\begin{align*}
& R_\ell: 
A^{\sotimes \ell} \longrightarrow \bigoplus_{\lambda=1}^{\ell} A^{\sotimes \lambda}, 
\qquad \ell \geq 1
\end{align*}
defined recursively by 
\begin{align}
\label{definition-R-equation}
R_\ell(a_1,...,a_\ell) &= \sum_{j=1}^\ell \sum_{\bfp\in \calC_\ell^j} 
\big( a_1 \rhd R_{p_1-1}(a_2,...,a_{p_1}) \big) 
\\ \nonumber
& \hspace{2cm} 
\sotimes \big( a_{p_1+1} \rhd R_{p_2-1}(a_{p_1+2},...,a_{p_1+p_2}) \big) 
\sotimes \cdots 
\\ \nonumber
& \hspace{3cm} 
\sotimes \big( a_{p_1+\cdots +p_{j-1}+1} \rhd 
R_{p_2-1}(a_{p_1+\cdots +p_{j-1}+2},...,a_{p_1+\cdots +p_j}) \big) ,  
\end{align}
where we denote 
$R_\ell(a_1,...,a_\ell):= R_\ell(a_1\sotimes\cdots \sotimes a_\ell)$
and where $R_0$ is understood as acting on $1$. 
\end{definition}

For instance, the first right operators are
\begin{align*}
R_1(a) & = a \rhd R_0(1) = a , 
\\
R_2(a,b) & = a\rhd R_1(b) + (a\rhd R_0(1)) \sotimes (b\rhd R_0(1)) \\ 
& = a\rhd b + a \sotimes b  
= \binom{|a|\!+\!1}{1}\, ab + a \sotimes b , 
\\
R_3(a,b,c) & = a\rhd R_2(b,c) + (a\rhd R_1(b)) \sotimes c
+ a \sotimes (b \rhd R_1(c)) + a \sotimes b\sotimes c \\ 
& = a\rhd (b\rhd c) 
+ a\rhd (b\sotimes c) 
+ a \sotimes (b\rhd c)
+ (a\rhd b)\sotimes c 
+ a\sotimes b\sotimes c 
\\
& = 
\left(\binom{|a|\!+\!1}{1}\binom{|b|\!+\!1}{1} + \binom{|a|\!+\!1}{2}\right)\, abc
+ \binom{|b|\!+\!1}{1}\, a \sotimes bc \\
& \hspace{2cm}
+ \binom{|a|\!+\!1}{1}\, ab\sotimes c 
+ a\sotimes b\sotimes c . 
\end{align*}

Note that the right recursive operator is not just a flip of the left recursive
one, basically because the recursion defining the two operators takes place
on the left and on the right-hand side of $\rhd$, which is not a symmetric
operation. The precise relationship between $R_\ell$ and $L_\ell$ is given in
Cor.~\ref{corollary-LR}, after some preliminary results. 

The right operators $R_\ell$ can also be described in a closed way. 

\begin{definition}
\label{definition-R_m}
Let $\calM_\ell$ be the set of sequences satisfying (\ref{sequence-M}). 
For any $\bfm\in \calM_\ell$, we define a length-homogeneous linear operator 
$R^\ell_{\bfm}:A^{\sotimes \ell} \longrightarrow T(A)$ which nests the operation
$\rhd$ in a multi-monomial $a_1\sotimes \cdots \sotimes a_\ell$ according to
the sequence $\bfm=(m_1,...,m_\ell)\in \calM_\ell$. The idea is the following: 
\begin{itemize}
\item
  The multi-monomial $R^\ell_{\bfm}(a_1,...,a_\ell)$ is constructed by nesting
  tensor monomials of the form $a_i\rhd Q_{i+1}(a_{i+1},a_{i+2},...)$
  one into the other one, where $Q_{i+1}(a_{i+1},a_{i+2},...)$ is a multi-monomial
  whose tensor factors can be single variables or monomials of the same form
  $a_j\rhd Q'_{j+1}(a_{j+1},a_{j+2},...)$. 
\item
  Every tensor monomial $a_i\rhd Q_{i+1}(a_{i+1},a_{i+2},...)$ is   determined
  by the length of the multi-monomial $Q_{i+1}(a_{i+1},a_{i+2},...)$ and that
  of the nested monomials of the same form. 
  The sequence $\bfm$ fixes the lengths of all the nested multi-monomials:
  \begin{itemize}
  \item
    The coefficient $m_1$ is the overall length of the multi-monomial
    $R^\ell_{\bfm}(a_1,...,a_\ell)$ in the tensor algebra $T(A)$, that is,
    we have $R^\ell_{\bfm}(a_1,...,a_\ell)\in A^{\sotimes m_1}$.
  \item
    For $i=2,...,\ell-1$, the coefficient $m_i$ is the length of the
    multi-monomial $Q_i(a_{i},a_{i+1},...)$ on which $a_{i-1}$ acts by $\rhd$:
    if $m_i=0$, then $a_{i-1}$ appears as an insolated tensor factor,
    if $m_i\neq 0$, then $a_{i-1}$ acts by $\rhd$ on a multimonomial of
    length $m_i$, which is determined by the values $m_j$ for $j>i$. 
  \end{itemize}
\end{itemize}
And now we give the algorithm to construct $R^\ell_{\bfm}(a_1,...,a_\ell)$: 
\begin{enumerate}
\item[(1)]
The coefficient $m_1$ tells us how many tensor factors we have to construct. 
\item[(2)] 
  Start with $a_1$ and read the coefficient $m_2$: if $m_2=0$ write
  $a_1 \otimes a_2 \cdots$, if $m_2\neq 0$ write $a_1\rhd (a_2\cdots)$
  and expect to close the parenthesis after a multi-monomial of length $m_2$. 
\item[(3)]
  Then read the next coefficient of $\bfm$ and repeat the procedure of (2).
  For any $i=2,...,\ell$, if $m_i=0$ write $a_{i-1} \otimes a_i \cdots$,
  if $m_i\neq 0$ write $a_{i-1}\rhd (a_i \cdots)$ and expect to close this
  parenthesis after a multi-monomial of length $m_i$. 
\item[(4)]
  The procedure stops with the coefficient $m_\ell$ which, by definition of
  $\bfm$, can be only $0$ or $1$, and tells if the last pattern is
  $a_{\ell-1}\otimes a_\ell$ or $a_{\ell-1}\rhd a_\ell$. 
\end{enumerate}
\end{definition}

\begin{example}
Let us give some examples of this algorithm, for $\ell=5$. 
Fix $a_1,...,a_5\in A$ and set $n_i=|a_i|$ for $i=1,...,5$. 
For $\bfm = (2,1,0,2,0)$, the multi-monomial $R^5_{(2,1,0,2,0)}(a_1,...,a_5)$ 
is composed of two tensor factors (because $m_1=2$). 
The variable $a_1$ acts by $\rhd$ on a multi-monomial of length $1$ 
(because $m_2=1$) which starts necessarily by $a_2$, 
and since $a_2$ does not act by $\rhd$ (because $m_3=0$), 
the first tensor factor is necessarily of the form $a_1\rhd a_2$. 
Then the second tensor factor starts with $a_3$ acting by $\rhd$ on a 
multi-monomial of length $2$ (because $m_4=2$), which starts 
necessarily by $a_4$. Since $a_4$ does not act by $\rhd$ (because $m_5=0$), 
the second tensor factor is necessarily of the form $a_3\rhd (a_4\sotimes a_5)$.
Therefore we finally have
\begin{align*}
R^5_{(2,1,0,2,0)}(a_1,...,a_5) 
& = (a_1\rhd a_2)\sotimes \big(a_3\rhd (a_4\sotimes a_5)\big) 
\\
& = \binom{n_1+1}{1} \binom{n_3+1}{2}\, a_1 a_2 \sotimes a_3 a_4 a_5.
\end{align*}

For $\bfm=(2,1,2,0,0)$, the variable $a_1$ still acts by $\rhd$ 
on a multi-monomial of length $1$ which starts necessarily by $a_2$, 
but this time $a_2$ itself acts by $\rhd$ on a multi-monomial of length $2$, 
and this exhausts the possible $\rhd$ operations. Finally, this time we have
\begin{align*}
R^5_{(2,1,2,0,0)}(a_1,...,a_5) 
& = [a_1\rhd (a_2\rhd (a_3\sotimes a_4)] \sotimes a_5 
\\
& = \binom{n_1+1}{1}\binom{n_2+1}{2}\, a_1 a_2 a_3 a_4 \sotimes a_5.
\end{align*}

Note that the binomial coefficients given by a sequence $\bfm\in \calM_\ell$
can be determined directly from the last $\ell-1$ digits, plus an 
extra null value. For $(2,1,0,2,0)$ we have exactly 
$$
\binom{n_1+1}{1} \binom{n_2+1}{0} \binom{n_3+1}{2} 
\binom{n_4+1}{0} \binom{n_5+1}{0} . 
$$
and for $(2,1,2,0,0)$ we have
$$
\binom{n_1+1}{1} \binom{n_2+1}{2} \binom{n_3+1}{0} 
\binom{n_4+1}{0} \binom{n_5+1}{0} . 
$$

Two more examples of the algorithm: for $\bfm=(3,0,2,0,0)$, 
\begin{align*}
R^5_{(3,0,2,0,0)}(a_1,...,a_5) 
& = a_1 \sotimes \big(a_2\rhd (a_3\sotimes a_4)\big) \sotimes a_5
\\
& = \binom{n_2+1}{2}\, a_1 \sotimes a_2 a_3a_4 \sotimes a_5 , 
\end{align*}
and for $\bfm=(4,0,1,0,0)$, 
\begin{align*}
R^5_{(4,0,1,0,0)}(a_1,...,a_5) 
& = a_1 \sotimes (a_2\rhd a_3) \sotimes a_4\sotimes a_5 
\\
& = \binom{n_2+1}{1}\, a_1 \sotimes a_2 a_3\sotimes a_4\sotimes a_5 .  
\end{align*}
\end{example}

\begin{lemma}
\label{lemma-R1}
For any $\ell\geq 1$ and any $(a_1,...,a_\ell)$, we have 
\begin{align*}
R_\ell(a_1,...,a_\ell) 
& = \sum_{\bfm\in\calM_\ell} R^\ell_{\bfm}(a_1,...,a_\ell). 
\end{align*}
\end{lemma}

For instance, for $\ell=1,2,3$ we have
\begin{align*}
R^1_{(1)}(a) & = a, 
\\
R^2_{(1,1)}(a,b) & = a\rhd b 
= \binom{n_1+1}{1}\, a b
\\
R^2_{(2,0)}(a,b) & = a\sotimes b
\\
R^3_{(1,1,1)}(a,b,c) & = a\rhd (b\rhd c)  
= \binom{n_1+1}{1}\binom{n_2+1}{1}\, a b c .
\\ 
R^3_{(1,2,0)}(a,b,c) & = a\rhd (b\sotimes c) 
= \binom{n_1+1}{2}\, a b c 
\\
R^3_{(2,0,1)}(a,b,c) & = a \sotimes (b\rhd c) 
= \binom{n_2+1}{1}\, a\sotimes b c 
\\ 
R^3_{(2,1,0)}(a,b,c) & = (a\rhd b)\sotimes c 
= \binom{n_1+1}{1}\, a b\sotimes c 
\\ 
R^3_{(3,0,0)}(a,b,c) & = a\sotimes b\sotimes c 
\end{align*}
Comparing with the value of $R_1$, $R_2$ and $R_3$ given above, the assertion 
is easily verified. 
\bigskip 

\begin{proof}
Let us call $\tilde{R}_\ell(a_1,...,a_\ell)$ the sum over 
$\bfm\in \calM_\ell$ of Lemma \ref{lemma-R1}, and prove that it solves 
equation (\ref{definition-R-equation}) by induction on $\ell$. 

For $\ell=1,2,3$ the assertion was proved in the examples. 
For any $\ell\geq 1$, we then suppose that on the 
right-hand side of eq.~(\ref{definition-R-equation}) we have 
$R_{p_i-1} = \tilde{R}_{p_i-1}$ for any $1\leq i\leq j$, and we set 
$$
P_i=p_1+p_2+\cdots +p_i , 
$$
so that 
\begin{align*}
R_\ell(a_1,...,a_\ell) 
&= \sum_{j=1}^\ell\ \sum_{\bfp\in \calC_\ell^j}\ 
\sum_{i=1}^j \sum_{\bfq^{(i)}\in \calM_{p_i-1}}
\big( a_1 \rhd \tilde{R}^{p_1-1}_{\bfq^{(1)}}(a_2,...,a_{p_1}) \big) \sotimes 
\\ 
& \hspace{-1cm}
\sotimes 
\big( a_{p_1+1} \rhd \tilde{R}^{p_2-1}_{\bfq^{(2)}}(a_{p_1+2},...,a_{P_2}) \big) 
\sotimes \cdots 
\sotimes \big( a_{P_{j-1}+1} \rhd 
\tilde{R}^{p_j-1}_{\bfq^{(j)}}(a_{P_{j-1}+2},...,a_{P_j}) \big). 
\end{align*}
In this sum, we can note the following things: 
\begin{itemize}
\item
The running value $j$ gives the length of the corresponding multi-monomial. 
\item
In the first tensor factor, the value $q^{(1)}_1$ represents the length of 
$\tilde{R}^{p_1-1}_{\bfq^{(1)}}(a_2,...,a_{p_1})$, that is, a sequence number 
associated to $a_1$, and more generally $\bfq^{(1)}$ rules the nested 
operations up to the variable $a_{p_1-1}$. 
The last variable $a_{p_1}$ does not act on further variables and so it 
should be associated to a missing value $0$. 
Therefore, the nested operations in the whole first tensor factor are ruled 
by the sequence $(\bfq^{(1)},0)$. 
\item
Similarly, for any $i\leq j$, the nested operations in the $i$th tensor factor 
are ruled by the sequence $(\bfq^{(i)},0)$. 
\end{itemize}
Let us then associate to this expression the sequence
$$
\bfm = (j,\bfq^{(1)},0,\bfq^{(2)},0,...,\bfq^{(j)}), 
$$
that is, 
\begin{align*}
& m_1 = j \\ 
& m_2 = q^{(1)}_1,\ \ldots\ ,\ m_{p_1} = q^{(1)}_{p_1-1}\ ,\ m_{p_1+1} = 0\ , \\ 
& m_{P_{i-1}+2} = q^{(i)}_1,\ \ldots\ ,\ m_{P_i} = q^{(i)}_{p_i-1}\ ,\ m_{P_i+1} = 0\ , 
\quad\mbox{for $1\leq i\leq j-1$} \\ 
& m_{P_{j-1}+2} = q^{(j)}_1,\ \ldots\ ,\ m_{P_j} = q^{(j)}_{p_j-1}, 
\end{align*}
which has precisely length 
$$
1+(p_1-1)+1+(p_2-1)+\cdots +1+(p_j-1) = p_1+\cdots +p_j = P_j= \ell.   
$$
Note that in the sum over the sequences $\bfp=(p_1,...,p_j)\in \calC^j_\ell$,  
where $p_i\geq 1$ for $i=1,...,j$, there occur the terms with $p_i-1=0$. 
In this case the multipolynomial $\tilde{R}_{p_i-1}=R_0=1$ has no variables, and 
the set $\calM_{p_1-1}=\calM_0$ is empty. The corresponding sequence $\bfq^{(i)}$ 
is then absent in $\bfm$, but its associated null value must be present, 
for any $i=1,...,j-1$, to preserve the total length $\ell$. 
Following the rules of the algorithm given in Def.~\ref{definition-R_m}, 
we can therefore write 
\begin{align*}
\big( a_1 \rhd \tilde{R}^{p_1-1}_{\bfq^{(1)}}(a_2,...,a_{p_1}) \big) \sotimes 
\big( a_{p_1+1} \rhd \tilde{R}^{p_2-1}_{\bfq^{(2)}}(a_{p_1+2},...,a_{P_2}) \big) 
\sotimes \cdots \qquad & 
\\ 
\cdots \sotimes \big( a_{P_{j-1}+1} \rhd 
\tilde{R}^{p_j-1}_{\bfq^{(j)}}(a_{P_{j-1}+2},...,a_{P_j}) \big) \quad 
& = \quad R_{\bfm}(a_1,...,a_\ell) . 
\end{align*}

Let us call 
$$
\calN_\ell = \{ \bfm = (j,\bfq^{(1)},0,\bfq^{(2)},0,...,\bfq^{(j)})\ |\ 
1\leq j\leq \ell,\ \bfp\in \calC_\ell^j,\ \bfq^{(i)}\in \calM_{p_i-1} 
\ \mbox{for}\ 1\leq i\leq j \} 
$$
the set of sequences obtained in this way. 
Then the equality 
\begin{align*}
R_\ell(a_1,a_2,...,a_\ell) 
& = \sum_{\bfm\in \calN_\ell} R^\ell_{\bfm}(a_1,...,a_\ell)
= \tilde{R}_\ell(a_1,a_2,...,a_\ell)
\end{align*}
holds if we show that $\calN_\ell = \calM_\ell$. 

Let us first show that $\calN_\ell \subset \calM_\ell$. 
For fixed $j$, $\bfp$ and $\bfq^{(1)},...,\bfq^{(j)}$, we have 
\begin{align*}
m_1+\cdots +m_\ell 
& = j+\sum_{i=1}^j \big(q^{(i)}_1+\cdots + q^{(i)}_{p_i-1}\big) \\ 
& = j+(p_1-1)+(p_2-1)+\cdots +(p_j-1) \\ 
& = j+ p_1+\cdots +p_j -j= \ell. 
\end{align*}
For any $h=1,...,\ell$, suppose that $h$ belongs to the $r$th block, for 
some $r\leq j$, that is, 
$$
h = P_{r-1}+1+k = p_1+\cdots +p_{r-1}+1+k, 
$$
with $1\leq k\leq p_r-1$. Then we have
\begin{align*}
m_1+\cdots +m_h 
& = j+\sum_{i=1}^{r-1} \big(q^{(i)}_1+\cdots + q^{(i)}_{p_i-1}\big) 
+ \big(q^{(r)}_1+\cdots + q^{(r)}_k\big)
\\ 
& = j+(p_1-1)+(p_2-1)+\cdots +(p_{r-1}-1) 
+ \big(q^{(r)}_1+\cdots + q^{(r)}_k\big) \\
& = (p_1+\cdots +p_{r-1})+\big(q^{(r)}_1+\cdots +q^{(r)}_k\big)+(j-r)+1 \\
& \geq P_{r-1}+1+k = h
\end{align*}
because $q^{(r)}_1+\cdots +q^{(r)}_k \geq k$ and $j-r\geq 0$. 

Finally, let us show that there is a bijection between $\calN_\ell$ and 
$\calM_\ell$. 
The set $\calM_\ell$ is well known to be in bijection with the set $PBT_{\ell+1}$ 
of planar binary trees with $\ell+1$ leaves (and a root). 
An explicit bijection $\Phi:\calM_\ell \longrightarrow PBT_{\ell+1}$ is described
in~\cite{BFK}, Definition 2.16, using the {\em over} and {\em under} 
grafting operations on trees, namely 
$$
t\over s = \raisebox{-.1\height}{\lgraft{s}{t}} 
\qquad\mbox{and}\qquad 
t\under s = \raisebox{-.1\height}{\rgraft{t}{s}}. 
$$
The first values of $\Phi$, for the empty sequence in $\calM_0$ 
and for $(1)\in \calM_1$ and $(2,0),(1,1)\in\calM_2$, are 
$$
\Phi(\ ) = \treeO, \qquad \Phi(1) = \treeA, 
\qquad \Phi(2,0) = \treeAB, \qquad \Phi(1,1) = \treeBA. 
$$
So, for our purpose, it is enough to show that $\calN_\ell$ is in bijection 
with $PBT_{\ell+1}$. 
For this, since $\calN_\ell\subset \calM_\ell$, consider the map $\Phi$ 
restricted to $\calN_\ell$ and let us show that the image $\Phi(\calN_\ell)$ 
coincides with $PBT_{\ell+1}$. 

For a given sequence 
$\bfm = (j,\bfq^{(1)},0,\bfq^{(2)},0,...,\bfq^{(j)})\in \calN_\ell$, we have: 
\begin{itemize}
\item
The sequence $\bfm$ is decomposable as $\bfm = (\bfm',\bfq^{(j)})$ 
into the two well-defined sequences 
$\bfm'=(j,\bfq^{(1)},0,\bfq^{(2)},0,...,\bfq^{(j-1)},0) \in \calM_{P_{j-1}+1}$ 
and $\bfq^{(j)}\in \calM_{p_j-1}$. 
In fact, if we set $\ell'=P_{j-1}+1$, we have 
\begin{align*}
m'_1+\cdots m'_{\ell'} & = 
j+(p_1-1)+\cdots +(p_{j-1}-1)= j+P_{j-1}-(j-1) = \ell',  
\end{align*}
and for any $h=1,...,\ell'$ one can see that $m'_1+\cdots m'_h\geq h$ 
with a computation similar to that used to show that $\bfm\in \calM_{\ell}$. 
 
According to the definition of $\Phi$, we then have 
$\Phi(\bfm) = \Phi(\bfm')\under \Phi(\bfq^{(j)})$. 
Graphically, if we denote the trees by $t=\Phi(\bfm)$, 
$t'=\Phi(\bfm')$ and $t_j=\Phi(\bfq^{(j)})$, this means that 
$$
t = \raisebox{-.1\height}{\rgraft{t'}{t_j}}.  
$$

\item 
The sequence $\bfm'$ is surely not decomposable because it is of the form 
$$
\bfm'=(m''_1+1,m''_2,...,m''_{\ell''},0)
$$ 
with 
$$
\bfm''=(j\!-\!1,\bfq^{(1)},0,\bfq^{(2)},0,...,\bfq^{(j-1)})\in \calM_{\ell''}, 
\qquad \ell''=\ell'-1. 
$$ 
The sequence $\bfm''$ indeed belongs to $\calM_{\ell''}$ for the same reason 
used to show that $\bfm\in \calM_\ell$. 
Then, the sequence $\bfm'$ is not decomposable in position $1$ because 
$m'_1=m''_1+1 \geq 2$, and it is not decomposable in any position 
$h=2,...,\ell'$ because $\bfm'' \in \calM_{\ell''}$ implies that 
$m'_1+\cdots +m'_h = m''_1+1+m''_2+\cdots + m''_h \geq h+1$, and therefore 
surely $m'_1+\cdots +m'_h \neq h$.  

Finally, according to the definition of $\Phi$, we then have 
$\Phi(\bfm') = \Phi(\bfm'') \over \treeA$. 
If we set $t''=\Phi(\bfm'')$, this means that 
$$
t' = \raisebox{-.1\height}{\lvertexgraft{t''}} 
$$
and therefore 
$$
t = \raisebox{-.1\height}{\veegraft{t''}{t_j}}. 
$$

\item
The same arguments can be applied to the sequence $\bfm''$ and its new 
components, until we reach a full description of the tree $t=\Phi(\bfm)$ 
in terms of the trees $t_i = \Phi(\bfq^{(i)})$, for $i=1,...,j$, namely 
$$
t = \raisebox{-.1\height}{\combLgraft{t_1}{t_2}{t_j}}. 
$$
Let us denote this tree by $G^j(t_1,...,t_j)$. 
\end{itemize}
In conclusion, if we let $j$ run from $1$ to $\ell$, we consider all possible 
sequences $\bfp\in \calC^j_\ell$ and for any $i=1,...,j$ all trees 
$t_i\in \Phi(\calM_{p_i-1})= PBT_{p_i}$, the result $G^j(t_1,...,t_j)$ is 
any possible tree with number of leaves given by 
$$
\left|G^j(t_1,...,t_j)\right| 
= 1+|t_1|+\cdots +|t_j| 
= 1+p_1+\cdots +p_j= \ell+1.
$$
In other words, we have 
$$
\Phi(\calN_\ell) = \big\{ t=G^j(t_1,...,t_j)\ |\ 
j=1,...,\ell,\ 
\bfp\in \calC^j_\ell,\ 
t_i \in PBT_{p_i},\ 1\leq i\leq j \big\}  = PBT_{\ell+1}. 
$$
\end{proof}

\begin{corollary}
\label{corollary-R1}
For any $\ell\geq 0$ and any $a_1,...,a_{\ell+1}\in A$, set $n_i=|a_i|\geq 1$ for 
$i=1,...,\ell+1$.  
Then, for any sequence $\bfm\in \calM_\ell$, we have 
\begin{align*}
a_1\rhd R_{\bfm}(a_2,...,a_{\ell+1}) 
&= \binom{n_1+1}{m_1}\cdots \binom{n_\ell+1}{m_\ell} \binom{n_{\ell+1}+1}{0}\
a_1 a_2\cdots a_{\ell+1} . 
\end{align*}
Therefore 
\begin{align*}
a_1\rhd R_\ell(a_2,...,a_{\ell+1}) 
&= d(n_1,...,n_{\ell+1})\ a_1 a_2\cdots a_{\ell+1} , 
\end{align*}
where the Lagrange coefficients $d(n_1,...,n_{\ell+1})$ are given in 
Def. \ref{definition-d}. 
\end{corollary}


To describe the left codivision we introduce a last set of operators 
corresponding to the labeled Lagrange coefficients. 

\begin{definition}
\label{definition-Re}
For any $\ell \geq 1$, let $\calE_\ell$ be the set of sequences of bits
$e_i\in \{1,2\}$, as in Def.~\ref{definition-de}. 
We call {\bf labeled right recursive operator} 
$R^{\bfe}:T(A)\longrightarrow T(A)$ the collection 
$R^{\bfe} = \{ R^{\bfe}_\ell,\ \ell\geq 0,\ \bfe\in \calE_\ell \}$  
of (non homogeneous) linear operators $R^{\bfe}_0=\Id:\F\longrightarrow \F$ and 
\begin{align*}
& R^{\bfe}_\ell: 
A^{\sotimes \ell} \longrightarrow \bigoplus_{\lambda=1}^{\ell} A^{\sotimes \lambda},
\qquad \ell \geq 1
\end{align*}
defined recursively by
\begin{align*}
& R_0^{(1)}(a) = a \qquad\mbox{and}\qquad R_0^{(2)}(a)= 0, 
\end{align*}
and, for $\ell \geq 2$ and for any $\bfe = (e_1,...,e_\ell)\in \calE_\ell$, by
\begin{align}
\label{definition-Re-equation}
R^{\bfe}_\ell(a_1,...,a_\ell) &= 
\sum_{j=1}^\ell \sum_{\bfp\in \calC_\ell^j} 
\big( R_1^{(e_1)}(a_1) \rhd R_{p_1-1}^{(e_2,...,e_{p_1})}(a_2,...,a_{p_1}) \big) 
\\ \nonumber
& \hspace{2cm} 
\sotimes \big( a_{p_1+1} \rhd
R_{p_2-1}^{(e_{p_1+2},...,e_{p_1+p_2})}(a_{p_1+2},...,a_{p_1+p_2}) \big) 
\sotimes \cdots 
\\ \nonumber
& \hspace{3cm} 
\cdots \sotimes \big( a_{p_1+\cdots +p_{j-1}+1} \rhd 
R_{p_j-1}^{(e_{p_1+\cdots +p_{j-1}+2},...,e_\ell)}
(a_{p_1+\cdots +p_{j-1}+2},...,a_\ell) \big) . 
\end{align}
It turns out that $R_\ell^{\bfe}=R_\ell$ if $\bfe=(1,1,...,1)$. 
If $\bfe$ starts by $2$, then $R_\ell^{\bfe}=0$. 
If $\bfe$ starts by $1$ and contains a bit value $e_i=2$ (in position $i$),
then $R_\ell^{\bfe}$ is obtained from $R_\ell$ by removing the term which
contains the factor $a_{i-1}\rhd Q_i$.
\bigskip 

For instance, for $\ell=2$, we have $\calE_1 = \{ (1,1),(1,2),(2,1),(2,2) \}$
and therefore 
\begin{align*}
R_2^{(1,1)}(a,b)
& = R_1^{(1)}(a)\rhd R_1^{(1)}(b) + R_1^{(1)}(a) \sotimes (b \rhd R_0(1)) 
\\ 
& = a\rhd b + a \sotimes b  
\\[.2cm] 
R_2^{(1,2)}(a,b)
& = R_1^{(1)}(a)\rhd R_1^{(2)}(b) + R_1^{(1)}(a) \sotimes (b \rhd R_0(1))  
\\
& = a \sotimes b   
\\[.2cm] 
R_2^{(2,1)}(a,b)
& = R_1^{(2)}(a)\rhd R_1^{(1)}(b) + R_1^{(2)}(a) \sotimes (b \rhd R_0(1)) = 0
\\[.2cm] 
R_2^{(2,2)}(a,b) 
& = R_1^{(2)}(a)\rhd R_1^{(2)}(b) + R_1^{(2)}(a) \sotimes (b \rhd R_0(1)) = 0. 
\end{align*}
For $\ell=3$, the set $\calE_2$ contains $8$ sequences, which give
\begin{align*}
R_3^{(1,1,1)}(a,b,c) 
& = a\rhd R_2^{(1,1)}(b,c) 
+ \big(a\rhd R_1^{(1)}(b)\big)\sotimes c 
+ a \sotimes \big(b\rhd R_1^{(1)}(c)\big)
+ a\sotimes b\sotimes c 
\\
& = a\rhd (b\rhd c) 
+ a\rhd (b\sotimes c) 
+ (a\rhd b)\sotimes c 
+ a \sotimes (b\rhd c)
+ a\sotimes b\sotimes c 
\\[.3cm]
R_3^{(1,1,2)}(a,b,c) 
& = a\rhd R_2^{(1,2)}(b,c) 
+ \big(a\rhd R_1^{(1)}(b)\big)\sotimes c 
+ a \sotimes \big(b\rhd R_1^{(2)}(c)\big)
+ a\sotimes b\sotimes c 
\\
& = a\rhd (b\sotimes c) 
+ (a\rhd b)\sotimes c 
+ a\sotimes b\sotimes c 
\\[.3cm]
R_3^{(1,2,1)}(a,b,c) 
& = a\rhd R_2^{(2,1)}(b,c) 
+ \big(a\rhd R_1^{(2)}(b)\big)\sotimes c 
+ a \sotimes \big(b\rhd R_1^{(1)}(c)\big)
+ a\sotimes b\sotimes c 
\\ 
& = a \sotimes (b\rhd c)
+ a\sotimes b\sotimes c 
\\[.3cm]
R_3^{(1,2,2)}(a,b,c) 
& = a\rhd R_2^{(2,2)}(b,c) 
+ \big(a\rhd R_1^{(2)}(b)\big)\sotimes c 
+ a \sotimes \big(b\rhd R_1^{(2)}(c)\big)
+ a\sotimes b\sotimes c 
\\ 
& = a\sotimes b\sotimes c 
\end{align*}
and finally\quad $R_3^{(2,1,1)} = R_3^{(2,1,2)} = R_3^{(2,2,1)} = R_3^{(2,2,2)} =0$.  
\end{definition}

The labeled right operations can also be given by a closed formula. 

\begin{lemma}
\label{lemma-Re3}
For any $\ell\geq 2$, for any sequence $\bfe\in\calE_\ell$ and for any 
$a_1,...,a_\ell\in A$, we have
\begin{align*}
R_\ell^{\bfe}(a_1,...,a_\ell) 
& = \sum_{\bfm\in \calM_\ell^{\bfe}} R^\ell_{\bfm}(a_1,...,a_\ell). 
\end{align*}

As a consequence, if for $a_1,...,a_{\ell+1}\in A$ we denote $n_i=|a_i|$ 
for $i=1,...,\ell+1$, we then have 
\begin{align*}
a_1\rhd R_\ell^{\bfe}(a_2,...,a_{\ell+1}) & = 
d_\ell^{\bfe}(n_1,...,n_\ell)\ a_1\cdots a_{\ell+1} \in A, 
\end{align*}
where the labeled Lagrange coefficients $d_\ell^{\bfe}(n_1,...,n_\ell)$ 
are given in Def. \ref{definition-de}. 
\end{lemma}

\begin{proof}
If $\bfe=(1,1,...,1)$, and if $\bfe$ starts by $2$, there is nothing to prove.
Otherwise, for any value $e_i=2$ in $\bfe$, we obtain
$R_\ell^{\bfe}(a_1,...,a_\ell)$ from $R_\ell(a_1,...,a_\ell)$ by removing the term
containing the factor $a_{i-1}\rhd Q_i$. 
By Lemma \ref{lemma-R1}, such a term is associated to a sequence 
$\bfm\in \calM_\ell$, and by Def.~\ref{definition-R_m} the factor 
$a_{i-1}\rhd Q_i$ corresponds to a non-zero value $m_i$. 
Therefore, in order to remove such terms, it suffices to consider sequences 
$\bfm$ with $m_i=0$ whenever $e_i=2$. 
\end{proof}


\begin{theorem}
\label{FaaDiBruno-operators}
The co-operations of the Fa\`a di Bruno coloop bialgebra can be equivalently 
defined in terms of the recursive operators as follows: 
\begin{align}
\nonumber
\DFdBex(x_n) 
& = x_n + y_n + \sum_{\ell=1}^{n-1} \sum_{\bfn\in \calC_n^{\ell+1}} \
x_{n_1} \rhd \big(y_{n_2}\sotimes\cdots \sotimes y_{n_{\ell+1}}\big)  , 
\\ \nonumber
\delta_r(x_n) 
& = u_n + \sum_{\ell=1}^{n-1} (-1)^\ell \sum_{\bfn \in \calC_n^{\ell+1}} 
L_\ell(u_{n_1},y_{n_2},...,y_{n_\ell})\rhd y_{n_{\ell+1}} 
\\ \nonumber
& = u_n + \sum_{\ell=1}^{n-1} (-1)^\ell \sum_{\bfn \in \calC_n^{\ell+1}} 
u_{n_1} \rhd R_\ell(y_{n_2},...,y_{n_{\ell+1}}) ,  
\\ \label{FaaDiBruno-operators-left-codivision}
\delta_l(x_n) 
& = v_n + \sum_{\ell=1}^{n-1} (-1)^\ell \sum_{\bfn \in \calC_n^{\ell+1}} 
\sum_{\bfe\in \calE_\ell} (-1)^{\bfe} \ x_{n_1}^{(e_1)}\rhd 
R_\ell^{\bfe}(x^{(e_2)}_{n_2},...,x^{(e_\ell)}_{n_\ell},v_{n_{\ell+1}})
\end{align}
where we recall that $u_n=x_n-y_n$, $v_n=y_n-x_n$, and also that
$(-1)^{\bfe}=(-1)^{e_1+\cdots +e_\ell-\ell}$ and that the bit value in $x^{(e)}_n$
tells us in which copy of $\Hdex \scoprod \Hdex$ falls the generator $x_n$,
cf. Def.~\ref{FaaDiBrunoColoopBialgebra}.
\end{theorem}

\begin{proof}
It follows from the definition of $\rhd$ given in Def.~\ref{definition-rhd}, 
the expression of $R_\ell$ given in Cor.~\ref{corollary-R1}, and that of 
$R_\ell^{\bfe}$ given in Lemma~\ref{lemma-Re3}. The equivalence of the 
presentations of the right codivision in terms of $R_\ell$ and $L_\ell$ 
is proved in Cor.~\ref{corollary-LR} in next section. 
\end{proof}

Note that in the term $x_{n_1}^{(e_1)}\rhd 
R_\ell^{\bfe}(x^{(e_2)}_{n_2},...,x^{(e_\ell)}_{n_\ell},v_{n_{\ell+1}})$ of the left
codivision (\ref{FaaDiBruno-operators-left-codivision}),
the labeled operator $R_\ell^{\bfe}=R_\ell^{(e_1,...,e_\ell)}$ is applied to variables
which are also labeled, but only by the last $\ell-1$ bits of $\bfe$. 
For instance, no labels affect the variables in 
\begin{align*}
\delta_l(x_2)
& = v_2 - x_1^{(1)} \rhd R_1^{(1)}(v_1) + x_1^{(2)} \rhd R_1^{(2)}(v_1)
\\
& = v_2 - x_1 \rhd v_1 + y_1 \rhd 0 
\\
& = v_2 - 2x_1 v_1
\end{align*}
but labels do affect the variables in 
\begin{align*}
\delta_l(x_3)
& = v_3 - \big( x_1^{(1)} \rhd R_1^{(1)}(v_2) + x_2^{(1)} \rhd R_1^{(1)}(v_1) \big) 
- \big( -x_1^{(2)} \rhd R_1^{(2)}(v_2) - x_2^{(1)} \rhd R_1^{(2)}(v_1) \big)
\\ 
& \hspace{1cm}
+ x_1^{(1)} \rhd R_2^{(1,1)}(x_1^{(1)},v_1) - x_1^{(1)} \rhd R_2^{(1,2)}(x_1^{(2)},v_1)
- x_1^{(2)} \rhd R_2^{(2,1)}(x_1^{(1)},v_1) 
\\ 
& \hspace{2cm}
+ x_1^{(2)} \rhd R_2^{(2,2)}(x_1^{(2)},v_1) 
\\
& = v_3 - (x_1 \rhd v_2 + x_2 \rhd v_1) + x_1 \rhd (x_1 \rhd v_1)
+ x_1 \rhd (x_1 \sotimes v_1) - x_1 \rhd (y_1 \sotimes v_1)
\\
& = v_3 - \left( \binom{2}{1} x_1 v_2 + \binom{3}{1} x_2 v_1 \right)
+ \left( \binom{2}{1}\binom{2}{1}+\binom{2}{2} \right) x_1^2 v_1 
- \binom{2}{2} x_1 y_1 v_1 
\\
& = v_3 - (2x_1 v_2 + 3 x_2 v_1) + 5 x_1^2 v_1 - x_1 y_1 v_1. 
\end{align*}


\subsection{Functoriality of the diffeomorphisms loop}

To prove the main theorem of this section we need some preliminary recurrence 
relations for the recursive operators, and consequently for the Lagrange 
coefficients.
\bigskip 

\begin{corollary}
\label{corollary-R2}
For any $\ell\geq 1$ and any sequence $(n_1,...,n_{\ell+1})$ of positive 
integers, the coefficients $d_{\ell+1}(a_1,...,a_\ell)$ satisfy the 
following recursive equation: 
\begin{align*}
d_\ell(n_1,...,n_\ell) &= \sum_{j=1}^\ell \sum_{\bfp\in \calC_\ell^j} 
\binom{n_1+1}{j}\ d_{p_1-1}(n_2,...,n_{p_1})\ d_{p_2-1}(n_{p_1+2},...,n_{p_1+p_2})\cdots
\\ \nonumber
& \hspace{4cm} 
\cdots d_{p_j-1}(n_{p_1+\cdots +p_{j-1}+2},...,n_\ell). 
\end{align*}
\end{corollary} 

\begin{proof}
Applying $a_1\rhd (\ )$ to the recursive expression
(\ref{definition-R-equation}) of $R_\ell(a_2,a_3,...,a_{\ell+1})$,
and using Cor.~\ref{corollary-R1}, immediately gives the result. 
\end{proof}

\begin{lemma}
\label{lemma-R2}
For any $\ell\geq 0$ and any $a_1,...,a_{\ell+1}\in A$, the following recursive 
equation holds:  
\begin{align*}
a_1 \rhd R_\ell(a_2,...,a_{\ell+1})
&= \sum_{i=0}^{\ell-1} (-1)^{\ell-1-i} 
\big( a_1 \rhd R_i(a_2,...,a_{i+1}) \big) 
\rhd (a_{i+2}\sotimes \cdots \sotimes a_{\ell+1} ). 
\end{align*}
Modulo the factor $a_1 a_2\cdots a_{\ell+1}$, this means that 
\begin{align*}
d_\ell(n_1,...,n_\ell) 
& = \sum_{i=0}^{\ell-1} (-1)^{\ell-1-i} 
\binom{n_1+\cdots + n_{i+1}+1}{\ell-i}\ d_i(n_1,...,n_i). 
\end{align*}
\end{lemma}

\begin{proof}
The two assertions are equivalent, and the second one appears as a recursion
for the coefficients in the non-commutative Lagrange inversion formula. 
It is essentially based on the Chu-Vandermonde identity and can be 
proved\footnote{
We warmly thank Jiang Zeng for pointing out this method to us.} 
using the hypergeometric function ${}_2F_1$ or using some trick as in 
\cite{BFK}, Lemma 2.15. 
\end{proof}

\begin{corollary}
\label{corollary-LR}
For any $\ell\geq 0$ and any $a_1,...a_{\ell+1}\in A$ we have 
\begin{align*}
a_1 \rhd R_\ell(a_2,...,a_{\ell+1}) &= L_\ell(a_1,...,a_\ell) \rhd a_{\ell+1}
\end{align*}
\end{corollary}

\begin{proof}
By induction on $\ell$. For $\ell=0,1$ the identity is easily verified, 
because 
\begin{align*}
a_1\rhd 1 & = a_1 = 1 \rhd a_1 
\\ 
a_1 \rhd R_1(a_2) & = a_1 \rhd a_2 = L_1(a_1) \rhd a_2. 
\end{align*}
Now suppose that for $i=1,...,\ell-1$ we have 
$a_1 \rhd R_i(a_2,...,a_{i+1}) = L_i(a_1,...,a_i) \rhd a_{i+1}$. 
Then by Lemma \ref{lemma-R2} and Def.~\ref{definition-L} we have
\begin{align*}
a_1 \rhd R_\ell(a_2,...,a_{\ell+1})
&= \sum_{i=0}^{\ell-1} (-1)^{\ell-1-i} 
\big( a_1 \rhd R_i(a_2,...,a_{i+1}) \big) 
\rhd (a_{i+2}\sotimes \cdots \sotimes a_{\ell+1} ) 
\\ 
& = \sum_{i=0}^{\ell-1} (-1)^{\ell-1-i} 
\big( L_i(a_1,...,a_i)\rhd a_{i+1} \big) 
\rhd (a_{i+2}\sotimes \cdots \sotimes a_{\ell+1} ) 
\\ 
& = \sum_{i=0}^{\ell-1} (-1)^{\ell-1-i} 
\Big( \big( L_i(a_1,...,a_i)\rhd a_{i+1} \big) 
\sotimes a_{i+2}\sotimes \cdots \sotimes a_\ell \Big) \rhd a_{\ell+1}  
\\ 
& = L_\ell(a_1,...,a_\ell) \rhd a_{\ell+1} . 
\end{align*}
\end{proof}

\begin{remark}
In the case $\bfe\neq (1,1,...,1)$, whether there exists an operator 
$L_\ell^{\bfe} :A^{\sotimes \ell} \longrightarrow T(A)$ such that 
$$
a_1\rhd R_\ell^{\bfe}(a_2,...,a_{\ell+1}) = L_\ell^{\bfe}(a_1,...,a_\ell)\rhd a_{\ell+1} 
$$
is an open question. 
\end{remark}

\begin{lemma}
\label{lemma-R3}
For any $\ell\geq 1$ and any $a_1,...,a_{\ell+1}\in A$, the following recursive 
equation holds:  
\begin{align*}
a_1 \rhd R_\ell(a_2,...,a_{\ell+1})
&= \sum_{i=1}^\ell (-1)^{i-1} 
\big( a_1 \rhd (a_2\sotimes\cdots \sotimes a_{i+1}) \big) 
\rhd R_{\ell-i}(a_{i+2}\sotimes \cdots \sotimes a_{\ell+1} ) . 
\end{align*}
Modulo the factor $a_1 a_2\cdots a_{\ell+1}$, and if we call $n_i=|a_i|$ for 
$i=1,...,\ell+1$, this means that 
\begin{align}
\label{lemma-R3-equation-d} 
d_\ell(n_1,...,n_\ell) 
& = \sum_{i=1}^\ell (-1)^{i-1} 
\binom{n_1+1}{i}\ d_{\ell-i}(n_1\!+\!\cdots \!+\! n_{i+1},n_{i+2},...,n_\ell). 
\end{align}
\end{lemma}

\begin{proof}
The two assertions are equivalent. Let us prove the second one by induction 
on $\ell$. 
Let us call $\tilde{d}_\ell(n_1,...,n_\ell)$ the right-hand side of 
equation (\ref{lemma-R3-equation-d}). 
For $\ell=1$, the sum in $\tilde{d}_1(n)$ has only one term for $i=1$,
which gives
\begin{align*}
\tilde{d}_1(n) & = (-1)^{1-1} \binom{n+1}{1}\ d_0 = d_1(n). 
\end{align*}
Now suppose that eq.~(\ref{lemma-R3-equation-d}) holds for any 
$1\leq k\leq \ell-1$, that is, we have
\begin{align*}
d_k(n_1,...,n_k) 
& = \sum_{i=1}^k (-1)^{i-1} 
\binom{n_1+1}{i}\ d_{k-i}(n_1\!+\!\cdots \!+\! n_{i+1},n_{i+2},...,n_k),  
\end{align*}
and prove it for $\ell$. For this, we write $d_\ell(n_1,...,n_\ell)$ 
using the recursion given in Lemma \ref{lemma-R2} as a sum over 
$0\leq k\leq \ell-1$, and separate the term $k=0$ to which we can not apply 
the inductive hypothesis. 
Then we expand the factor $d_k(n_1,...,n_k)$ using the inductive 
hypothesis and exchange the sums over $k$ and $i$. We finally obtain 
\begin{align*}
d_\ell(n_1,...,n_\ell) 
& = \sum_{k=1}^{\ell-1} (-1)^{\ell-1-k} 
\binom{n_1\!+\!\cdots \!+\! n_{k+1}+1}{\ell-k}\ d_k(n_1,...,n_k) 
+ (-1)^{\ell-1} \binom{n_1+1}{\ell}
\\ 
& \hspace{-2.5cm} 
= \sum_{i=1}^{\ell-1} (-1)^{i-1} \binom{n_1\!+\!1}{i}\ 
\sum_{k=i}^{\ell-1} (-1)^{\ell-1-k}  
\binom{n_1\!+\!\cdots \!+\! n_{k+1}\!+\!1}{\ell-k}\ 
d_{k-i}(n_1\!+\!\cdots \!+\! n_{i+1},n_{i+2},...,n_k)
\\ 
& \hspace{-1cm} 
+ (-1)^{\ell-1} \binom{n_1+1}{\ell} . 
\end{align*}
Then, $d_\ell(n_1,...,n_\ell)$ is equal to 
\begin{align*}
\tilde{d}_\ell(n_1,...,n_\ell) 
& = \sum_{i=1}^{\ell-1} (-1)^{i-1} \binom{n_1+1}{i}\ 
d_{\ell-i}(n_1\!+\!\cdots \!+\! n_{i+1},n_{i+2},...,n_\ell)
\\
& \hspace{1cm}
+ (-1)^{\ell-1} \binom{n_1+1}{\ell} , 
\end{align*}
if and only if, for any $1\leq i\leq \ell-1$, we have 
\begin{align*}
\sum_{k=i}^{\ell-1} (-1)^{\ell-1-k}  
\binom{n_1\!+\!\cdots \!+\! n_{k+1}\!+\!1}{\ell-k}\ 
d_{k-i}(n_1\!+\!\cdots \!+\! n_{i+1},n_{i+2},...,n_k)
&  
\\ 
& \hspace{-3cm} 
= d_{\ell-i}(n_1\!+\!\cdots \!+\! n_{i+1},n_{i+2},...,n_\ell) . 
\end{align*}
This identity is easily verifyed by setting $j=k-i$, 
$p_1 = n_1+\cdots + n_{i+1}$ and $p_j= n_{i+j}$ for $2\leq j\leq \ell-1-i$, 
since it gives
\begin{align*}
\sum_{j=0}^{\ell-i-1} (-1)^{\ell-i-1-j}  
\binom{p_1\!+\!\cdots \!+\! p_{j+1}\!+\!1}{\ell-i-j}\ 
d_j(p_1,p_2,...,p_j)
& = d_{\ell-i}(p_1,p_2,...,p_{\ell-i}) 
\end{align*}
which holds again by Lemma \ref{lemma-R2}. 
\end{proof}

\begin{corollary}
\label{corollary-L3}
For any $\ell\geq 1$ and any $a_1,...,a_\ell\in A$, the following recursive 
equation holds:  
\begin{align*}
L_\ell(a_1,...,a_\ell) 
&= \sum_{i=1}^{\ell-1} (-1)^{i-1} 
L_{\ell-i}\big(a_1 \rhd (a_2\sotimes\cdots \sotimes a_{i+1}),a_{i+2},...,a_\ell\big) 
+ (-1)^{\ell-1} a_1\sotimes\cdots\sotimes a_\ell .
\end{align*}
\end{corollary}

\begin{proof}
It suffices to write 
\begin{align*}
a_1 \rhd R_\ell(a_2,...,a_{\ell+1})
&= \sum_{i=1}^{\ell-1} (-1)^{i-1} 
\big( a_1\rhd (a_2 \sotimes\cdots \sotimes a_{i+1}) \big) 
\rhd R_{\ell-i}(a_{i+2}\sotimes \cdots \sotimes a_{\ell+1} ) 
\\ 
& \hspace{2cm} 
+ (-1)^\ell a_1 \rhd (a_2\sotimes\cdots \sotimes a_{\ell+1})
\end{align*}
after Lemma \ref{lemma-R3}, and to apply the equality 
$b_1\rhd R_j(b_2,...,b_{j+1}) = L_j(b_1,...,b_j)\rhd b_{j+1}$ 
everywhere. 
\end{proof}


\begin{lemma}
\label{lemma-Re1}
For any $\ell\geq 2$, any $\bfe\in\calE_\ell$ and any $a_1,...,a_\ell\in A$, 
we have
\begin{align*}
R_\ell^{\bfe}(a_1,...,a_\ell) 
& = R_1^{(e_1)}(a_1) \rhd R_{\ell-1}^{(e_2,...,e_\ell)}(a_2,...,a_\ell)
\\ 
& \hspace{-1cm}
+ \sum_{i=1}^{\ell-1}\ R_i^{(e_1,...,e_i)}(a_1,...,a_i) \sotimes 
\Big(a_{i+1}\rhd R_{\ell-i-1}^{(e_{i+2},...,e_\ell)}(a_{i+2},...,a_\ell)\Big) .  
\end{align*}
\end{lemma}

\begin{proof}
The term $j=1$ in the defining recursion (\ref{definition-Re-equation}) gives 
exactly 
$$
R_1^{(e_1)}(a_1) \rhd R_{\ell-1}^{(e_2,...,e_{\ell-1})}(a_2,...,a_\ell), 
$$
se it remains to prove that 
\begin{align}
\nonumber
& \sum_{j=2}^\ell \sum_{\bfp\in \calC_\ell^j} 
\big( R_1^{(e_1)}(a_1) \rhd R_{p_1-1}^{(e_2,...,e_{p_1})}(a_2,...,a_{p_1}) \big) 
\\ \label{lemma-Re1-equation}
& \hspace{1cm}
\sotimes \big( a_{p_1+1}
\rhd R_{p_2-1}^{(e_{p_1+2},...,e_{p_1+p_2})}(a_{p_1+2},...,a_{p_1+p_2}) \big) 
\sotimes \cdots 
\\ \nonumber
& \hspace{2cm}
\cdots \sotimes \big( a_{p_1+\cdots +p_{j-1}+1} \rhd 
R_{p_j-1}^{(e_{p_1+\cdots +p_{j-1}+2},...,e_\ell)}
(a_{p_1+\cdots +p_{j-1}+2},...,a_\ell) \big) 
= 
\\ \nonumber
& \hspace{3cm}
= \sum_{i=1}^{\ell-1}\ R_i^{(e_1,...,e_i)}(a_1,...,a_i) \sotimes 
\Big(a_{i+1}\rhd R_{\ell-i-1}^{(e_{i+2},...,e_\ell)}(a_{i+2},...,a_\ell)\Big) . 
\end{align}
Let us prove this identity by induction. For $\ell=2$ and $3$, it is easy to
verify on the above examples that 
\begin{align*}
R_2^{\bfe}(a,b) 
& = R_1^{(e_1)}(a) \rhd R_1^{(e_2)}(b) 
+ R_1^{(e_1)}(a) \sotimes \big(b\rhd R_0(1)), 
\\[.3cm]
R_3^{\bfe}(a,b,c) 
& = R_1^{(e_1)}(a) \rhd R_2^{(e_2,e_3)}(b,c) 
+ R_1^{(e_1)}(a) \sotimes \big(b\rhd R_1^{(e_3)}(c)\big)
+ R_2^{(e_1,e_2)}(a,b)\sotimes (c \rhd R_0(1)).
\end{align*}
Now suppose it holds up to order $\ell-1$, and let us prove it at order $\ell$. 

Consider the left-hand side of eq.~(\ref{lemma-Re1-equation}). Since $j\geq 2$, 
we can write 
$$
\calC^j_\ell = \bigcup_{i=j-1}^{\ell-1} \calC^{j-1}_i \times \calC^1_{\ell-i}
$$
and decompose $\bfp\in \calC^j_\ell$ into $(p_1,...,p_{j-1})\in \calC^{j-1}_i$ 
and $(p_j)\in \calC^1_{\ell-i}$ for any value $i=j-1,...,\ell-1$. 
We then have $p_1+\cdots +p_{j-1}=i$ and $p_j=\ell-(p_1+\cdots +p_{j-1})=\ell-i$. 
Therefore the left-hand side can be written as 
\begin{align*}
& \sum_{j=2}^\ell \sum_{i=j-1}^{\ell-1} \sum_{(p_1,...,p_{j-1})\in \calC_i^{j-1}} 
\delta_{e_1,1}\ \big( a_1 \rhd R_{p_1-1}^{(e_2,...,e_{p_1})}(a_2,...,a_{p_1}) \big) 
\sotimes \cdots 
\\ 
& \hspace{2cm}
\cdots \sotimes \big( a_{P_{j-2}+1}
\rhd R_{p_{j-1}-1}^{(e_{p_{j-2}+2},...,e_{p_{j-1}})}(a_{P_{j-2}+2},...,a_i) \big) \sotimes 
\\ 
& \hspace{3cm}
\sotimes \big( a_{i+1} \rhd 
R_{\ell-i-1}^{(e_{i+2},...,e_\ell)}(a_{i+2},...,a_\ell) \big) = 
\\[.3cm] 
& \hspace{.5cm}
= \sum_{i=1}^{\ell-1}\  
\left( \sum_{k=1}^i \sum_{(p_1,...,p_k)\in \calC_i^k} 
\delta_{e_1,1}\ \big( a_1 \rhd R_{p_1-1}^{(e_2,...,e_{p_1-1})}(a_2,...,a_{p_1}) \big) 
\sotimes \cdots \right. 
\\ 
& \hspace{2cm}
\left. \cdots \sotimes \big( a_{P_{k-1}+1}
\rhd R_{p_k-1}^{(e_{P_{k-1}+2},...,e_{P_k})}(a_{P_{k-1}+2},...,a_i) \big) \right) \sotimes 
\\ 
& \hspace{3cm}
\sotimes \big( a_{i+1} \rhd 
R_{\ell-i-1}^{(e_{i+2},...,e_\ell)}(a_{i+2},...,a_\ell) \big), 
\end{align*}
where $P_k=p_1+\cdots +p_k$. 
Applying the inductive hypothesis to the sum over $k$ leads to the result. 
\end{proof}


\begin{theorem}
\label{Diff(A)-proalgebraic-loop}
The associative algebra $\HFdBex$ is indeed a coloop bialgebra and 
represents the loop of formal diffeomorphisms $\Diff$ as a functor 
$\Diff:\As \longrightarrow \Loop$. 

As a consequence, given an associative algebra $A$, a series 
$a=\sum_{n\geq 0} a_n\,\lambda^{n+1} \in \Diff(A)$ can be seen as an algebra 
homomorphism $a:\HFdBex\longrightarrow A$ defined on the generators 
of $\HFdBex$ by $a(x_n)=a_n$, and the right and left division $a\slash b$ and 
$a\backslash b$ are given at any order $n$ by the following closed formulas: 
\begin{align*}
(a\slash b)_n & = \mu_A\,(a\scoprod b)\,\delta_r(x_n) 
\\ 
& = a_n-b_n + \sum_{\ell=1}^{n-1} (-1)^\ell \sum_{\bfn \in \calC_n^{\ell+1}} 
d_\ell(n_1,...,n_\ell)\ 
(a_{n_1}-b_{n_1})\, b_{n_2}\cdots b_{n_{\ell+1}} , 
\\ 
(a\backslash b)_n & = \mu_A\,(a\scoprod b)\,\delta_l(x_n) 
\\
& = b_n-a_n + \sum_{\ell=2}^{n-1} (-1)^\ell \!\sum_{\bfn\in \calC_{n}^{\ell+1}} 
\sum_{\bfe\in \calE_\ell} (-1)^{\bfe}\ d_\ell^{\bfe}(n_1,...,n_\ell)\ 
a_{n_1} c^{(e_1)}_{n_2}\cdots c^{(e_{\ell-1})}_{n_\ell}\, (b_{n_{\ell+1}}-a_{n_{\ell+1}}) ,
\end{align*}
where $c^{(e_i)}_{n_i}= a_{n_i}$ if $e_i=1$ and $c^{(e_i)}_{n_i}= b_{n_i}$ if $e_i=2$.  
\end{theorem}

\begin{proof}
The free associative algebra $\HFdBex$ clearly represents the sets 
$\Diff(A)$ over associative algebras $A$, and the comultiplication $\DFdBex$ 
is just the Fa\`a di Bruno comultiplication $\DFdB$ seen with values in 
$\HFdB\scoprod \HFdB$ instead of $\HFdB\otimes \HFdB$, therefore 
it clearly represents the loop law given in Definition 
\ref{definition-Diff-loop}. 
Thus, the theorem is proved if we show that $\HFdBex$ is indeed a coloop 
bialgebra. 

The comultiplication $\DFdBex$ satisfies the compatibility relation with 
the standard counit, because $\DFdB$ does, and coassociativity is not required. 
So it remains to check that the codivisions $\delta_r$ and $\delta_l$ 
given in Def.~\ref{definition-Diff-loop} satisfy the identities 
(\ref{right-cocancellation}) and (\ref{left-cocancellation}). Since these maps 
are algebra morphisms, it suffices to verify these identities on the 
generators $x_n$, for any $n\geq 1$. 
\bigskip 

\noindent
i) Let us start with the right codivision and show that it satisfies  
the first identity (\ref{right-cocancellation}), namely
$$ 
(\Id\scoprod \mu)\,  (\delta_r\scoprod\Id)\,  \DFdBex(x_n) = x_n,
$$ 
which explicitely gives the recurrence (with $u_n=x_n-y_n$) 
\begin{align}
\label{right-codivision-recursion-a}
\delta_r(x_n) 
&= u_n -\sum_{m=1}^{n-1} \sum_{\ell=1}^{n-m} \sum_{\bfk\in \calC_{n-m}^\ell} 
\delta_r(x_m)\rhd \big(y_{k_1}\sotimes \cdots \sotimes y_{k_\ell} \big). 
\end{align}
Expanding $\delta_r(x_n)$ in terms of the left recursive opeators, 
this equation becomes
\begin{align*}
& 
u_n + \sum_{\ell=1}^{n-1} (-1)^\ell \sum_{\bfn\in \calC_n^{\ell+1}} 
L_\ell(u_{n_1},y_{n_2},...,y_{n_\ell}) \rhd y_{n_{\ell+1}}
\\ 
& \hspace{1cm}
= u_n -\sum_{m=1}^{n-1} \sum_{j=1}^{n-m} \sum_{\bfq\in \calC_{n-m}^j} 
\sum_{i=0}^{m-1} (-1)^i \sum_{\bfp\in \calC_m^{i+1}} 
\Big( L_i(u_{p_1},y_{p_2},...,y_{p_i}) \rhd y_{p_{i+1}}\Big) 
\rhd \big(y_{q_1}\sotimes \cdots \sotimes y_{q_j} \big) 
\\ 
& \hspace{1cm}
= u_n +\sum_{\ell=i+j=1}^{n-1} \sum_{i=0}^{\ell-1} (-1)^{i+1} \sum_{m=i+1}^{n-1} 
\underset{\bfq\in \calC_{n-m}^{\ell-i}}{\sum_{\bfp\in \calC_m^{i+1}}} 
\Big( L_i(u_{p_1},y_{p_2},...,y_{p_i})  \rhd y_{p_{i+1}}\Big) 
\rhd \big(y_{q_1}\sotimes \cdots \sotimes y_{q_{\ell-i}} \big) . 
\end{align*}
Now, since 
$$
\bigcup_{m=i+1}^{n-1} \calC_m^{i+1} \times \calC_{n-m}^{\ell-i} 
\cong \calC_{n}^{\ell+1}, 
$$
let us call $\bfn=(\bfp,\bfq)$, that is,
$$
(n_1,n_2,...,n_{\ell+1}) = (p_1,...,p_{i+1},q_1,...,q_j). 
$$
Then, the recursion (\ref{right-codivision-recursion-a})
is equivalent, for any $n\geq 1$, any $1\leq \ell\leq n-1$ and any 
$\bfn\in \calC_{n}^{\ell+1}$, to the equation
\begin{align*}
& L_\ell(u_{n_1},y_{n_2},...,y_{n_\ell}) \rhd y_{n_{\ell+1}}
\\ 
& \hspace{2cm}
= \sum_{i=0}^{\ell-1} (-1)^{\ell-1-i} 
\big( L_i(u_{n_1},y_{n_2},...,y_{n_i})  \rhd y_{n_{i+1}}\big) 
\rhd (y_{n_{i+2}}\sotimes \cdots \sotimes y_{n_{\ell+1}} ) 
\\ 
& \hspace{2cm}
= \sum_{i=0}^{\ell-1} (-1)^{\ell-1-i} 
\Big( \big( L_i(u_{n_1},y_{n_2},...,y_{n_i})  \rhd y_{n_{i+1}}\big) 
\sotimes y_{n_{i+2}}\sotimes \cdots \sotimes y_{n_\ell} \Big) 
\rhd y_{n_{\ell+1}} , 
\end{align*}
which holds by definition of $L_\ell$. 

The second identity (\ref{right-cocancellation}), namely 
$$
(\Id\scoprod\mu)\,  (\DFdBex\scoprod\Id)\,  \delta_r(x_n) = x_n,   
$$ 
is better developed using the expansion over the right recursive operators, 
and explicitely gives the recurrence 
\begin{align}
\label{right-codivision-recursion-b}
& \sum_{\ell=1}^{n-1} \sum_{\bfn\in \calC^{\ell+1}_n} (-1)^\ell\ 
x_{n_1}\rhd R_\ell(y_{n_2},...,y_{n_{\ell+1}}) 
\\
\nonumber
& \hspace{1cm} 
= \sum_{m=1}^{n-1} 
\sum_{p=1}^{n-m} \sum_{i=1}^{p} \sum_{\bfp\in \calC^i_p} 
\sum_{j=0}^{n-m-p} \sum_{\bfq\in \calC^j_{n-m-p}} (-1)^{j+1} 
\big(x_m\rhd (y_{p_1}\sotimes\cdots\sotimes y_{p_i})\big) 
\rhd R_j(y_{q_1},...,y_{q_j}) . 
\end{align}
Rewriting the sums in terms of $m=1,...,n-1$, $\ell=i+j=1,...,n-m$, 
$i=1,...,\ell$ and $j=\ell-i$, this gives a sum over 
$\bfp\in\calC_p^{i}$ and $\bfq\in\calC_{n-m-p}^{\ell-i}$ for 
$p=i,...,n-m$. That is, we get a sum over $\bfk=(\bfp,\bfq)\in \calC_{n-m}^\ell$
and consequently a sum over $\bfn=(m,\bfk)\in \calC_n^{\ell+1}$: 
\begin{align*}
& \sum_{\ell=1}^{n-1} \sum_{\bfn\in \calC^{\ell+1}_n} (-1)^\ell\ 
x_{n_1}\rhd R_\ell(y_{n_2},...,y_{n_{\ell+1}}) 
\\
& \hspace{1cm} 
= \sum_{m=1}^{n-1} 
\sum_{\ell=1}^{n-m} \sum_{i=1}^{\ell}
\sum_{p=i}^{n-m} \sum_{\bfp\in \calC^i_p} \sum_{\bfq\in \calC^{\ell-i}_{n-m-p}} 
(-1)^{\ell-i+1} 
\big(x_m\rhd (y_{p_1}\sotimes\cdots\sotimes y_{p_i})\big) 
\rhd R_{\ell-i}(y_{q_1},...,y_{q_{\ell-i}}) 
\\
& \hspace{1cm} 
= \sum_{m=1}^{n-1} 
\sum_{\ell=1}^{n-m} \sum_{\bfk\in \calC^\ell_{n-m}} 
\sum_{i=1}^{\ell} (-1)^{\ell-i+1} 
\big(x_m\rhd (y_{k_1}\sotimes\cdots\sotimes y_{k_i})\big) 
\rhd R_{\ell-i}(y_{k_{i+1}},...,y_{k_\ell}) 
\\
& \hspace{1cm} 
= \sum_{\ell=1}^{n-1} \sum_{\bfn\in \calC^{\ell+1}_n} (-1)^\ell
\sum_{i=1}^{\ell} (-1)^{i-1} 
\big(x_{n_1}\rhd (y_{n_2}\sotimes\cdots\sotimes y_{n_{i+1}})\big) 
\rhd R_{\ell-i}(y_{n_{i+2}},...,y_{n_{\ell+1}}) .  
\end{align*}
Therefore, for any $n\geq 2$, any $\ell=1,..., n-1$ and any sequence 
$\bfn\in \calC^{\ell+1}_n$, eq.~(\ref{right-codivision-recursion-b})
is equivalent to the recurrence equation 
\begin{align*}
x_{n_1}\rhd R_\ell(y_{n_2},...,y_{n_{\ell+1}}) 
& = \sum_{i=1}^{\ell} (-1)^{i-1} 
\big(x_{n_1}\rhd (y_{n_2}\sotimes\cdots\sotimes y_{n_{i+1}})\big) 
\rhd R_{\ell-i}(y_{n_{i+2}},...,y_{n_{\ell+1}}), 
\end{align*}
which is proved in Lemma \ref{lemma-R3}. 
\bigskip 

\noindent
ii) Let us show now that the left codivision given in 
Def.~(\ref{FaaDiBrunoColoopBialgebra}) satisfies the identities 
(\ref{left-cocancellation}). 
The first identity (\ref{left-cocancellation}), namely 
$$ 
(\mu\scoprod \Id)\,  (\Id\scoprod \delta_l)\,  \DFdBex(x_n) = y_n,
$$ 
explicitely gives the recurrence (with $v_n=y_n-x_n$)
\begin{align}
\label{left-codivision-recursion-a}
\delta_l(x_n)  
&= v_n - \sum_{m=1}^{n-1}\ \sum_{\lambda=1}^{n-m}\ \sum_{\bfk\in \calC_{n-m}^\lambda} 
x_m\rhd \big(\delta_l(y_{k_1})\sotimes\cdots \sotimes \delta_l(y_{k_\lambda})\big), 
\end{align}
where $y_k$ is just the $k$th generator $x_k$ in the second copy of the free 
product algebra $\HFdBex \scoprod \HFdBex$, therefore the formula for 
$\delta_l(y_k)$ is just the same as for $\delta_l(x_k)$. 

To show this, we consider the expansion
(\ref{FaaDiBruno-operators-left-codivision}) of $\delta_l(x_n)$ given in
Thm.~\ref{FaaDiBruno-operators}. Since $R_\ell^{\bfe}=0$ when $e_1=2$,
we have $x_{n_1}^{(e_1)}=x_{n_1}$ and we can rewrite
(\ref{FaaDiBruno-operators-left-codivision}) as 
\begin{align*}
\delta_l(x_n)  
& = v_n + \!\sum_{m=1}^{n-1} x_m\rhd \! 
\sum_{\ell=1}^{n-m} (-1)^\ell \!\! 
\underset{\bfe\in \calE_\ell}{\sum_{\bfn \in \calC_{n-m}^\ell}} (-1)^{\bfe}  
R_\ell^{\bfe}(x^{(e_2)}_{n_1},...,x^{(e_\ell)}_{n_{\ell-1}},v_{n_\ell}). 
\end{align*}
Then eq.~(\ref{left-codivision-recursion-a}) is clearly verified for $n=1$, 
because $\delta_l(x_1)=v_1$, and for any $n\geq 2$ and any $m=1,...,n-1$, it is 
equivalent to the equation
\begin{align}
\label{left-codivision-recursion-a-1}
\sum_{\lambda=1}^\mu \sum_{\bfk\in \calC_\mu^\lambda} 
\delta_l(x_{k_1})\sotimes\cdots \sotimes \delta_l(x_{k_\lambda}) 
& = 
\\ 
\nonumber
& \hspace{-4cm}
= v_\mu - \sum_{\ell=2}^\mu (-1)^\ell \sum_{\bfn \in \calC_\mu^\ell}\  
\sum_{\bfe\in \calE_\ell} (-1)^{\bfe} \ 
R_\ell^{\bfe}(x^{(e_2)}_{n_1},...,x^{(e_\ell)}_{n_{\ell-1}},v_{n_\ell}), 
\end{align}
for any $\mu=n-m=1,...,n-1$.

Let us prove this equation by induction on $\mu$. 
For $\mu=1$ we again have $\delta_l(x_1)=v_1$. 
So, suppose that eq.~(\ref{left-codivision-recursion-a-1}) holds up to order 
$\mu-1$ and prove it at order $\mu$. 

On the left-hand side of eq.~(\ref{left-codivision-recursion-a-1}), 
we separate the term $\lambda=1$ and observe that, for $\lambda\geq 2$, 
we can decompose $\bfk=(k_1,...,k_{\lambda-1},k_\lambda)\in \calC_\mu^\lambda$ 
into $(\bfq,\nu) \in \calC_{\mu-\nu}^{\lambda-1} \times \calC_\nu^1$ with 
\begin{align*}
q_i &= k_i \qquad\mbox{for $i=1,...,\lambda-1$} 
\\
\nu &= k_\lambda .
\end{align*}
Since 
$$
\calC_\mu^\lambda 
= \bigcup_{\nu=1}^{\mu-\lambda+1} \calC_{\mu-\nu}^{\lambda-1} \times \calC_\nu^1, 
$$
the left-hand side of eq.~(\ref{left-codivision-recursion-a-1}) 
can then be written as 
\begin{align*} 
& \delta_l(x_\mu) + \sum_{\nu=1}^{\mu-1} 
\left( \sum_{i=1}^{\mu-\nu} \sum_{\bfq\in \calC_{\mu-\nu}^i} 
\delta_l(x_{q_1})\sotimes\cdots \sotimes \delta_l(x_{q_i}) \right) 
\sotimes \delta_l(x_\nu) . 
\end{align*}
We then apply the inductive hypothesis (\ref{left-codivision-recursion-a-1})
to the sum over $i=1,...,\mu-\nu$, and expand the single factors
$\delta_l(x_\mu)$ and $\delta_l(x_\nu)$ as in
(\ref{FaaDiBruno-operators-left-codivision}), thus obtaining
\begin{align}
\nonumber
\sum_{\lambda=1}^\mu \sum_{\bfk\in \calC_\mu^\lambda} 
\delta_l(x_{k_1})\sotimes\cdots \sotimes \delta_l(x_{k_\lambda}) 
& = 
\\ \label{left-codivision-recursion-a-2}
& \hspace{-4.5cm}
= v_\mu + \sum_{\lambda=1}^{\mu-1} (-1)^\lambda \sum_{\bfn \in \calC_\mu^{\lambda+1}} 
\sum_{\bfe\in \calE_\lambda} (-1)^{\bfe} \ 
x^{(e_1)}_{n_1}\rhd R_\lambda^{\bfe}(x^{(e_2)}_{n_2},...,x^{(e_\lambda)}_{n_\lambda},v_{n_{\lambda+1}})
\\ \nonumber
& \hspace{-4cm}
+ \sum_{\nu=1}^{\mu-1} \left( 
v_{\mu-\nu} - \sum_{i=2}^{\mu-\nu} (-1)^i \sum_{\bfp \in \calC_{\mu-\nu}^i} 
\sum_{\bfe'\in \calE_i} (-1)^{\bfe'} \ 
R_i^{\bfe'}(x^{(e_2')}_{p_1},...,x^{(e_i')}_{p_{i-1}},v_{p_i}) 
\right) \sotimes 
\\  \nonumber
& \hspace{-3cm}
\sotimes \left( 
v_\nu + \sum_{j=1}^{\nu-1} (-1)^j \sum_{\bfq \in \calC_\nu^{j+1}} 
\sum_{\bfe''\in \calE_j} (-1)^{\bfe''} \ 
x^{(e_1'')}_{q_1}\rhd R_j^{\bfe''}(x^{(e_2'')}_{q_2},...,x^{(e_j'')}_{q_j},v_{q_{j+1}}) 
\right) . 
\end{align}
Finally, it remains to prove that the right-hand side of
eq.~(\ref{left-codivision-recursion-a-1}) coincides with the right-hand side of
eq.~(\ref{left-codivision-recursion-a-2}). The first term $v_\mu$ appears
in both formulas, let us compare the other terms.

The first term in eq.~(\ref{left-codivision-recursion-a-2}) is
\begin{align*}
A & = \sum_{\lambda=1}^{\mu-1} (-1)^\lambda \sum_{\bfn \in \calC_\mu^{\lambda+1}} 
\sum_{\bfe''\in \calE_\lambda} (-1)^{\bfe''} \ x_{n_1}^{(e_1'')}\rhd
R_\lambda^{\bfe''}(x^{(e_2'')}_{n_2},...,x^{(e_\lambda'')}_{n_\lambda},v_{n_{\lambda+1}}) . 
\end{align*}
We apply the trick
\begin{align}
\label{trick1}
a & = R_1^{(1)}(a) = \sum_{(e')\in \calE_1} R_1^{(e')}(a)
\end{align}
to the element $x_{n_1}^{(e_1'')}$, then set $\ell = \lambda+1$ and
$\bfe = (e',\bfe'')=(e',e_1'',...,e_\lambda'')$, and get 
\begin{align*}
A & = - \sum_{\ell=2}^\mu (-1)^\ell \sum_{\bfn \in \calC_\mu^\ell} 
\sum_{\bfe\in \calE_\ell} (-1)^{\bfe} \ 
\left( R_1^{(e_1)}(x^{(e_2)}_{n_1}) \rhd
R_{\ell-1}^{(e_2,...,e_\ell)}(x^{(e_3)}_{n_2},...,x^{(e_\ell)}_{n_{\ell-1}},v_{n_\ell}) 
\right).  
\end{align*}
The second term in eq.~(\ref{left-codivision-recursion-a-2}) is
\begin{align*}
B & = \sum_{\nu=1}^{\mu-1} v_{\mu-\nu}\sotimes v_\nu.
\end{align*}
We apply the second trick
\begin{align}
\nonumber
v_{n} & = -(x^{(1)}_n-x^{(2)}_n) = -\sum_{(e')\in\calE_1} (-1)^{e'-1} x^{(e')}_n
\\ \label{trick2}
& = -\sum_{\bfe=(e_1,e_2)\in\calE_2} (-1)^{\bfe} R_1^{(e_1)}(x^{(e_2)}_n)
\end{align}
to the element $v_{\mu-\nu}$, and get 
\begin{align*}
B & = - \sum_{\nu=1}^{\mu-1}\ \sum_{\bfe\in\calE_2} (-1)^{\bfe}\ 
R_1^{(e_1)}(x^{(e_2)}_n) \sotimes v_\nu
\\ 
& = - \sum_{\bfn \in \calC_\mu^2}\ \sum_{\bfe\in \calE_2} (-1)^{\bfe} \ 
R_1^{(e_1)}(x^{(e_1)}_{n_1}) \sotimes v_{n_2}.  
\end{align*}
Using again (\ref{trick2}), the third term in
eq.~(\ref{left-codivision-recursion-a-2}) becomes 
\begin{align*}
C & = \sum_{\nu=1}^{\mu-1} \sum_{j=2}^{\nu-1} (-1)^j \sum_{\bfq \in \calC_\nu^{j+1}} 
\sum_{\bfe'''\in \calE_j} (-1)^{\bfe'''} v_{\mu-\nu}\sotimes \Big(x^{(e_1''')}_{q_1}\rhd
R_j^{\bfe'''}(x^{(e_2''')}_{q_2},...,x^{(e_j''')}_{q_j},v_{q_{j+1}})\Big)
\\ 
& = - \sum_{\nu=1}^{\mu-1} \sum_{j=2}^{\nu-1} (-1)^j \sum_{\bfq \in \calC_\nu^{j+1}} 
\underset{\bfe'''\in \calE_j}{\underset{\bfe''\in \calE_1}{\sum_{\bfe'\in \calE_1}}}
(-1)^{\bfe'} (-1)^{\bfe''} (-1)^{\bfe'''} \
R_1^{(e')}(x^{(e'')}_{\mu-\nu})\sotimes \Big(x^{(e_1''')}_{q_1}\rhd
R_j^{\bfe'''}(x^{(e_2''')}_{q_2},...,x^{(e_j''')}_{q_j},v_{q_{j+1}})\Big). 
\end{align*}
We set $\ell=1+1+j$, $\bfn=(\mu-\nu,\bfq)\in \calC_\mu^\ell$ and
$\bfe=(e',e'',\bfe''')\in \calE_\ell$, and obtain 
\begin{align*}
C & = - \sum_{\ell=3}^{\mu-1} (-1)^\ell \sum_{\bfn \in \calC_\mu^\ell}
\sum_{\bfe\in \calE_\ell} (-1)^{\bfe}\  
R_1^{(e_1)}(x^{(e_2)}_{n_1})\sotimes \Big(x_{n_2}^{(e_3)}\rhd 
R_{\ell-2}^{(e_3,...,e_\ell)}(x^{(e_4)}_{n_3},...,x^{(e_\ell)}_{n_{\ell-1}},v_{n_\ell}) 
\Big) .
\end{align*}
The fourth term in eq.~(\ref{left-codivision-recursion-a-2}) is
\begin{align*}
D & = - \sum_{\nu=1}^{\mu-1} \sum_{i=2}^{\mu-\nu} (-1)^i \sum_{\bfp \in \calC_{\mu-\nu}^i} 
\sum_{\bfe'\in \calE_i} (-1)^{\bfe'} \ 
R_i^{\bfe'}(x^{(e_2')}_{p_1},...,x^{(e_i')}_{p_{i-1}},v_{p_i}) \sotimes v_\nu. 
\end{align*}
We write $v_{p_i}=-\sum_{\bfe''\in\calE_1} (-1)^{\bfe''} x^{(e'')}_n$ using
(\ref{trick1}), and set $\ell=i+1$, $\bfn=(\bfp,\nu)\in \calC_\mu^\ell$
and $\bfe=(\bfe',\bfe'')\in \calE_\ell$. Then we have 
\begin{align*}
D & = - \sum_{\ell=3}^\mu (-1)^\ell \sum_{\bfn \in \calC_\mu^\ell}\ 
\sum_{\bfe\in \calE_\ell} (-1)^{\bfe} \ \sum_{i=2}^{\ell-1}
R_i^{(e_1,...,e_i)}(x^{(e_2)}_{n_1},...,x^{(e_{\ell-1})}_{n_i}) \sotimes v_\ell. 
\end{align*}
With similar manipulations, setting $\ell=i+j+1$ and
$\bfn=(\bfp,\bfq)\in C_\mu^\ell$, the last term in
eq.~(\ref{left-codivision-recursion-a-2}) is
\begin{align*}
E & = - \sum_{\nu=1}^{\mu-1} \sum_{i=2}^{\mu-\nu} \sum_{j=2}^{\nu-1} (-1)^{i+j}
\sum_{\bfp \in \calC_{\mu-\nu}^i} \sum_{\bfq \in \calC_\nu^{j+1}}
\sum_{\bfe'\in \calE_i}\sum_{\bfe'''\in \calE_j}  (-1)^{\bfe'} (-1)^{\bfe'''}  
\\ 
& \hspace{2cm}
R_i^{\bfe'}(x^{(e_2')}_{p_1},...,x^{(e_i')}_{p_{i-1}},v_{p_i}) \sotimes
\Big(x_{q_1}\rhd R_j^{\bfe'''}(x^{(e_2''')}_{q_2},...,x^{(e_j''')}_{q_j},v_{q_{j+1}})\Big) 
\\[.2cm]
& = - \sum_{\ell=4}^\mu (-1)^\ell \sum_{\bfn \in \calC_\mu^\ell}
\sum_{\bfe\in \calE_\ell} (-1)^{\bfe} \sum_{i=2}^{\ell-2}\ 
R_i^{(e_1,...,e_i)}(x^{(e_2)}_{n_1},...,x^{(e_{i+1})}_{n_i}) \sotimes 
\\ 
& \hspace{2cm}
\sotimes \Big(x^{(e_{i+2})}_{n_{i+1}}\rhd 
R_{\ell-i-1}^{(e_{i+2},...,e_\ell)}(x^{(e_{i+3})}_{n_{i+2}},...,x^{(e_\ell)}_{n_{\ell-1}},
v_{n_\ell})\Big) .  
\end{align*}
We now observe that the sum $B$ extends $C$ to the value $\ell=2$, and that
$D$ extends $E$ to the value $\ell=3$. Alltogether, we have  
\begin{align*}
B+C+D+E 
& = - \sum_{\ell=3}^\mu (-1)^\ell \sum_{\bfn \in \calC_\mu^\ell}
\sum_{\bfe\in \calE_\ell} (-1)^{\bfe} \sum_{i=1}^{\ell-1}\ 
R_i^{(e_1,...,e_i)}(x^{(e_2)}_{n_1},...,x^{(e_{i+1})}_{n_i}) \sotimes 
\\ 
& \hspace{2cm}
\sotimes \Big(x^{(e_{i+2})}_{n_{i+1}}\rhd R_{\ell-i-1}^{(e_{i+2},...,e_\ell)}
(x^{(e_{i+3})}_{n_{i+2}},...,x^{(e_\ell)}_{n_{\ell-1}},v_{n_\ell})\Big) .  
\end{align*}
Therefore, eq.~(\ref{left-codivision-recursion-a-1}) is then equivalent, 
for any $\ell\geq 3$, any $\bfn\in\calC_\nu^\ell$ and any $\bfe\in\calE_\ell$, 
to the following recursion 
\begin{align*}
R_\ell^{\bfe}(x^{(e_2)}_{n_1},...,x^{(e_\ell)}_{n_{\ell-1}},v_{n_\ell}) 
& = R_1^{(e_1)}(x^{(e_2)}_{n_1})\rhd 
R_{\ell-1}^{(e_2,...,e_\ell)}(x^{(e_3)}_{n_2},...,x^{(e_\ell)}_{n_{\ell-1}},v_{n_\ell}) 
\\ 
& \hspace{-2cm}
+ \sum_{i=1}^{\ell-1}\ 
R_i^{(e_1,...,e_i)}(x^{(e_2)}_{n_1},...,x^{(e_{i+1})}_{n_i}) 
\sotimes \Big(x^{(e_{i+2})}_{n_{i+1}}\rhd R_{\ell-i-1}^{(e_{i+2},...,e_\ell)}
(x^{(e_{i+3})}_{n_{i+2}},...,x^{(e_\ell)}_{n_{\ell-1}},v_{n_\ell})\Big) ,   
\end{align*}
which is proved in Lemma \ref{lemma-Re1}. 
\bigskip 

The second identity (\ref{left-cocancellation}), namely 
$$
(\mu\scoprod\Id)\,  (\Id\scoprod\DFdBex)\,  \delta_l(x_n) = y_n,   
$$ 
can not be expressed as a recurrence on $R_\ell^{\bfe}$, because these operators 
do not show up explicitely to which factor of $\HFdBex\scoprod\HFdBex$ 
the variables belong. 
Then, let us use the recursion (\ref{left-codivision-recursion-a}) to 
describe $\delta_l$ and prove the second identity by induction on $n$.  

The identity is verified for $n=1$ because we have 
\begin{align*}
(\mu\scoprod\Id)\,  (\Id\scoprod\DFdBex)\,  \delta_l(x_1) 
&= x_1+y_1-x_1 = y_1. 
\end{align*}
Then, suppose it holds up to the degree $n-1$
Since $\DFdBex$ and $\mu$ are algebra homomorphisms, if we apply the 
operator $D=(\mu\scoprod\Id)\,  (\Id\scoprod\DFdBex)$ to the expression 
\begin{align*}
\delta_l(x_n)  
&= v_n - \sum_{m=1}^{n-1}\ \sum_{\lambda=1}^{n-m}\ \sum_{\bfk\in \calC_{n-m}^\lambda} 
x_m\rhd \big(\delta_l(y_{k_1})\sotimes\cdots \sotimes \delta_l(y_{k_\lambda})\big), 
\end{align*}
we obtain, for $D\big(\delta_l(x_n)\big)$, the sum of 
\begin{align*}
D(v_n) 
& = y_n + \sum_{m=1}^{n-1}\ \sum_{\lambda=1}^{n-m}\ \sum_{\bfk\in \calC_{n-m}^\lambda} 
x_{m}\rhd (y_{k_1} \sotimes\cdots\sotimes y_{k_\ell}) 
\end{align*}
and of 
\begin{align*}
- \sum_{m=1}^{n-1}\ \sum_{\lambda=1}^{n-m}\ \sum_{\bfk\in \calC_{n-m}^\lambda} 
x_m\rhd \Big( D\big(\delta_l(y_{k_1})\big) \sotimes\cdots 
\sotimes D\big(\delta_l(y_{k_\lambda})\big)\Big).
\end{align*}
Therefore the second identity is satisfied if, for any $m=1,...,n-1$, any 
$\lambda=1,...,n-m$ and any $\bfk\in \calC_{n-m}^\lambda$, we have 
\begin{align*}
D\big(\delta_l(y_{k_1})\big) \sotimes\cdots 
\sotimes D\big(\delta_l(y_{k_\lambda})
& = y_{k_1} \sotimes\cdots\sotimes y_{k_\ell}, 
\end{align*}
which is true by inductive hypothesis. 
\end{proof}


\subsection{Properties of the diffeomorphisms loop}

\begin{proposition}
\label{proposition-antipode}
The coloop bialgebra $\HFdBex$ has a two-sided antipode $S$ such that 
\begin{align*}
\delta_r & = (\Id\scoprod S)\,  \DFdBex,
\end{align*}
while the identity $\delta_l = (S \scoprod \Id)\,  \DFdBex$ does not hold.
Moreover, the antipode in the Fa\`a di Bruno coloop bialgebra coincides with 
that in the non-commutative Fa\`a di Bruno Hopf algebra given in \cite{BFK}, 
that is, 
\begin{align*}
S(x_n) 
& = - \sum_{\ell=0}^{n-1} (-1)^\ell \sum_{\bfn \in \calC_n^{\ell+1}} 
d_{\ell+1}(n_1,...,n_{\ell+1})\ x_{n_1}\, x_{n_2}\cdots x_{n_{\ell+1}} . 
\end{align*}
\end{proposition}

\begin{proof}
i) In a coloop bialgebra, the left and right antipodes are given respectively 
by 
$$
S_l=(\Id\scoprod \varepsilon)\,  \delta_l
\qquad\mbox{and}\qquad 
S_r=(\varepsilon\scoprod \Id)\,  \delta_r, 
$$
cf.~(\ref{antipode-left-right}). 
Let us show that for $\HFdBex$ these two operators coincide, and therefore 
the two-sided antipode is well defined by $S:=S_l=S_r$. 

Indeed, let us fix $n\geq 1$. For the right antipode we have
\begin{align*}
S_r(x_n) 
& = \sum_{\ell=0}^{n-1} (-1)^\ell \sum_{\bfn \in \calC_n^{\ell+1}} 
d_{\ell+1}(n_1,...,n_{\ell+1})\ 
\big(\varepsilon(x_{n_1})-y_{n_1}\big)\, y_{n_2}\cdots y_{n_{\ell+1}} 
\\
& = - \sum_{\ell=0}^{n-1} (-1)^\ell \sum_{\bfn \in \calC_n^{\ell+1}} 
d_{\ell+1}(n_1,...,n_{\ell+1})\ x_{n_1}\, x_{n_2}\cdots x_{n_{\ell+1}} , 
\end{align*}
where $n_1>0$ implies $\varepsilon(x_{n_1})=0$, and where we renamed 
the variables $y$ as $x$ because $S_r$ takes values in $\HFdBex$. 
For the left antipode we have 
\begin{align*}
S_l(x_n) 
& = \varepsilon(y_n)-x_n 
- \sum_{\bfn \in \calC_n^2} d_1(n_1)\ x_{n_1} \big(\varepsilon(y_{n_2})-x_{n_2}\big)
\\
& \hspace{1cm}
+ \sum_{\ell=2}^{n-1} (-1)^\ell \!\sum_{\bfn\in \calC_{n}^{\ell+1}} 
\sum_{\bfe\in \calE_{\ell-1}} (-1)^{\bfe}\ d_{\ell+1}^{\bfe}(n_1,...,n_{\ell+1})\ 
(\Id\scoprod \varepsilon) \big(x_{n_1} x^{(e_1)}_{n_2}\cdots x^{(e_{\ell-1})}_{n_\ell}\, 
v_{n_{\ell+1}}\big).
\end{align*}
Since $\varepsilon$ kills the terms where some $y$ appears, in the sum over 
the sequences $\bfe\in \calE_{\ell-1}$ there only remains the sequence 
$\bfe=(1,1,...,1)$, for which $d_{\ell+1}^{\bfe}=d_{\ell+1}$ and $(-1)^{\bfe}=+$, 
and therefore we have 
\begin{align*}
S_l(x_n) & = -x_n + \sum_{\bfn \in \calC_n^2} d_1(n_1)\ x_{n_1}\,x_{n_2}
\\
& \hspace{1cm}
- \sum_{\ell=2}^{n-1} (-1)^\ell \!\sum_{\bfn\in \calC_{n}^{\ell+1}} 
\ d_{\ell+1}(n_1,...,n_{\ell+1})\ x_{n_1} x_{n_2}\cdots x_{n_{\ell+1}} 
\\ 
& = S_r(x_n). 
\end{align*}

ii) Let us now prove the identity $\delta_r = (\Id\scoprod S)\,  \DFdBex$.  
For any generator $x_n$ of $\HFdBex$, we have
\begin{align*}
(\Id\scoprod S)\,  \DFdBex(x_n) 
& = x_n + S(y_n) 
+ \sum_{j=1}^{n-1} \sum_{\bfn\in \calC_n^{j+1}} \binom{n_1+1}{j}\ 
x_{n_1} \ S(y_{n_2})\cdots S(y_{n_{j+1}}) 
\\ 
& \hspace{-3cm}
= x_n - y_n - \sum_{\ell=1}^{n-1} (-1)^\ell \sum_{\bfn \in \calC_n^{\ell+1}} 
d_{\ell+1}(n_1,...,n_{\ell+1})\ y_{n_1}\, y_{n_2}\cdots y_{n_{\ell+1}}
\\
& \hspace{-2cm}
+ \sum_{j=1}^{n-1} \sum_{\bfn\in \calC_n^{j+1}} 
\sum_{p_1=1}^{n_2} \! \cdots \! \sum_{p_j=1}^{n_{j+1}} 
(-1)^{p_1+\cdots +p_j} 
\sum_{\bfq^1\in \calC_{n_2}^{p_1}} \!\! \cdots \!\! \sum_{\bfq^j\in \calC_{n_{j+1}}^{p_j}}
\binom{n_1+1}{j}\ d_{p_1}(q^1_1,...,q^1_{p_1}) \cdots 
\\
& \hspace{2cm}
\cdots d_{p_j}(q^j_1,...,q^j_{p_j})\
x_{n_1} \ y_{q^1_1} \cdots y_{q^1_{p_1}} \cdots y_{q^j_1} \cdots y_{q^j_{p_j}} . 
\end{align*}
Set $\ell=p_1+\cdots + p_j$, then $1\leq j \leq \ell \leq n-1$. 
Since 
$$
\bigcup_{n_1=1}^{n-\ell}\, \bigcup_{p_1=1}^{n_2} \cdots \bigcup_{p_j=1}^{n_{j+1}} \ 
C^1_{n_1} \times C^{p_1}_{n_2} \times \cdots \times C^{p_j}_{n_{j+1}} = C^{\ell+1}_n,
$$
because $n_1+n_2+\cdots + n_{j+1}=n$, if we rename the sequence 
$(n_1,q^1_1,...,q^1_{p_1},...,q^j_1,...,q^j_{p_j})$ as $\bfn\in C^{\ell+1}_n$, 
the sum over $j$ becomes
\begin{align*}
& \sum_{\ell=1}^{n-1} (-1)^\ell \sum_{\bfn \in \calC_n^{\ell+1}} 
\left( 
\sum_{j=1}^\ell \sum_{\bfp\in \calC_\ell^j} 
\binom{n_1+1}{j}\ d_{p_1}(n_2,...,n_{p_1+1}) \cdots 
d_{p_j}(n_{P_{j-1}+2},...,n_{\ell+1})\right)\  
x_{n_1} \ y_{n_2} \cdots y_{n_{\ell+1}}, 
\end{align*}
where $Pi=p_1+\cdots +p_i$ for $i=1,...,j$. 
Using the recurrence proved in Corollary \ref{corollary-R2}, 
we finally obtain 
\begin{align*}
(\Id\scoprod S)\,  \DFdBex(x_n) 
& = x_n - y_n - \sum_{\ell=1}^{n-1} (-1)^\ell \sum_{\bfn \in \calC_n^{\ell+1}} 
d_{\ell+1}(n_1,...,n_{\ell+1})\ y_{n_1}\, y_{n_2}\cdots y_{n_{\ell+1}} 
\\ 
& \hspace{2cm} 
+ \sum_{\ell=1}^{n-1} (-1)^\ell \sum_{\bfn \in \calC_n^{\ell+1}} 
d_{\ell+1}(n_1,...,n_{\ell+1})\ x_{n_1}\, y_{n_2}\cdots y_{n_{\ell+1}} 
\\
& = \delta_r(x_n). 
\end{align*}

iii) The first counterexample to the analogue identity 
$\delta_l = (S \scoprod \Id)\,  \DFdBex$ is on the generator $x_3$, 
for which we have 
\begin{align*}
(S \scoprod \Id)\,  \DFdBex(x_3) &= 
v_3 - (2x_1v_2 + 3x_2v_1) + 5x_1^2v_1 - x_1v_1y_1 ,
\end{align*}
where $v_n=y_n-x_n$, while 
\begin{align*}
\delta_l(x_3) &= 
v_3 - (2x_1v_2 + 3x_2v_1) + 5x_1^2v_1 - x_1y_1v_1 . 
\end{align*}
\end{proof}

This result allows us on one side to deduce some properties of the loop of
formal diffeomorphisms, and on the other side to compare the Fa\`a di Bruno
coloop bialgebra with the non-commutative Fa\`a di Bruno Hopf algebra. 

\begin{corollary}
\begin{enumerate}
\item
The proalgebraic loop $\Diff$ is not right alternative, nor power associative. 
\item
Nevertheless, $\Diff$ has two-sided inverses and, for a given an associative 
algebra $A$ and an element $a\in \Diff(A)$, the inverse 
$a^{-1} = a\backslash e = e \slash a$ is given by the usual Lagrange 
coefficients, namely 
\begin{align*}
\big(a^{-1}\big)_n 
& = - \sum_{\ell=0}^{n-1} (-1)^\ell \sum_{\bfn \in \calC_n^{\ell+1}} 
d_{\ell+1}(n_1,...,n_{\ell+1})\ a_{n_1}\, a_{n_2}\cdots a_{n_{\ell+1}} . 
\end{align*}
\item
The inversion allows us to construct the right division, that is, 
$a \slash b = a\circ (e\slash b)$ for any $a,b\in \Diff(A)$, 
but it does not allow us to construct the left division, because
$b\backslash a \neq (b\backslash e)\circ a$ if $a_n\neq 0$ and $b_m\neq 0$ 
for some $n,m\geq 1$.
\end{enumerate}
\end{corollary}

\begin{proof}
1. 
The loop $\Diff(A)$ is right alternative if and only if the coloop bialgebra 
$\HFdBex$ is right coalternative, that is $(\Id\scoprod\mu)\,K=0$, where 
$K=(\DFdBex\scoprod\Id)\,\DFdBex-(\Id\scoprod\DFdBex)\,\DFdBex$ is the 
coassociator. 

The first deviation from right alternativity appears on the generator $x_5$. 
If we temporarily denote by $x_n=x_n^{(1)}$, $y_n=x_n^{(2)}$ and $z_n=x_n^{(3)}$ 
the three copies of the generators in $H\scoprod H\scoprod H$, we get 
\begin{align*}
K(x_5) 
& = 6x_3(y_1z_1-z_1y_1) \\ 
& \quad + x_2\big[3(y_2z_1-z_2y_1)+3(y_1z_2-z_1y_2) 
  + (8y_1^2z_1-7y_1z_1y_1-z_1y_1^2) \\ 
& \qquad\qquad + (8y_1z_1^2-7z_1y_1z_1-z_1^2y_1)\big] \\ 
& \quad +x_1\big[(y_3z_1-z_3y_1)+(y_2z_2-z_2y_2)+(y_1z_3-z_1y_3) 
+3(y_2y_1z_1-y_2z_1y_1) \\ 
& \qquad +2(y_1y_2z_1-y_1z_2y_1)+2(y_1^2z_2-y_1z_1y_2)+2(y_1z_1z_2-z_1y_1z_2) \\ 
& \qquad +(3y_2z_1^2-2z_2y_1z_1)+(2y_1z_2z_1-3z_1y_2z_1) \\ 
& \qquad\qquad + (5y_1^2z_1^2+y_1z_1^3 - 4y_1z_1y_1z_1-y_1z_1^2y_1-z_1y_1z_1^2)\big]
\end{align*}
and therefore
$$
(\Id\scoprod\mu)\,K(x_5) = x_1(y_2y_1^2-y_1y_2y_1) \neq 0. 
$$

The generator $x_5$ corresponds to the power $\lambda^6$ of usual series
with substitution law, therefore the deviation from right alternativity
computed on the generator $x_5$ can be detected by comparing the values
$(a\circ b)\circ b$ and $a\circ(b\circ b)$ for the two series 
$$
a(\lambda)=\lambda+a_1\lambda^2 \qquad\mbox{and}\qquad 
b(\lambda)=\lambda+b_1\lambda^2+b_2\lambda^3 
$$
with $a_1=1$ and $b_2b_1^2\neq b_1b_2b_1$. 
For instance\footnote{
The authors warmly thank J. M. P\'erez-Izquierdo for comunicating 
this example.}, by taking the $2\times 2$ elementary matrices $b_1=E_{11}$ and 
$b_2=E_{21}$, for which $b_2b_1^2=b_2$ and $b_1b_2b_1=0$. 
\medskip 

The same computation shows that $\Diff(A)$ is not power associative, because 
$$
\mu\, (\Id\scoprod\mu)\,K(x_5) = x_1x_2x_1^2-x_1^2x_2x_1 \neq 0. 
$$
For a series $c(\lambda)=\lambda+c_1\lambda^2+c_2\lambda^3$, 
we then have $(c\circ c)\circ c \neq c\circ (c\circ c)$ if 
$c_1c_2c_1^2 \neq c_1^2c_2c_1$. For instance, this is verified for the two 
$2\times 2$ matrices 
$$
c_1 = \left(\begin{array}{cc} 1 & 1 \\ 0 & 1 \end{array}\right) 
\qquad\mbox{and}\qquad 
c_2= \left(\begin{array}{cc} 1 & 0 \\ 1 & 0 \end{array}\right) , 
$$
for which we have 
$$
c_1c_2c_1^2 = \left(\begin{array}{cc} 2 & 4 \\ 1 & 2 \end{array}\right) 
\qquad\mbox{and}\qquad 
c_1^2c_2c_1= \left(\begin{array}{cc} 3 & 3 \\ 1 & 1 \end{array}\right) . 
$$
\bigskip 

2. 
The left and right inverses of $a\in \Diff(A)$ can be found using 
respectively the left antipode $S_l$ and right antipode $S_r$ of $\HFdBex$, 
according to the standard rule
$$
(e \slash a)_n = a\big(S_l(x_n)\big) 
\qquad\mbox{and}\qquad 
(a\backslash e)_n = a\big(S_r(x_n)\big) . 
$$
By Proposition \ref{proposition-antipode} we have $S_r=S_l$, therefore 
$e \slash a=a\backslash e$. 
\bigskip 

3. 
The identity $a \slash b = a\circ (e\slash b)$ in the loop $\Diff(A)$ is 
equivalent to the identity $\delta_r = (\Id\scoprod S_r)\,  \DFdBex$ 
in $\HFdBex$, proved in Proposition \ref{proposition-antipode}. 
The analogue identity for the left division does not hold. 
\end{proof}



The commutative Fa\`a di Bruno Hopf algebra which represents the classical 
proalgebraic group $\Diff$, mentioned at the beginning of 
section~\ref{section-diffeomorphisms}, admits a non-commutative lift 
\cite{BFK}
\begin{align*}
\HFdBnc & = \F\langle x_n,\ n\geq 1\rangle, \qquad (x_0=1) 
\\
\DFdBnc(x_n) & = \sum_{m=0}^n x_m \otimes \sum_{(k)} x_{k_0}\cdots x_{k_m} ,
\end{align*}
where the sum is over the set of tuples $(k_0,k_1,k_2,...,k_m)$ 
of non-negative integers such that $k_0+k_1+k_2+\cdots +k_m=n-m$. 
Since $\Diff$ is not a group over associative algebras, the existence 
of this Hopf algebra is not {\em a priori} ensured by the extention 
of the functor $\Diff$ from $\Com_\F$ to $\As_\F$. 

\begin{corollary}
The image of the coloop bialgebra $\HFdBex$ under the canonical projection 
$\pi$ given in Def.~\ref{canonical-projection} is the non-commutative 
Fa\`a di Bruno Hopf algebra $\HFdBnc$, that is, 
$$
(\HFdBex)^\otimes = \HFdBnc. 
$$
\end{corollary}

\begin{proof}
Indeed, we have $(\HFdBex)^\otimes = \HFdBnc$ as an algebra, and eventhough 
$\DFdBex$ is not coassociative, the comultiplication $(\DFdBex)^\otimes$ coincides 
with $\DFdBnc$ and therefore it is coassociative with respect to the 
component-wise multiplication in $\HFdBnc \sotimes \HFdBnc$. 

The assertion is then proved because, by Prop.~\ref{proposition-antipode}, 
the antipode in $\HFdBex$ is unique and coincides with that in $\HFdBnc$ 
on generators.  
\end{proof}


\section{Appendix: Categorical proofs with tangles}
\label{appendix}

Tangle diagrams are an efficient tool to prove formal (categorical) properties. 
Tangles are drawings suitable to represent operations and co-operations in a 
monoidal category, cf.~\cite{Lyubashenko} \cite{Yetter}, 
and therefore can be used to encode the structure of coloops in a category 
$(\catC,\scoprod,I)$. In the context of non-associative algebras they have been 
used in \cite{PerezIzquierdoShestakov} to code deformations of the enveloping 
algebra of a Malcev algebra, seen as the infinitesimal structure of a Moufang 
loop.

Tangles are drawings to be read from the top to the bottom as concatenation 
of operations acting on objects related by the monoidal product, and not by 
a cartesian (or tensor) product.
Here is the list of the tangles needed to represent all the operations and
the co-operations in coloops, with their defining identities. 
\medskip 

\noindent
\parbox{18cm}{
{\bf Categorical maps} \\ \nobreak
\begin{tabular}{p{0.3cm}p{2.4cm}p{2cm}p{3.5cm}p{3.5cm}p{2cm}}
\\[.1cm]
%
%
$\tau$ & twist \qquad 
&
$\begin{tangle}
\HH \x
\end{tangle}$ 
& 
invertible
& 
$\begin{tangle}
\HH \x \\  
\HH \xx  
\end{tangle} 
\ = \ 
\begin{tangle}
\HH \id \step \id \\  
\HH \id \step \id
\end{tangle}$ 
&
Sect. \ref{subsection-coloops}
\\[.8cm]
%
%
$\mu$ & folding map \qquad 
&
$\begin{tangle}
\HH \cu
\end{tangle}$ 
& 
associative 
& 
$\begin{tangle}
\HH \hcu \step[0.5] \id \\  
\HH \step[0.5] \cu  
\end{tangle} 
\ = \ 
\begin{tangle}
\HH \id \step[0.5] \hcu \\  
\HH \hcu  
\end{tangle}$ 
&
Sect. \ref{subsection-coloops}
\\[.5cm]
& & & commutative 
& 
$\begin{tangle}
\HH \x \\  
\HH \cu  
\end{tangle} 
\ = \ 
\begin{tangle}
\HH \id \step \id \\  
\HH \cu  
\end{tangle}$ 
&
Sect. \ref{subsection-coloops}
\\[.5cm]
& & &  
unital 
& 
$\begin{tangle}
\HH \unit \step \id \\  
\HH \cu  
\end{tangle} 
\ = \ 
\begin{tangle}
\HH \id \step \unit \\  
\HH \cu  
\end{tangle} 
\ = \ 
\begin{tangle}
\HH \id \\
\HH \id 
\end{tangle}$ 
&
Sect. \ref{subsection-coloops}
\\[.8cm]
%
%
$u$ & unit & 
\quad 
$\begin{tangle}
\HH \unit
\end{tangle}$ 
& folding morphism 
& 
$\begin{tangle}
\HH \unit \step \unit \\  
\HH \cu  
\end{tangle} 
\ = \ 
\begin{tangle}
\HH \unit 
\end{tangle}$ 
&
Sect. \ref{subsection-coloops}
\\[.2cm]
\end{tabular}
}
\bigskip 

\noindent
\parbox{18cm}{
{\bf Coloop structure maps} \\ \nobreak
\begin{tabular}{p{0.3cm}p{2.4cm}p{2cm}p{3.5cm}p{3.5cm}p{2cm}}
\\[.1cm]
%
%
$\Delta$ & comultiplication
&
\quad
$\begin{tangle}
\HH \cd
\end{tangle}$ 
&
folding morphism 
& 
$\begin{tangle}
\cucd \\
\end{tangle}
\ = \
\begin{tangles}{lcr}
\HH \cd && \cd \\
\HH \id & \x & \id \\
\HH \cu && \cu
\end{tangles}$ 
&
Sect. \ref{subsection-coloops}
\\[.5cm]
& & & unital 
& 
$\begin{tangle}
\HH \step[.5] \unit \\  
\HH \cd 
\end{tangle} 
\ = \ 
\begin{tangle}
\HH \unit \step \unit 
\end{tangle}$ 
&
Sect. \ref{subsection-coloops}
\\[.5cm]
& & & counital 
& 
$\begin{tangle}
\HH \cd \\  
\HH \step \id \\  
\HH \counit 
\end{tangle} 
\ = \ 
\begin{tangle}
\HH \cd  \\  
\HH \id \\
\HH \step \counit 
\end{tangle} 
\ = \ 
\begin{tangle}
\HH \id \\
\HH \id 
\end{tangle}$ 
&
Eq. (\ref{counitary})
\\[.8cm]
%
%
$\varepsilon$ & counit & 
\quad 
$\begin{tangle}
\counit
\end{tangle}$ 
& folding morphism 
& 
$\begin{tangle}
\HH \cu \\
\step[.5] \counit \\  
\end{tangle} 
\ = \ 
\begin{tangle}
\counit \step \counit 
\end{tangle}$ 
&
Sect. \ref{subsection-coloops}
\\[.5cm]
& & &  
unital 
& 
$\begin{tangle}
\HH \unit \\  
\HH \counit 
\end{tangle} 
\ = \ 
\emptyset $
&
Sect. \ref{subsection-coloops}
\end{tabular}
}
%
%
\begin{tabular}{p{0.3cm}p{2.4cm}p{2cm}p{3.5cm}p{3.5cm}p{2cm}}
\\[.3cm]
%
%
$\delta_r$ & \parbox{2.6cm}{right codivision}
& 
\quad
$\begin{tangle}
\cd* \step[-1] \object{r} 
\end{tangle}$ 
& folding morphism 
& 
$\begin{tangle}
\HH \cu \\ 
\HH \cd* \\
\end{tangle}
\ = \
\begin{tangles}{lcr}
\HH \cd* && \cd* \\
\HH \id & \x & \id \\
\HH \cu && \cu
\end{tangles}$ 
&
Sect. \ref{subsection-coloops}
\\[.5cm]
& & & unital 
& 
$\begin{tangle}
\HH \step[.5] \unit \\  
\HH \cd* 
\end{tangle} 
\ = \ 
\begin{tangle}
\HH \unit \step \unit 
\end{tangle}$ 
&
Sect. \ref{subsection-coloops}
\\[.5cm]
& & & right cocancellation 
& 
$\begin{tangle}
\HH \step[.5] \cd \\  
\HH \cd* \step[.5] \hd \\  
\HH \id \step \cu
\end{tangle} 
\ = \ 
\begin{tangle}
\HH \step[.5] \cd* \\  
\HH \cd \step[.5] \hd \\  
\HH \id \step \cu
\end{tangle} 
\ = \ 
\begin{tangle}
\HH \id \\ 
\HH \id \step \unit
\end{tangle}$ 
&
Eq. (\ref{right-cocancellation})
\\[.8cm]
%
%
%
$\delta_l$ & left codivision 
& 
$\begin{tangle}
\cd* \step[-1] \object{l} 
\end{tangle}$ 
& folding morphism 
& 
$\begin{tangle}
\HH \cu \\ 
\HH \cd* \\
\end{tangle}
\ = \
\begin{tangles}{lcr}
\HH \cd* && \cd* \\
\HH \id & \x & \id \\
\HH \cu && \cu
\end{tangles}$ 
&
Sect. \ref{subsection-coloops}
\\[.5cm]
& & & unital 
& 
$\begin{tangle}
\HH \step[.5] \unit \\  
\HH \cd* 
\end{tangle} 
\ = \ 
\begin{tangle}
\HH \unit \step \unit 
\end{tangle}$ 
&
Sect. \ref{subsection-coloops}
\\[.5cm]
& & & left cocancellation 
& 
$\begin{tangle}
\HH \step[.5] \cd \\  
\HH \hdd \step[.5] \cd* \\  
\HH \cu \step \id
\end{tangle} 
\ = \ 
\begin{tangle}
\HH \step[.5] \cd* \\  
\HH \hdd \step[.5] \cd \\  
\HH \cu \step \id
\end{tangle} 
\ = \ 
\begin{tangle}
\HH \step \id \\ 
\HH \unit \step \id
\end{tangle}$ 
&
Eq. (\ref{left-cocancellation})
\end{tabular}
\bigskip 

\noindent
\parbox{18cm}{
{\bf Further coloop maps} \\ \nobreak
\begin{tabular}{p{0.3cm}p{2.4cm}p{2cm}p{3.5cm}p{3.5cm}p{2cm}}
\\[.1cm]
%
%
$S_r$ & right antipode 
& 
$\begin{tangle}
\HH \morph r 
\end{tangle} 
:= \ 
\begin{tangle}
\HH \object{{\ }^r} \\[-.2cm] 
\HH \cd* \\
\HH \step \id \\
\HH \counit 
\end{tangle}$ 
& folding morphism 
& 
$\begin{tangle}
\HH \step[.5] \cu \\ 
\morph r \\
\end{tangle}
\ = \
\begin{tangles}{lcr}
\HH \morph r  \morph r \\
\step \cu 
\end{tangles}$ 
&
Sect. \ref{subsection-coloops}
\\[.6cm]
& & & unital 
& 
$\begin{tangle}
\HH \step \unit \\  
\morph r
\end{tangle} 
\ = \ 
\begin{tangle}
\HH \unit 
\end{tangle}$ 
&
Sect. \ref{subsection-coloops}
\\[.8cm]
%
%
$S_l$ & left antipode 
& 
$\begin{tangle}
\HH \morph l 
\end{tangle} 
:= \ 
\begin{tangle}
\HH \object{{\ }^l} \\[-.3cm] 
\HH \cd* \\
\HH \id \\
\HH \step \counit 
\end{tangle}$ 
& folding morphism 
& 
$\begin{tangle}
\HH \step[.5] \cu \\ 
\HH \morph l \\
\end{tangle}
\ = \
\begin{tangles}{lcr}
\HH \morph l \morph l \\
\step \cu 
\end{tangles}$ 
&
Sect. \ref{subsection-coloops}
\\[.6cm]
& & & unital 
& 
$\begin{tangle}
\HH \step \unit \\  
\HH \morph l 
\end{tangle} 
\ = \ 
\begin{tangle}
\HH \unit 
\end{tangle}$ 
&
Sect. \ref{subsection-coloops}
\end{tabular}
}
\bigskip

\noindent
\parbox{18cm}{
{\bf Properties of coloops} \\ \nobreak
\begin{tabular}{p{3.5cm}p{6cm}p{3.5cm}p{2cm}}
\\[.1cm]
%
%
& 
$\mu\, \delta_r = \mu\, \delta_l = u\, \varepsilon$
&
$\begin{tangle}
\HH \object{{\ }^r} \\[-.2cm] 
\HH \cd* \\
\HH \id \step \id \\
\HH \cu 
\end{tangle} 
\ = \ 
\begin{tangle}
\HH \object{{\ }^l} \\[-.3cm]   
\HH \cd* \\ 
\HH \id \step \id \\
\HH \cu  
\end{tangle} 
\ = \ 
\begin{tangle}
\\[.5cm]
\HH \counit \\ 
\HH \unit 
\end{tangle}$ 
&
Eq. (\ref{multiplication-codivision})
\\[.8cm]
%
%
partial counitality
& 
$(\Id\scoprod \varepsilon)\, \delta_r 
= (\varepsilon\scoprod\Id)\, \delta_l = \Id$
& 
$\begin{tangle}
\HH \object{{\ }^r} \\[-.2cm] 
\HH \cd* \\
\HH \id  \\
\HH \step \counit 
\end{tangle} 
\ = \ 
\begin{tangle}
\HH \object{{\ }^l} \\[-.3cm]   
\HH \cd* \\ 
\HH \step \id  \\ 
\HH \counit  
\end{tangle} 
\ = \ 
\begin{tangle}
\HH \id \\ 
\HH \id
\end{tangle}$ 
&
Eq. (\ref{counit-codivision})
\\[.8cm]
%
%
\parbox{3.5cm}{left and right \\ 5-terms identities}
& 
$\mu\,(S_r\scoprod \Id)\,\D = \mu\,(\Id \scoprod S_l)\,\D = u\,\varepsilon$
& 
$\begin{tangle}
\HH \step \cd \\
{\hh \morph r} \id \\ 
\HH \step \cu 
\end{tangle} 
\ = \ 
\begin{tangle}
\HH \cd \\
\id {\hh \morph l} \\ 
\HH \cu 
\end{tangle} 
\ = \ 
\begin{tangle}
\\[.5cm]
\HH \counit \\ 
\HH \unit 
\end{tangle}$ 
&
Eq. (\ref{antipode-left-right-5terms})
\end{tabular}
}
\bigskip

\noindent
{\bf Proof of Eq. (\ref{multiplication-codivision})}: For $\delta_r$
\begin{align*}
\begin{tangle}
\\[-.3cm] 
\HH \object{{\ }^r} \\[-.2cm] 
\HH \cd* \\
\HH \id \step \id \\
\HH \cu 
\end{tangle} 
\ = \ 
\begin{tangle}
\\[.5cm]
\HH \step[.5] \cd* \\
\HH \cd \step[.5] \hd \\
\HH \step \cu \\ 
\HH \counit
\end{tangle} 
\ = \ 
\begin{tangle}
\\[.5cm]
\HH \id \\
\HH \id \step \unit \\ 
\HH \counit 
\end{tangle} 
\ = \ 
\begin{tangle}
\\[.5cm]
\HH \counit \\ 
\HH \unit 
\end{tangle} 
\end{align*}
and similarly for $\delta_l$. 
\bigskip 

\noindent
{\bf Proof of Eq. (\ref{counit-codivision})}: For $\delta_r$
\begin{align*}
\begin{tangle}
\HH \object{{\ }^r} \\[-.2cm] 
\HH \cd* \\
\HH \id  \\
\HH \step \counit 
\end{tangle} 
\ = \ 
\begin{tangle}
\HH \step[.5] \cd \\ 
\HH \cd* \step[.5] \hd \\
\HH \id \\
\HH \step \counit \step \counit 
\end{tangle} 
\ = \ 
\begin{tangle}
\HH \step[.5] \cd \\ 
\HH \cd* \step[.5] \hd \\
\HH \id \step \cu \\ 
\HH \step \step[.5] \id \\
\HH \step \step[.5] \counit 
\end{tangle} 
\ = \ 
\begin{tangle}
\\[.5cm]
\HH \id \\
\HH \id \step \unit \\ 
\HH \step \counit 
\end{tangle} 
\ = \ 
\begin{tangle}
\HH \id \\
\HH \id 
\end{tangle} 
\end{align*}
and similarly for $\delta_l$. 
\bigskip 

\noindent
{\bf Proof of Eq. (\ref{antipode-left-right-5terms})}: For $S_r$
\begin{align*}
\begin{tangle}
\HH \step \cd \\
{\hh \morph r} \id \\ 
\HH \step \cu 
\end{tangle} 
\ = \ 
\begin{tangle}
\HH \step[.5] \cd \\
\HH \cd* \step[.5] \hd \\
\HH \step \id \step \id \\
\HH \counit \step \cu \\
\end{tangle} 
\ = \ 
\begin{tangle}
\\[.5cm]
\HH \counit \\ 
\HH \unit 
\end{tangle} 
\end{align*}
and similarly for $S_l$. 
\bigskip 

\noindent
\parbox{18cm}{
{\bf Properties of cogroups} \\ \nobreak
\begin{tabular}{p{3.5cm}p{6cm}p{3.5cm}p{2cm}}
\\[.1cm]
%
%
coassociative
& 
$(\D\scoprod \Id)\,\D = (\Id\scoprod\D)\,\D$
&
$\begin{tangle}
\HH \step[0.5] \cd  \\
\HH \cd \step[0.5] \id 
\end{tangle} 
\ = \ 
\begin{tangle}
\HH \cd \\ 
\HH \id \step[0.5] \cd 
\end{tangle}$
&
Eq. (\ref{coassociative})
\\[.8cm]
%
%
unique antipode
&
$S_r=S_l=:S$ 
&
$\begin{tangle}
\HH \object{{\ }^r} \\[-.2cm] 
\HH \cd* \\
\HH \step \id \\
\HH \counit 
\end{tangle}
\ =\ 
\begin{tangle}
\HH \object{{\ }^l} \\[-.3cm] 
\HH \cd* \\
\HH \id \\
\HH \step \counit 
\end{tangle}
\ =: 
\begin{tangle}
\HH \morph S 
\end{tangle}$ 
&
Prop. \ref{coloop-cogroup}
\\[.8cm]
%
%
5-terms identity
& 
$\mu\,(S\scoprod \Id)\,\D = \mu\,(\Id\scoprod S)\,\D = u\,\varepsilon$
& 
$\begin{tangle}
\HH \step \cd \\
{\hh \morph S} \id \\ 
\HH \step \cu 
\end{tangle} 
\ = \ 
\begin{tangle}
\HH \cd \\
\id {\hh \morph S} \\ 
\HH \cu 
\end{tangle} 
\ = \ 
\begin{tangle}
\\[.5cm]
\HH \counit \\ 
\HH \unit 
\end{tangle}$ 
&
Eq. (\ref{antipode-5-terms})
\\[.8cm]
%
%
\parbox{3.5cm}{right coinverse \\ property}
& 
$\delta_r = (\Id\scoprod S)\,\D$
& 
$\begin{tangle}
\HH \step[.5] \id \step[.5] \object r \\[-.2cm] 
\HH \cd* \\ 
\HH \id \step \id
\end{tangle} 
\ = \ 
\begin{tangle}
\HH \cd \\
\id {\hh \morph S} \\ 
\end{tangle}$
&
Eq. (\ref{coinverse})
\\[.8cm]
%
%
\parbox{3.5cm}{left coinverse \\ property}
& 
$\delta_l = (S\scoprod \Id)\,\D$
& 
$\begin{tangle}
\HH \step[.5] \id \step[.5] \object l \\[-.2cm] 
\HH \cd* \\ 
\HH \id \step \id
\end{tangle} 
\ = \ 
\begin{tangle}
\HH \step \cd \\
{\hh \morph S} \id \\ 
\end{tangle}$ 
&
Eq. (\ref{coinverse})
\end{tabular}
}
\bigskip


\begin{lemma}
[cf. Prop. \ref{coloop-cogroup}]
If a coloop $H$ is coassociative, then it has a two-sided antipode satisfying 
the 5-terms identity (\ref{antipode-5-terms}) and the coinverse properties 
(\ref{coinverse}). Therefore it is a cogroup. 
\end{lemma}

\begin{proof}
Assume that $\D$ is coassociative. then we have $S_l=S_r$ because 
$$
\begin{tangle}
\HH \morph l 
\end{tangle} 
\ = \ 
\begin{tangle}
\cd  \\
\counit \step \morph l \\ 
\HH \unit \step[2] \id \\ 
\cu
\end{tangle} 
\ = \ 
\begin{tangle}
\step[1.5] \cd  \\
\HH \step \cd \step[1.5] \id \\ 
\morph r \id \step[.5] \morph l \\ 
\HH \step \cu \step[1.5] \id \\ 
\step[1.5] \cu
\end{tangle} 
\ = \
\begin{tangle}
\step \cd  \\
\HH \step \id \step[1.5] \cd \\ 
\morph r \step[.5] \id \morph l \\ 
\HH \step \id \step[1.5] \cu \\ 
\step \cu
\end{tangle} 
\ = \
\begin{tangle}
\step \cd  \\
\morph r \step \counit \\ 
\HH \step \id \step[2] \unit \\ 
\step \cu
\end{tangle} 
\ = \
\begin{tangle}
\HH \morph r 
\end{tangle} 
$$
and therefore $S:= S_l=S_r$ satisfies the 5-terms identity becuase of 
(\ref{antipode-left-right-5terms}). 

For the the coinverse properties, let us show that the operator 
$R:= (\Id\scoprod S) \D$ satisfies the right cocancellations 
(\ref{left-cocancellation}), and therefore it coincides with $\delta_r$. 
In fact, we have
$$
\begin{tangle}
\HH \step[.5] \cd \\  
\HH \step[-.5] \object{R} \step[.5] \cd* \step[.5] \hd \\  
\HH \id \step \cu
\end{tangle} 
\ = \ 
\begin{tangle}
\HH \step[.5] \cd \\  
\HH \cd \step[.5] \hd \\ 
\id \morph S \id \\  
\HH \id \step \cu
\end{tangle} 
\ = \ 
\begin{tangle}
\HH \step[.5] \cd \\  
\HH \hdd \step[.5] \cd \\ 
\id \morph S \id \\  
\HH \id \step \cu
\end{tangle} 
\ = \ 
\begin{tangle}
\HH \cd \\  
\id \step \counit \\ 
\HH \id \step \unit
\end{tangle} 
\ = \ 
\begin{tangle}
\HH \id \\ 
\HH \id \step \unit
\end{tangle}
$$
and
$$ 
\begin{tangle}
\HH \object{R} \step[.5] \cd* \\  
\HH \cd \step[.5] \hd \\  
\HH \id \step \cu
\end{tangle} 
\ = \ 
\begin{tangle}
\HH \step[.5] \cd \\  
\HH \cd \step[.5] \hd \\ 
\id \step \id \morph S \\  
\HH \id \step \cu
\end{tangle} 
\ = \ 
\begin{tangle}
\HH \step[.5] \cd \\  
\HH \hdd \step[.5] \cd \\ 
\id \step \id \morph S \\  
\HH \id \step \cu
\end{tangle} 
\ = \ 
\begin{tangle}
\HH \cd \\ 
\id \step \counit \\ 
\HH \id \step \unit
\end{tangle} 
\ = \ 
\begin{tangle}
\HH \id \\ 
\HH \id \step \unit
\end{tangle}
$$
Similarly, the operator $(S\scoprod \Id)\D$ satisfies the left-cocancellations 
(\ref{left-cocancellation}), and therefore it coincides with $\delta_l$. 
\end{proof}


\bibliographystyle{alpha}

\end{document}